\documentclass[11pt, reqno]{amsart}
\usepackage{array,booktabs,tabularx}
\usepackage{graphicx}
\usepackage{amssymb}
\usepackage{enumerate}
\usepackage{color}
\usepackage{hyperref}
\usepackage{esint}
\usepackage{mathtools}
\usepackage{dsfont}
\usepackage{tikz}
\usepackage{ esint }
\usepackage[show]{ed}
\usepackage{comment}
\usepackage[inline]{enumitem}
\usepackage{soul}
\usepackage{needspace}
\usepackage[normalem]{ulem}
\newcommand{\Tinj}{\tau_{_{\!\text{inj}H}}}
\newcommand{\TM}{T^*\!M}
\newcommand{\Ti}{T}

\newcommand{\hyp}{\Sigma_{_{\!H}}}
\newcommand{\hypq}{\Sigma_{_{\!H_q}}}
\newcommand{\Hs}[1]{\text{\raisebox{.2cm}{${_{\!H_{\scl}^{#1}(M)}}$}}}

\newcommand{\scl}{\operatorname{scl}}

\newcommand{\ad}{\operatorname{ad}}

\newtheorem{theorem}{Theorem}

\newtheorem{lemma}{Lemma}
\newtheorem{proposition}[lemma]{Proposition}

\theoremstyle{definition}
\newtheorem{definition}{Definition}
\newtheorem{corollary}[lemma]{Corollary}
\newtheorem{remark}{Remark}

\newcommand{\R}{{\mathbb R}}

\newcommand{\G}{{\mathcal G}}

\newcommand{\ep}{\varepsilon}

\newcommand{\Id}{\operatorname{Id}}
\newcommand{\mc}[1]{\mathcal{#1}}
\newcommand{\re}{\mathbb{R}}
\newcommand{\Span}{\operatorname{span}}

\newcommand{\WFh}{\operatorname{WF_h}}
\newcommand{\MSh}{\operatorname{MS_h}}

\newcommand{\red}[1]{{\color{purple}{#1}}}

\newcommand{\blue}[1]{{\color{blue}{#1}}}

\newcommand{\SNH}{S\!N^*\!H}

\newcommand{\SM}{S^*\!M}

\newcommand{\SigH}{\Sigma_{_{\!H}}}
\newcommand{\comp}{\operatorname{comp}}

\newcommand{\LambdaH}{\Lambda^\tau_{_{\!\hyp}}}
\newcommand{\LambdaHq}{\Lambda^\tau_{_{\!\hypq}}}

\newcommand{\LM}{{_{\!L^2(M)}}}

\newcommand{\class}{\delta}

\def\XXint#1#2#3{{\setbox0=\hbox{$#1{#2#3}{\int}$} \vcenter{\hbox{$#2#3$}}\kern-.5\wd0}}

\DeclareMathOperator{\vol}{vol}

\DeclareMathOperator{\supp}{supp}
\DeclareMathOperator{\inj}{inj}
\DeclareMathOperator{\conj}{{inj}}

\DeclareMathOperator{\A}{\mathcal A}

\DeclareMathOperator{\I}{\mathcal I}
\DeclareMathOperator{\J}{\mathcal J}

\newcommand{\e}{\varepsilon}

\numberwithin{equation}{section}
\numberwithin{lemma}{section}

\newcommand{\Lpexp}{n(\frac{1}{2}-\frac{1}{p})+\e}
\newcommand{\Om}{\Omega}
\newcommand{\T}{\mathcal{T}}
\newcommand{\sub}[1]{_{_{#1}}}

\newcommand{\Op}{{\widetilde{Op}}}

\title[Growth of high $L^p$ norms via geodesic beams]{Growth of high $L^p$ norms for eigenfunctions:\\ an application of geodesic beams}

\author{Yaiza Canzani}
\address{Department of Mathematics,University of North Carolina at Chapel Hill, Chapel Hill, NC, USA}
\email{canzani@email.unc.edu}

\author{Jeffrey Galkowski}
\address{Department of Mathematics, University College London, London, United Kingdom}
\email{j.galkowski@ucl.ac.uk }

\date{}

\begin{document}

%

\begin{abstract}
This work concerns $L^p$ norms of high energy Laplace eigenfunctions, $(-\Delta_g-\lambda^2)\phi_\lambda=0$, $\|\phi_\lambda\|_{L^2}=1$. In~\cite{So88}, Sogge gave optimal estimates on the growth of $\|\phi_\lambda\|_{L^p}$ for a general compact Riemannian manifold. The goal of this article is to 
give general dynamical conditions guaranteeing quantitative improvements in $L^p$ estimates for $p>p_c$, where $p_c$ is the critical exponent. We also apply the results of~\cite{CG19dyn} to obtain quantitative improvements in concrete geometric settings including all product manifolds.
These are  the first results improving estimates for the $L^p$ growth of eigenfunctions that only require dynamical assumptions. In contrast with previous  improvements, our assumptions are local in the sense that they depend only on the geodesics passing through a shrinking neighborhood of a given set in $M$. Moreover, the article gives a structure theorem for eigenfunctions which saturate the quantitatively improved $L^p$ bound. Modulo an error, the theorem describes these eigenfunctions as finite sums of quasimodes which, roughly, approximate zonal harmonics on the sphere scaled by $1/\sqrt{\log \lambda}$.

\end{abstract}
\vspace*{-1.2cm}
\maketitle


\section{Introduction}
Let $(M,g)$ be a smooth, compact, Riemannian manifold of dimension $n$ and consider normalized Laplace eigenfunctions: solutions to 
$$
(-\Delta_g-\lambda^2)\phi_\lambda=0,\qquad \|\phi_\lambda\|_{L^2(M)}=1.
$$
This article studies the growth of $L^p$ norms of the eigenfunctions, $\phi_\lambda$, as $\lambda\to \infty$. Since the work of Sogge~\cite{So88}, it has been known that there is a change of behavior in the growth of $L^p$ norms for eigenfunctions at the \emph{critical exponent} $p_c:=\tfrac{2(n+1)}{n-1}$. In particular, 
\begin{equation} 
\label{e:stdBounds}
\|\phi_\lambda\|_{L^p(M)}\leq C\lambda^{\delta(p)},\qquad \delta(p):=\begin{cases} \frac{n-1}{2}-\frac{n}{p}& p_c\leq p\\ \frac{n-1}{4}-\frac{n-1}{2p}& 2\leq p\leq p_c.
\end{cases}
\end{equation}
For $p\geq p_c$,~\eqref{e:stdBounds} is saturated by the zonal harmonics on {the round sphere} $S^n$. On the other hand, for $p\leq p_c$, these bounds are saturated by the highest weight spherical harmonics on $S^n$, also known as Gaussian beams.  In a very strong sense, {the authors showed in~\cite[page 4]{CG19a} that any eigenfunction saturating~\eqref{e:stdBounds} for $p>p_c$ behaves like a zonal harmonic}, while Blair--Sogge~\cite{BlSo15,BlSo17} showed that for $p<p_c$ such eigenfunctions behave like Gaussian beams. {In the $p\leq p_c$, Blair--Sogge have recently made substantial progress on improved $L^p$ estimates on manifolds with non-positive curvature~\cite{BlSo19,blair2018concerning,BlSo15b}}

This article concerns the behavior of $L^p$ norms for high $p$; that is, for $p>p_c$. {While there has been a great deal of work on $L^p$ norms of eigenfunctions~\cite{KTZ,HeRi,Ta19,Ta18a,SoggeTothZelditch,SoggeZelditch,SoZe16, TZ02,ToZe03} this article departs from the now standard approaches. We both adapt the geodesic beam methods developed by the authors in~\cite{GT,GDefect,GJEDP,CGT, CG17,GT18a,CG19dyn,CG19a} and develop a new second microlocal calculus used to understand the number of points at which $|u_\lambda|$ can be large. By doing this} we give general dynamical conditions guaranteeing quantitative improvements over~\eqref{e:stdBounds} for $p>p_c$. In order to work in compact subsets of phase space, we semiclassically rescale our problem. Let $h=\lambda^{-1}$ and, abusing notation slightly, write $\phi_\lambda=\phi_h$ so that 
$$
(-h^2\Delta_g-1)\phi_h=0,\qquad \|\phi_h\|_{L^2(M)}=1.
$$
We also work with the semiclassical Sobolev spaces {$H^s_{\text{scl}}(M)$, $s\in \re$, defined by the norm}
$$
\|u\|_{\Hs{s}}^2:=\langle (-h^2\Delta_g+1)^su,u\rangle_{\LM}.
$$

We start by stating a consequence of our main theorem.
Let ${\Xi}$ denote the collection of maximal unit speed geodesics for $(M,g)$. {For $m$ a positive integer, $r>0$,  $t\in \R$, and $x \in M$}  define
$$
{\Xi}_x^{m,r,t}:=\big\{\gamma\in \Xi: \gamma(0)=x,\,\exists\text{ at least }m\text{ conjugate points to } x \text{ in }\gamma(t-r,t+r)\big\},
$$
where we count conjugate points with multiplicity. Next, for  a set $V \subset M$ write
$$
\mc{C}_{_{\!V}}^{m,r,t}:=\bigcup_{x\in V}\{\gamma(t): \gamma\in \Xi_x^{m,r,t}\}.
$$
Note that if $r_t \to 0^+$ as $|t|\to \infty$, then saying  {$y \in \mc{C}_x^{n-1,r_t,t}$} for $t$ large indicates that $y$ behaves like  a point that is maximally conjugate to $x$. This is the case for every point $x$ on the sphere when $y$ is either equal to $x$ or its antipodal point.  The following result applies under the assumption that points are not maximally conjugate and obtains quantitative improvements. 

\begin{theorem}
\label{t:noConj}
Let $p>p_c$, $U \subset M$, and assume there exist $t_0>0$ and $a>0$  so  
$$
\inf_{x_1,x_2\in U}d\big(x_1, \mc{C}_{x_2}^{n-1,r_t,t}\big)\geq r_t,\qquad\text{ for } t\geq t_0,
$$
with $r_t=\frac{1}{a}e^{-at}.$ 
Then, there exist $C>0$ and $h_0>0$ so that for $0<h<h_0$ and $u \in {\mc{D}'}(M)$
$$
\|u\|_{L^p(U)}\leq Ch^{-\delta(p)}\left(\frac{\|u\|_{\LM}}{\sqrt{\log h^{-1}}}\;+\; \frac{\sqrt{\log h^{-1}}}{h}\big\|(-h^2\Delta_g-1)u\big\|_{\Hs{\frac{n-3}{2}-\frac{n}{p}}}\right).
$$
\end{theorem}
\noindent The  assumption in Theorem~\ref{t:noConj} {rules} out maximal conjugacy {of any two points} $x,y\in U$ uniformly up to time $\infty$, and we expect it to hold on a generic manifold $M$ with $U=M$.
Since Theorem~\ref{t:noConj} includes the case of manifolds without conjugate points, it generalizes the work of~\cite{HaTa15}, where it was shown that logarithmic improvements {in $L^p$ norms for} $p>p_c$ are possible on manifolds with non-positive curvature. One family of examples where the assumptions of Theorem~\ref{t:noConj} hold is that of product manifolds~\cite[Lemma 1.1]{CG19a} i.e. $(M_1\times M_2, g_1 \oplus g_2)$ where $(M_i,g_i)$ are non-trivial compact Riemannian manifolds. Note that this family of examples includes manifolds with large numbers of conjugate points e.g. $S^{2}\times {M}$. 

{The proof of Theorem~\ref{t:noConj} gives a great deal of information about eigenfunctions which may saturate $L^p$ bounds $(p>p_c)$.  Our next theorem describes the structure of such eigenfunctions. This theorem shows that an eigenfunction can saturate the \emph{logarithmically improved} $L^p$ norm near at most \emph{boundedly many} points. Moreover, modulo an error small in $L^p$,  near each of these points the eigenfunction can be decomposed as a sum of quasimodes which are similar to the highest weight spherical harmonics scaled by $h^{\frac{n-1}{4}}/\sqrt{\log h^{-1}}$ whose number is nearly proportional to $h^{\frac{1-n}{2}}$. In the theorem below the quasimodes are denoted by $v_j$ and, while similar to highest weight spherical harmonics (a.k.a Gaussian beams), they are not as tightly localized to a geodesic segment and do not have Gaussian profiles. We refer to these quasimodes as geodesic beams (see Remark~\ref{r:geodesicBeams}). 

\begin{theorem}\label{t:JeffsFavorite} Let $p>p_c$. There exist $c,C>0$ such that the following holds.
Suppose the same assumptions as Theorem~\ref{t:noConj}. Let $0<\delta_1<\delta_2<\frac{1}{2}$, $h^{\delta_2}\leq R(h)\leq h^{\delta_1}$, and  $\{x_\alpha\}_{\alpha \in \mc{I}(h)} \subset M$ be a maximal $R(h)$-separated set. 
Let $u \in \mc{D}'(M)$ with $\|(-h^2\Delta_g-1)u\|_{H_h^{\frac{n-3}{2}}}=o\big(\frac{h}{\log h^{-1}}\|u\|_{L^2}\big)$,
and for $\e>0$
$$
\mc{S}\sub{U}(h, \e,u):=\Big\{\alpha \in \mc{I}(h): \|u\|_{L^\infty (B(x_\alpha,R(h)))} \geq  \frac{\e h^{\frac{1-n}{2}}}{\sqrt{\log h^{-1}}}\|u\|_{\LM}, \;\; B(x_\alpha,R(h))\cap U\neq \emptyset\Big\}.
$$
Then, for
all $\e>0$ there are $N_\e>0$ and $h_0>0$ such that  $|\mc{S}\sub{U}(h,\e,u)|\leq N_\e$ for all $0<h\leq h_0$.

Moreover, there is collection of geodesic tubes $\{\T_j\}_{j \in \mc{L}(\e,u)}$ of radius $R(h)$ (see Definition~\ref{d: cover}), with indices satisfying  $\mc{L}(\e,u)=\cup_{i=1}^{C}\J_i$ and   $\T_k\cap \T_\ell=\emptyset$  for $k,\ell\in \J_i$   with $k\neq \ell$, such that 
$$
u=u_e+\frac{1}{\sqrt{\log h^{-1}}}\sum_{j \in \mc{L}(\e,u)}v_{j},
$$
where $v_j$ is microsupported in $\T_j$, $|\mc{L}(\e,u)|\leq C \e^{-2}R(h)^{1-n}$, and for all $p\leq q\leq \infty$,
\begin{gather*}
\|u_e\|_{L^q}\leq {\e h^{-\delta(q)}}({\log h^{-1}})^{-\frac{1}{2}}\|u\|_{L^2},\\
\|v_{j}\|_{L^2}\leq C\e^{-1}R(h)^{\frac{n-1}{2}}\|u\|_{L^2},\qquad \quad\|Pv_j\|_{L^2}\leq C\e^{-1}R(h)^{\frac{n-1}{2}}h\|u\|_{L^2}.
\end{gather*}
 Finally, with
$
\mc{L}(\e,u,\alpha):=\big\{j\in \mc{L}(\e,u):\, \pi(\T_j)\cap B(x_\alpha, 3R(h))\neq \emptyset\big\},
$
 for every $\alpha\in \mc{S}\sub{U}(h,\e,u)$,
$$
c \e^2 R(h)^{1-n}\leq |\mc{L}(\e,u,\alpha)|\leq CR(h)^{1-n},\qquad \quad\sum_{j\in \mc{L}(\e,u,\alpha)}\|v_j\|^2_{L^2} \geq c^2\e^2.$$

\end{theorem}


The decomposition of $u$ into geodesic beams $v_j$ is illustrated in Figure \ref{f:structure}. One covers $S^*M$ with a collection of tubes $\{\mc{T}_j\}$  of radius $R(h)$ that run along a geodesic. Each geodesic beam $v_j$ corresponds to microlocalizing $u$ to the tube $\mc{T}_j$. 

Let $u$ be a quasimode with $\|Pu\|=o(h/\log h^{-1})\|u\|$. Note that, by interpolation Theorem~\ref{t:noConj} implies that for each $p>p_c$ there is $N>0$ such that if $\|u\|_{L^p}\geq \e h^{-\delta(p)}/\sqrt{\log h^{-1}}\|u\|_{L^2}$, then $\|u\|_{L^\infty}\geq \e^{N}h^{\frac{1-n}{2}}\|u\|_{L^2}$.  
In particular, for $u$ to saturate the logarithmically improved $L^p$ bound, it follows that $\mc{S}\sub{M}(h,\e^N,u)$ is non-empty. Theorem~\ref{t:JeffsFavorite} then gives that $\mc{S}\sub{M}(h,\e^N,u)$ has a uniformly bounded number of points and at these points the quasimode $u$ needs to consist of at least $c \ep^{2N} R(h)^{1-n}$ geodesic beams whose combined $L^2$ mass is at least $c \e^{N}$. Since $\dim(S^*_{x_\alpha}M)=n-1,$ this implies that there is a positive measure set of directions through $x_\alpha$ among which $u$ is spreading its mass nearly uniformly.

The proofs of Theorems~\ref{t:noConj} and~\ref{t:JeffsFavorite} hinge on a much more general theorem which does not require global geometric assumptions on $(M,g)$ and, in particular, Theorem \ref{t:JeffsFavorite} holds without modification under the assumptions of Theorem~\ref{t:main bound} below. {(We actually prove Theorem \ref{t:JeffsFavorite} under the more general assumptions, see Section \ref{s:JeffsFavorite}).}   As far as the authors are aware,  Theorem~\ref{t:noConj}  is the first result giving quantitative estimates for the $L^p$ growth of eigenfunctions that \emph{only} requires dynamical assumptions. We emphasize that, in contrast with previous improvements on Sogge's $L^p$ estimates, the assumptions in Theorem~\ref{t:main bound} below are purely dynamical and, moreover, are local in the sense that they depend only on the geodesics passing through a shrinking neighborhood of a given set in $M$. Moreover, the techniques do not require long-time wave parametrices.
}

\begin{figure}
\centering
\includegraphics[width=16cm]{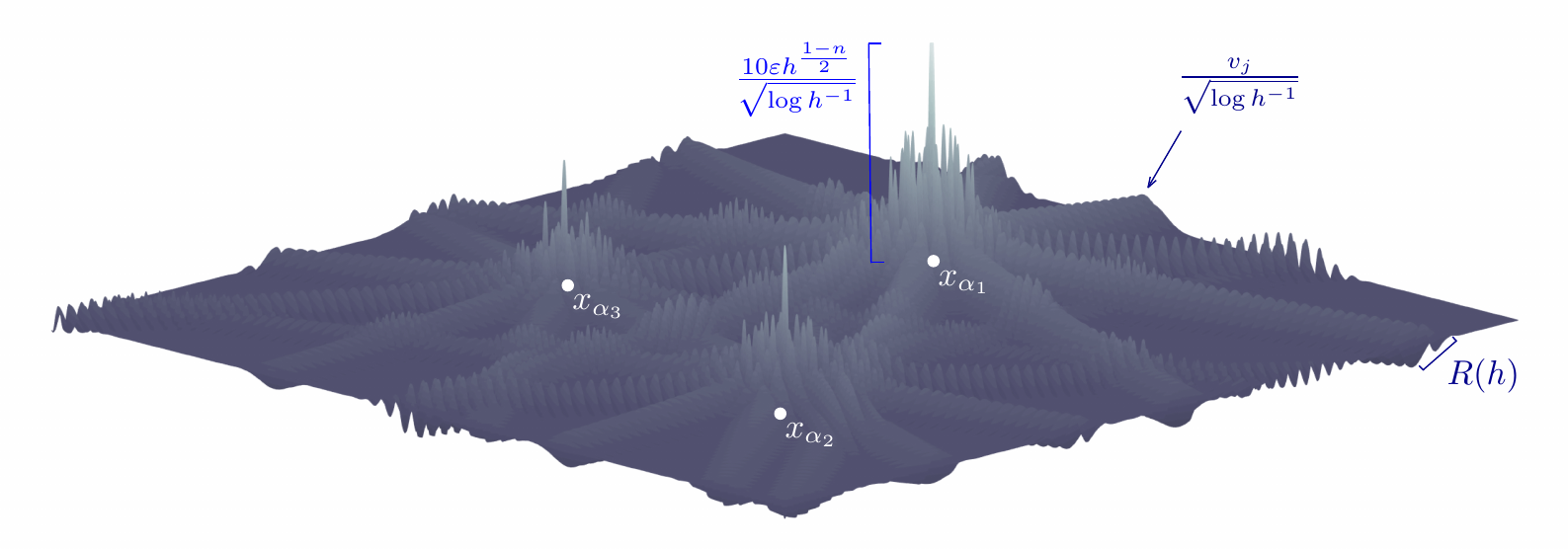}
\caption{\label{f:structure} The figure illustrates a function $u$ that saturates the $L^\infty$ bound at three points $x_{\alpha_1},x_{\alpha_2},x_{\alpha_3}$ viewed as a superposition of geodesic beams $v_j$. Each ridge corresponds to a beam $v_j$ and is microsupported on a tube $\mc{T}_j$ of radius $R(h)$. }
\end{figure}

Theorem~\ref{t:main bound} {below} controls $\|u\|_{L^p(U)}$ using an assumption on the maximal volume of long geodesics joining any two given points in $U$. For our proof, it is necessary to control the number points in $U$ where the $L^\infty$ norm of $u$ can be large. 
This is a very delicate and technical part of the argument, as the points in question may be approaching one another at rates $\sim h^\delta$ as $h\to 0^+$, with $0<\delta<\frac{1}{2}$.
We {overcome this problem} by developing a second microlocal calculus in Section~\ref{S:co-isotrop} which, after a delicate microlocal argument, yields an uncertainty type principle controlling the amount of $L^2$ mass shared along short geodesics connecting two nearby points. We expect that additional development of these counting techniques will have many other applications, e.g. to estimates on $L^p$ norms $p\leq p_c$.

To state our theorem, we need to introduce a few geometric objects. First, consider {the Hamiltonian function $p \in C^\infty(T^*M{\setminus\{0\}})$,}
\[p(x,\xi)=|\xi|_g-1,\]
 and let  
 $\varphi_t:T^*M\setminus 0 \to T^*M\setminus 0$
 denote the Hamiltonian flow for $p$ {at time $t$}. We also define the \emph{maximal expansion rate } and the \emph{Ehrenfest time} at frequency $h^{-1}$ respectively:
\begin{equation}
\label{e:LambdaMax}
\Lambda_{\max}:=\limsup_{|t|\to \infty}\frac{1}{|t|}{\log} \sup_{S^*M}\|d\varphi_t(x,\xi)\|, \qquad 
T_e(h):=\frac{\log h^{-1}}{2\Lambda_{\max}},
\end{equation}
where $\|\cdot\|$ denotes the norm in any metric on $T(T^*M)$.
Note that $\Lambda_{\max}\in[0,\infty)$, and if $\Lambda_{\max}=0$ we may replace it by an arbitrarily small positive constant. 
 We next describe a cover of $\SM$ by geodesic tubes. 

For each $\rho_0\in \SM$, {the co-sphere bundle to $M$}, let $H_{_{\!\rho_0}}\subset M$ be a hypersurface so that $\rho_0\in \SNH_{_{\!\rho_0}}$, {the unit conormal bundle to $H_{_{\!\rho_0}}$}. Then, let 
\[\mc{H}_{_{\!\rho_0}}\;\subset\; T_{H_{_{\!\rho_0}}}^*M=\{(x, \xi)\in T^*M:\; x \in H_{_{\!\rho_0}}  \}\] 
be a hypersurface containing $\SNH_{_{\!\rho_0}}$. {Next,} for $q\in \mc{H}_{_{\!\rho_0}}$, {$\tau>0$}, we define {the tube through $q$ of radius $R(h)>0$ and `length' $\tau+R(h)$ as}
\begin{equation}\label{e:tubes}
\Lambda_{q}^\tau(R(h)):=\bigcup_{|t|\leq \tau +R(h)} \varphi_t(B_{_{\mc{H}_{_{\!\rho_0}}}}\!({q},R(h))),
 B_{_{\mc{H}_{_{\!\rho_0}}}}\!({q}, R(h)):=\{\rho \in \mc{H}_{_{\!\rho_0}}:\;  d(\rho,{q})\leq R(h)\},
\end{equation}
{and $d$ is {distance induced by} the Sasaki metric on $\TM$ (See e.g.~\cite[Chapter 9]{BlairSasaki} for a description of the Sasaki metric). {Note that the tube runs along the geodesic through $q\in H_{\rho_0}$. Similarly, for $A \subset S^*M$, we define $\Lambda_{A}^\tau(R(h))$ in the same way, replacing $q$ with $A$ in \eqref{e:tubes}.} }

\begin{definition} \label{d: cover} Let $A\subset \SM$, $r>0$, and  $\{\rho_j(r)\}_{j=1}^{N_r} \subset A$ {for some $N_r>0$}. We say the collection of tubes {$\{\Lambda_{\rho_j}^\tau(r)\}_{j=1}^{N_r}$} is a \emph{$(\tau, r)$-cover} of a set $A\subset  \SM$ provided
 $$\Lambda_A^\tau(\tfrac{1}{2}r) \subset\bigcup_{j=1}^{N_r}\T_j,\qquad \T_j:=\Lambda_{\rho_j}^{\tau}(r).$$
\end{definition}

\noindent{Given a $(\tau,r)$ cover $\{\T_j\}_{j\in \J}$  for $\SM$, for each  $x\in M$ we define} 
$$\J_{\!x}:=\{j\in \J:\; \pi(\T_j)\cap B(x,r)\neq \emptyset\}.$$

{We are now ready to state Theorem \ref{t:main bound}}, where we give {\emph{explicit dynamical conditions}} guaranteeing quantitative improvements in $L^p$ norms. 

\begin{theorem}\label{t:main bound} There exists $\tau\sub{M}>0$ such that for all $p>p_c$  {and {$\e_0>0$}} the following holds. Let $U \subset M$,  \, ${0}<{\delta_1}<{\delta_2}<\frac{1}{2}$  {and let}  $h^{{\delta_2}}\leq R(h)\leq h^{{\delta_1}}$ {for all $h>0$}. Let $1\leq T(h)\leq (1-2{\delta_2})T_e(h)$ and let $t_0>0$ be $h$-independent. Let  $\{\T_j\}_{j\in \J}$ be a $(\tau, R(h))$ cover for $\SM$ {for some $0<\tau<\tau\sub{M}$}. 

Suppose that for any pair of points $x_1,x_2\in {U}$, the tubes over $x_1$ can be partitioned into {a disjoint union} $\J_{\! x_1}=\mc{B}\sub{x_1,x_2}\sqcup\mc{G}\sub{x_1,x_2}$ where 
$$
\bigcup_{j\in \mc{G}\sub{x_1,x_2}}\varphi_t(\T_j)\cap S^*_{B(x_2,R(h))}M=\emptyset,\qquad {t\in [t_0,T(h)].}$$
Then, there are $h_0>0$ and $C>0$ so that for all $u\in \mc{D}'(M)$, and $0<h<h_0$,
\begin{multline}
\label{e:LpestFinal}
\|u\|_{L^p({U})}\leq Ch^{-\delta(p)}\Bigg(\frac{\sqrt{t_0}}{\sqrt{T(h)}}+\Big[\sup_{x_1,x_2\in U}|\mc{B}\sub{x_1,x_2}|R(h)^{n-1}\Big]^{{\frac{1}{6+\e_0}(1-\frac{p_c}{p})}}\Bigg)\\{\times}\Bigg(\|u\|_{L^2}+\frac{T(h)}{h}\|(-h^2\Delta_g-1)u\|_{H_h^{\frac{n-3}{2}-\frac{n}{p}}}\Bigg).
\end{multline}
\end{theorem}

In order to interpret~\eqref{e:LpestFinal}, note that we think of the tubes $\mc{G}\sub{x_1,x_2}$ and $\mc{B}\sub{x_1,x_2}$ as respectively good (or non- looping) and bad (or looping) tubes. Then, observe that $|\mc{B}\sub{x_1,x_2}|R(h)^{n-1}\sim \vol\big(\bigcup_{j\in \mc{B}\sub{x_1,x_2}} \T_j\cap S^*_{x_1}M)$, and $\bigcup_{j\in \mc{B}\sub{x_1,x_2}}\T_j$ is the set of points over $x_1$ which may loop through $x_2$ in time $T(h)$. Therefore, if the volume of points in $S^*_{x_1}M$ looping through $x_2$ is bounded by { $T(h)^{{-(3+\e_0)(1-\frac{p_c}{p})^{-1}}}$},~\eqref{e:LpestFinal} provides {$T(h)^{-\frac{1}{2}}$} improvements over the standard $L^p$ bounds. We expect these {non-looping type} assumptions {to be} valid on generic manifolds. 

{Theorem~\ref{t:main bound} can be used to obtain improved $L^p$ resolvent bounds~\cite[Theorem 2.21]{Cu20} and, as shown their, are stable by certain rough perturbations. These estimates in turn can be used to construct complex geometric optics solutions and solve certain inverse problems~\cite{DSFKeCa13}.}

As in~\cite[Theorem 5 and Section 5]{CG19a}, the assumptions of Theorem~\ref{t:main bound} can be verified in certain integrable situations with $T(h)\gg \log h^{-1}$, thus producing $o((\log h^{-1})^{-\frac{1}{2}})$ improvements. Moreover, in~\cite{CG19dyn}, we used these types of good and bad tubes to understand averages and $L^\infty${-norms} under various assumptions on $M$, including that it has Anosov geodesic flow or non-positive curvature. Since our results do not require parametrices for the wave-group, we expect that the arguments leading to Theorem~\ref{t:main bound} will provide \emph{polynomial} improvements over Sogge's estimates on manifolds where Egorov type theorems hold for longer than logarithmic times. 

\begin{remark}
The proofs below adapt directly to  the case of quasimodes for real principal type semiclassical pseudodifferential operators of Laplace type. That is, to operators with principal symbol $p$  {satisfying both} $\partial_{\xi}p\neq 0$ on $\{p=0\}$ and {that} $\{p=0\}\cap T^*_xM$ has positive definite second fundamental form. {This is the case, for example, for Schr\"odinger operators away from the forbidden region.} However, for concreteness and simplicity of exposition, we have chosen to consider only the Laplace operator.
\end{remark}

\subsection{Discussion of the proof {of Theorem~\ref{t:main bound}}}

{Our method for proving Theorem~\ref{t:main bound} differs from the standard approaches for treating $L^p$ norms in two major ways. It}  hinges on adapting the geodesic beam techniques constructed by the authors~\cite{CG19a}, and on the development of a new second-microlocal calculus.

 We start in Section~\ref{s:tubes} by covering $S^*M$ with  tubes of radius $R(h)$. Then, in Section~\ref{s:estNear}, we decompose the function $u$, {whose $L^p$ norm we wish to study}, into geodesic beams i.e. into pieces microlocalized along each of these tubes. We then sort these {beams} into collections which carry $\sim 2^{-k}\|u\|_{L^2}$ mass and study the collections for each $k$ separately.  

In order to understand the $L^p$ norm of $u$, we next decompose the manifold into balls of radius $R(h)$. By constructing a good cover of $M$, we are able to think the $L^p$ norm of a function on $M$ as the $L^p$ norm of a function on a disjoint union of balls of radius $R(h)$. In each ball, $B$, we are able to apply the methods from~\cite{CG19a} to understand the $L^\infty$ norm of $u$ on $B$ in terms of the number of tubes with mass $\sim 2^{-k}\|u\|_{L^2}$ passing over that ball. 

To bound the $L^p$ norm with $p<\infty$, it then remains to understand {the number of balls on which} the function $u$ can have a certain $L^\infty$ norm. In Section~\ref{S:k_1, k_2} we first observe that when $u$ has relatively low $L^\infty$ norm on a ball, this ball {can be} neglected by interpolation with Sogge's $L^{p_c}$ estimate. It thus remains to understand {\emph{the number of balls} $B$ on which} the $L^\infty$ norm of $u$ can be large (i.e. close to extremal).  In fact, we will show that the number of balls such that $\|u\|_{L^\infty(B(x_\alpha,R(h)))}\geq C h^{\frac{1-n}{2}}/\sqrt{\log h^{-1}}$ is bounded \emph{uniformly} in $h$. That is, there is some number $N$ such that there are at most $N$ such balls for any value of $h>0$. This is the content of Theorem \ref{t:JeffsFavorite} and is proved in Section~\ref{s:JeffsFavorite}. It is in this step where a crucial new ingredient is input. 

The new method allows us to control the size of the set on which an eigenfunction (or quasimode) can have high $L^\infty$ norm. The method relies on understanding how much $L^2$ mass can be effectively shared along {short geodesics joining} two nearby points in such a way as to produce large $L^\infty$ norm at both points. That is, {if $x_\alpha$ and $x_\beta$ are nearby points on $M$, and} if $|u(x_\alpha)|$ and $|u(x_\beta)|$ are near extremal, how much total $L^2$ mass must the tubes over $x_\alpha$ and $x_\beta$ carry? 

In order to understand this sharing phenomenon, we develop a new second microlocal calculus associated to a Lagrangian foliation $L$ over a co-isotropic {submanifold $\Gamma \subset T^*M$}. This calculus allows for simultaneous localization along a leaf of $L$ and along $\Gamma$. The calculus, which is developed in Section~\ref{s:anisotropic}, can be seen as an interpolation between those in~\cite{DyZa} and ~\cite{SjZw:99}.  It is then the incompatibility between the calculi coming from two nearby points which allows us to control this sharing {of mass}. This incompatibility is demonstrated in Section~\ref{s:uncertainMe} {in the form of an  uncertainty principle type of estimate}.

Once the number of balls with high $L^\infty$ norm is understood, it remains to employ the non-looping techniques from~\cite{CG19a} where the $L^2$ mass on a collection of tubes is estimated using its non-looping time (see Section~\ref{s:loopMe}).

\subsection{Outline of the paper}
In section~\ref{s:tubes}, we construct the covers of $S^*M$ by tubes and $T^*M$ by balls which are necessary in the rest of the article. Section~\ref{s:mainThm} contains the proof of Theorems~\ref{t:JeffsFavorite} and~\ref{t:main bound}. This proof uses the anisotropic calculus developed in Section~\ref{s:anisotropic} and the almost orthogonality results from Section~\ref{s:uncertainMe}. Section~\ref{s:dynamical} contains the necessary dynamical arguments to prove Theorem~\ref{t:noConj} using Theorem~\ref{t:main bound}.

\noindent {\sc Acknowledgements.} 
The authors are grateful to the National Science Foundation for support under grants DMS-1900519 (Y.C) and DMS-1502661, DMS-1900434 (J.G.). {Y.C. is grateful to the Alfred P. Sloan Foundation. }  

\section{Tubes Lemmata}
\label{s:tubes}

{The next few lemmas are aimed at constructing $(\tau,r)$-good covers and partitions of various subsets of $\TM$ (see also~\cite[Section 3.2]{CG19a}).
\begin{definition}[good covers and partitions] \label{d:good cover} Let $A\subset \TM$, $r>0$, and  $\{\rho_j(r)\}_{j=1}^{N_r} \subset A$ {be a collection of points, for some $N_r>0$. Let $\mathfrak{D}$ be a positive integer}. We say that the collection of tubes $\{\Lambda_{\rho_j}^\tau(r)\}_{j=1}^{{N_r}}$ is a \emph{$( \mathfrak{D},\tau, r)$-good cover} of  $A\subset  \TM$ provided it is a $(\tau,r)$-cover of $A$ and there exists a partition $\{\mathcal{J}_\ell\}_{\ell=1}^{\mathfrak{D}}$ of $\{1, \dots, N_r\}$ so that for every $\ell\in \{1, \dots, \mathfrak{D}\}$
    \[
    \Lambda_{\rho_j}^\tau (3r)\cap \Lambda_{\rho_i}^\tau(3r)=\emptyset,\qquad i,j\in \mathcal{J}_\ell, \qquad  i\neq j.
    \]
In addition, {for $0\leq \delta\leq \frac{1}{2}$ and} $R(h)\geq {8}h^\delta$, we say that a collection $\{\chi_j\}_{j=1}^{N_h}\subset S_\delta(\TM;[0,1])$ is a \emph{$\delta$-good partition for $A$   
associated {to a} $(\mathfrak{D},\tau, R(h))$-good cover}
if  $\{\chi_j\}_{j=1}^{N_h}$ is bounded in $S_\delta$ and
$$
    \text{(1)} \supp \chi_j \subset {\Lambda_{\rho_j}^\tau(R(h))},\qquad\qquad
    \text{(2) $\sum_{j=1}^{N_h}\chi_j\geq 1 \;\text{on}\;  \Lambda_{A}^{{\tau/2}}(\tfrac{1}{2}R(h)).$}
$$
\end{definition}}
\begin{remark}
{We  show below that for any compact Riemannian manifold $M$, there are $\mathfrak{D}_{_{\!M}},R_0,\tau_0>0$, depending only on $(M,g)$, such that for $0<\tau<\tau_0$, $0<r<R_0$, there exists a $(\mathfrak{D}_{_{\!M}},\tau,r)$ good cover for $\SM$.}
\end{remark}

We start by constructing a useful cover of any Riemannian manifold with bounded curvature.
\begin{lemma}
\label{l:cover1}
Let $\tilde M$ be a compact Riemannian manifold. There exist $\mathfrak{D}_n>0$, depending only on $n$, and $R_0>0$ {depending only on $n$ and a lower bound for the sectional curvature of $\tilde M$},  so that the following holds. For $0<r<R_0$, there exist a finite collection of points $\{x_\alpha\}_{\alpha \in \I}\subset \tilde M$, $\I=\{1,\dots, N_r\}$, and a partition $\{\mathcal{I}_i\}_{i=1}^{\mathfrak{D}_n}$ of $\I$ so that 
\begin{equation*}
\begin{gathered}
\tilde M\subset \bigcup_{\alpha \in \I}B(x_{\alpha},r), \qquad\qquad
B(x_{\alpha_1},3r)\cap B(x_{\alpha_2},3{r})=\emptyset, \qquad \alpha_1,\alpha_2\in \mathcal I_i,\quad \alpha_1 \neq \alpha_2,\\ \text{$\{x_\alpha\}_{\alpha \in \I}$ is an $\frac{r}{2}$ maximal separated set in $\tilde M$.}
\end{gathered}
\end{equation*}
\end{lemma}

\begin{proof}
Let $\{x_\alpha\}_{\alpha \in \I}$ be a maximal $\frac{r}{2}$ separated set in $\tilde M$. Fix $\alpha_0 \in \I $ and suppose that $B(x_{\alpha_0},3{r})\cap B(x_\alpha,3{r})\neq \emptyset$ for all {$\alpha\in \mathcal K_{{{\alpha_0}}} \subset \I$}. Then for all ${\alpha \in \mathcal K_{{{\alpha_0}}}}$, $B(x_\alpha,\tfrac{{r}}{2})\subset B(x_{\alpha_0},8{{r}}).$ In particular,
$$
\sum_{{\alpha\in \mathcal K_{{{\alpha_0}}}}}\vol(B(x_\alpha,\tfrac{{{r}}}{2}))\leq \vol(B(x_{\alpha_0},8{{r}})). 
$$

Now, there {exist} ${R_0}>0$ depending on {$n$ and} a lower bound on the {sectional} curvature of $\tilde M$, {and $\mathfrak{D}_n>0$ depending only on $n$, so that {for all  $0<{r}<{R_0}$}},
\begin{equation}\label{e:asterix}
\vol(B(x_{\alpha_0},8{{r}}))\leq \vol(B(x_\alpha,14{{r}}))\leq \mathfrak{D}_{n}\vol(B(x_\alpha,\tfrac{{{r}}}{2})).
\end{equation}
Hence, it follows from \eqref{e:asterix} that
$$
\sum_{{\alpha\in \mathcal K_{{{\alpha_0}}}}}\vol(B(x_\alpha,\tfrac{{r}}{2}))\leq \vol(B(\rho_{\alpha_0},8{r}))\leq \frac{{\mathfrak{D}_n}}{{| \mathcal K_{{{\alpha_0}}}|}}\sum_{{\alpha\in \mathcal K_{{{\alpha_0}}}}}\vol(B(x_\alpha,\tfrac{{{r}}}{2})).
$$
In particular, $|\mathcal K_{{{\alpha_0}}}|\leq {\mathfrak{D}_n}$.

At this point we have proved that each of the balls $B(x_\alpha,3r)$ intersects at most ${\mathfrak{D}_n}-1$ other balls. We now construct the sets $\mathcal I_1,\dots, \mathcal I_{{\mathfrak{D}_n}}$ using a greedy algorithm. We will say that {the index $\alpha_1$ \emph{intersects} the index $\alpha_2$} if 
$$
B(x_{\alpha_1},3r)\cap B(x_{\alpha_2},3r)\neq \emptyset.
$$
We place the index $1\in \mathcal I_1$. Then suppose we have placed the indices  $\{1,\dots, \alpha\}$ in $\mathcal I_1,\dots, \mathcal I_{{\mathfrak{D}_n}}$ so each of the $\mathcal I_i$'s consists of disjoint indices. Then, since $\alpha+1$ intersects at most ${\mathfrak{D}_n}-1$ indices, it is disjoint from $\mathcal I_i$ for some $i$. We add the index $\alpha$ to $\mathcal I_i$. By induction we obtain the partition $\mathcal I_1,\dots ,\mathcal I_{{\mathfrak{D}_n}}$. 

Now, suppose that there exists $x \in \tilde M$ so that {$x \notin \bigcup_{\alpha \in \I} B(x_{\alpha},r)$}. Then, 
$
 \min_{\alpha \in \I} d(x,x_{\alpha})\geq r,
$
a contradiction of the $r/2$ maximality of $x_{\alpha}$.

\end{proof}

In order to construct our microlocal partition, we first fix a smooth hypersurface $H\subset M$, and choose Fermi normal coordinates $x=(x_1,x')$ in a neighborhood  of $H=\{x_1=0\}$. We write  $(\xi_1, \xi') \in T_x^*M$ for the dual coordinates.
Let 
\begin{equation}
\label{e:hyp}
\hyp:=\Big\{(x, \xi)\in S^*\sub{H}M \big |\;\, |\xi_1|\geq \tfrac{1}{2} \Big\} 
\end{equation}
We then consider
\begin{equation}
\label{e:hyp2}
\mathcal{H}\sub{\hyp}:=\{(x, \xi)\in {T^*\sub{H}M} \mid\;\;\; |\xi_1|\geq \tfrac{1}{2},\;\; \tfrac{1}{2}<|\xi|_{g(x)}<\tfrac{3}{2} \}. 
\end{equation}
Then, $\mc{H}\sub{\hyp}$ is transverse to the geodesic flow and there is $0<\Tinj<1$ so that the map
\begin{equation}
    \label{e:psi}
 \Psi:[-\Tinj,\Tinj]\times \mc{H}\sub{\hyp} \to  \TM, \qquad \qquad \Psi(t,\rho):=\varphi_t(\rho),
\end{equation}
is injective.
{Our next lemma shows that there is $\mathfrak{D}_n>0$ depending only on $n$ such that one can construct a $(\mathfrak{D}_n,\tau, r)$-good cover of $\SigH$.}

\begin{lemma}
\label{l:cover}
There exist $\mathfrak{D}_{n}>0$ depending only on $n$,  ${R_0=R_0(n,H)}>0$, such that for $0<{r_1}<{R_0}$, $0<r_0\leq \frac{{r_1}}{2} $, there exist {points} $\{\rho_j\}_{j=1}^{N_{r_1}}\subset \hyp$ {and a partition $\{\mathcal J_{i}\}_{i=1}^{\mathfrak{D}_{n}}$ of $\{1,\dots, N_{r_1}\}$ so that for all $0<\tau<\frac{\Tinj}{2}$ 
\begin{center}
\begin{itemize*}
    \item $\LambdaH(r_0)\subset \bigcup_{j=1}^{N_{r_1}}\Lambda_{_{\rho_j}}^\tau({r_1}),\qquad\qquad$
    \item $\Lambda_{_{\rho_j}}^\tau(3{r_1})\cap \Lambda_{_{\rho_\ell}}^\tau(3{r_1})=\emptyset, \qquad j,\ell\in \mathcal J_i,\quad j\neq \ell.$
\end{itemize*}
\end{center}
}
\end{lemma}

\begin{proof}
We first apply Lemma~\ref{l:cover1} to ${\tilde M=\hyp}$ to obtain $R_0>0$ depending only on $n$ {and the sectional curvature of $H$ {and that of $M$ near $H$},} so that for $0<r_1<R_0$, there exist $\{\rho_j\}_{j=1}^{N_{r_1}}\subset \hyp$ and a partition $\{\mathcal{J}_i\}_{i=1}^{\mathfrak{D}_n}$ of $\{1,\dots,N_{r_1}\}$ such that 
\begin{equation*}
\begin{gathered}
    \hyp\subset \bigcup_{j=1}^{{N_{r_1}}}B(\rho_j,{r_1}),\qquad\qquad 
B({\rho_j},{3r_1})\cap B(\rho_\ell,3{r_1})=\emptyset, \qquad j,\ell\in \mathcal J_i,\quad j\neq \ell,\\
 \text{$\{\rho_j\}_{j=1}^{{N_{r_1}}}$ is {an} ${\frac{r_1}{2}}$ maximal {separated set} in $\hyp$.}
    \end{gathered}
\end{equation*}

Now, {suppose that} $j,\ell\in \mc{J}_i$ and
$$
\Lambda_{_{\rho_\ell}}^\tau(3{r_1})\cap \Lambda_{_{\rho_{{j}}}}^\tau(3{r_1})\neq \emptyset.
$$
Then, there exist $q_\ell\in B(\rho_\ell,3{r_1})\cap\mc{H}_{\Sigma_{{H}}}$, $q_{{j}}\in B(\rho_{{j}},3{r_1})\cap\mc{H}_{\Sigma_{{H}}}$ and {$ t_\ell,{t_j}\in[-\tau, \tau]$}  so that 
$
\varphi_{_{\!t_\ell-t_{{j}}}}(q_\ell)=q_{{j}}.
$
{Here, $\mc{H}_{\Sigma}$ is the hypersurface defined in \eqref{e:hyp2}}.
In particular, {for $\tau<\Tinj/2$}, this implies that $q_\ell=q_{{j}}$, $t_\ell=t_{{j}}$ and hence $B(\rho_\ell,3{r_1})\cap B(\rho_{{j}},3{r_1})\neq \emptyset$ a contradiction.

Now, suppose $r_0\leq {r_1}$ and that there exists $\rho \in \LambdaH(r_0)$ so that {$\rho \notin \bigcup_{j=1, \dots, N_{r_1}} \Lambda_{_{\rho_j}}^\tau(r_1)$}. Then, there are $|t|<\tau+r_0$ and $q\in \mc{H}_{\Sigma_{{H}}}$ so that
$$
\rho=\varphi_t(q),\qquad d(q,\SigH)<r_0,\qquad \min_{{j=1, \dots, N_{r_1}}}d(q, \rho_j)\geq {r_1}.
$$
In particular, {there exists  $\tilde{\rho}\in \hyp$ with $d(q, \tilde \rho)<r_0$ such that for all $j=1, \dots, N_{r_1}$}, 
$$
d(\tilde{\rho},\rho_j)\geq d(q,\rho_j)-d(q,\tilde{\rho})>{r_1}-r_0.
$$
This contradicts the maximality of $\{\rho_j\}_{j=1}^{N_{r_1}}$ if $r_0\leq {r_1}/2$. 
\end{proof}

{We proceed to build a $\delta$-good partition of unity associated to the cover we constructed in Lemma \ref{l:cover}. {The key feature in this partition {is} that it is invariant under the {geodesic} flow. Indeed, the partition is built so that its quantization commutes with the operator ${P=-h^2\Delta-I}$ in a neighborhood of $\SigH$.}}
\begin{proposition}\label{l:nicePartition}
There exist $\tau_1=\tau_1(\Tinj)>0$ and  $\e_1=\e_1(\tau_1)>0$, and given $0<\delta<\tfrac{1}{2}$,{ $0<\e\leq \e_1$}, there exists $h_1>0$,
so that for any $0<\tau\leq \tau_1$, and $R(h)\geq 2h^\delta$, the following holds.

There exist $C_1>0$  so that  for all $0<h\leq h_1$ and every $(\tau,R(h))$-cover of $\SigH$
 there exists a partition of unity $\chi_j\in S_\class\cap  C^\infty_c(\TM ;[-C_1h^{1-2\delta},1+C_1h^{1-2\delta}])$ on $\LambdaH({\frac{1}{2}R(h)})$
for which
$$
\begin{gathered}
\supp \chi_j\subset \Lambda_{\rho_j}^{\tau+{\e}}(R(h)),\qquad
\MSh([P,Op_h(\chi_j)])\cap \LambdaH({\e})=\emptyset,\\
 {\sum_{j}{\chi_j}\equiv 1 \text{\;on \;}  {\Lambda^\tau\sub{\Sigma_{H}}(\tfrac{1}{2}R(h))},}
\end{gathered}
$$
and $\{\chi_j\}_j$ is bounded in $S_\delta$, and $[-h^2\Delta_g,Op_h(\chi_j)]$ is bounded in $\Psi_\delta$. 
\end{proposition}

{\begin{proof}
The proof is identical to that of~\cite[Proposition 3.4]{CG19a}.{ Although the claim that $\sum_{j}{\chi_j}\equiv 1$ on $ {\Lambda^\tau\sub{\Sigma_{H}}(\tfrac{1}{2}R(h))}$ does not appear its statement, it is included in its proof.}
\end{proof}}

\section{Proof of Theorem \ref{t:main bound}}
\label{s:mainThm}
For each $q \in \SM$, choose {a hypersurface} $H_q \subset M$ with $q \in S\!N^*\!H_q$ and $\tau_{_{\inj\! H_{q}}}>\frac{\inj(M)}{2}$, where {$\tau_{_{\inj\! H_{q}}}$ is defined in \eqref{e:psi} and} $\inj(M)$ is the injectivity radius of $M$. 
{We next use Lemma~\ref{l:cover} to generate a cover of $\Sigma\sub{H_q}$. Lemma~\ref{l:cover} yields the existence of  $\mathfrak{D}_n>0$ depending only on $n$ and  ${R_0=R_0(n,H_q)}>0$, 
such that the following holds. 
Since by assumption $R(h)\leq h^{\delta_{{1}}}$, there is $h_0>0$ such that $ h^{\delta_2}\leq R(h) \leq R_0$ for all $0<h<h_0$.}
Also, {set $r_1:=R(h)$ and $r_0:=\tfrac{1}{2}R(h)$}. Then, by Lemma~\ref{l:cover} there exist $N\sub{R(h)}\!=N\sub{R(h)}(q,R(h))>0$, $\{\rho_j\}_{j\in \J_{\! q}}\subset \hypq$ with ${\J_{\! q}}= \{1,\dots, N\sub{R(h)}\}$,  and a partition $\{\mathcal J_{q,i}\}_{i=1}^{{\mathfrak{D}_n}}$ of $\J_{\! q}$, so that for all $0<\tau<\frac{{\Tinj}_q}{2}$ 
\begin{align}
    &\bullet\; \LambdaHq({\tfrac{1}{2}R(h)})\subset \bigcup_{j\in \J_{\! q}}\Lambda_{_{\rho_j}}^\tau({R(h)}), \label{e:cover}\\
    &\bullet\; \bigcup_{i=1}^{{\mathfrak{D}_n}}{\mathcal J_{q,i}}=\J_{\! q}, \label{e:partition}\\
    &\bullet\; \Lambda_{_{\rho_{j_1}}}^\tau(3{R(h)})\cap \Lambda_{_{\rho_{j_2}}}^\tau(3{R(h)})=\emptyset, \qquad {j_1},{j_2}\in \mathcal J_{q,i},\quad {j_1}\neq {j_2}. \label{e:disjoint partition} 
\end{align}
By \eqref{e:cover} there is an $h$-independent open neighborhood of $q$, $V_q \subset \SM$,  covered by tubes as in Lemma~\ref{l:cover}. Since $\SM$ is compact, we may choose $\{q_\ell\}_{\ell=1}^L$ with $L$ independent of $h$, so that $\SM\subset \cup_{\ell=1}^L V_{q_\ell}$. 
In particular, 
if  $0<\tau \leq \min_{1\leq\ell\leq L} \tau_{_{\!H_{q_\ell}}}$, and for each $\ell\in \{1, \dots, L\}$ we let $$ \mathcal{T}_{q_\ell,j}= \Lambda_{_{\rho_j}}^\tau({R(h)}),$$
then there is $\mathfrak{D}_{_{\!M}}>0$ such that 
$\bigcup_{\ell=1}^L\{{\T_{q_\ell,j}}\}_{j\in \J_{q_\ell}}$ is a $(\mathfrak{D}_{_M},\tau,R(h))$-good cover for $\SM$. Let $\{{\psi_{q_\ell}}\}_{\ell=1}^L \subset C_c^\infty(\TM)$ satisfy
$$
\begin{gathered} \supp\psi_{q_\ell}\subset \{(x,\xi)\in \TM\setminus\!\{0\}\mid \; \big(x, \tfrac{\xi}{|\xi|_g}\big)\in V_{q_\ell}\} \qquad \forall \ell=1, \dots, L,\\
\sum_{\ell=1}^L\psi_{q_\ell}\equiv 1 \text{ in an $h$-independent neighborhood of }\SM.
\end{gathered}
$$
{We split the analysis of $u$ in two parts: near and away from the characteristic variety $\{p=0\}=S^*M$. In what follows we use $C$ to denote a positive constant that may change from line to line.}

\subsection{It suffices to study  $u$ near the characteristic variety}
\label{s:estAway}

In this section we reduce the study of $\|u\|_{L^p(U)}$ to an $h$-dependent neighborhood of the characteristic variety $\{p=0\}=S^*M$.
We will use repeatedly the following result.
\begin{lemma}\label{L:lp bound}
For all $\e>0$ and all $p\geq 2$, there exists $C>0$ such that 
\begin{equation}
\label{l:basicLp}
\|u\|_{L^p}\leq C h^{n(\frac{1}{p}-\frac{1}{2})}\|u\|_{H_h^{n(\frac{1}{2}-\frac{1}{p})+\e}}.
\end{equation}
\end{lemma}
\begin{proof}
By~\cite[Lemma 6.1]{GDefect} (or more precisely its proof), for any $\e>0$, there exists $C_\e\geq 1$ so that
$
\|\Id\|_{H_h^{\frac{n}{2}+\e}\to L^\infty}\leq C_\e h^{-\frac{n}{2}}.
$
By complex interpolation of $\Id:L^2\to L^2$ and $\Id:H_h^{\frac{n}{2}+\e}\to L^\infty$ with $\theta=\frac{2}{p}$ we obtain 
$
\|\Id\|_{H_h^{(\frac{n}{2}+\e)(1-\theta)}\to L^p}
\leq C_\e^{1-\theta}h^{-\frac{n}{2}(1-\theta)},
$
and this yields \eqref{l:basicLp}.
\end{proof}

 Observe that
$$
u=\sum_{\ell=1}^L Op_h(\psi_{q_\ell})u + \Big(1-\sum_{\ell=1}^L Op_h(\psi_{q_\ell})\Big)u.
$$
Note that since $1-\sum_{\ell=1}^L\psi_{q_\ell}{=} 0$  in  an $h$-independent neighborhood of $S^*M=\{p=0\}$,
by the standard elliptic parametrix construction (e.g.~\cite[Appendix E]{ZwScat}) there is $E\in \Psi^{-2}(M)$ with
\begin{equation}\label{e:S0 parametrix}
1-\sum_{\ell=1}^L Op_h(\psi_{q_\ell})=E{P}+O(h^\infty)_{\Psi^{-\infty}}.
\end{equation}

Next, combining \eqref{e:S0 parametrix} with Lemma \ref{L:lp bound}, and using that  $h^{n(\frac{1}{p}-\frac{1}{2})}=h^{-\delta(p)+\frac{1}{2}}  h^{-1}$, we have
\begin{align}\label{e:hugo}
\Big\|\Big(1-\sum_{\ell=1}^L Op_h(\psi_{q_\ell})\Big)u\Big\|_{L^p}
&\leq C h^{n(\frac{1}{p}-\frac{1}{2})}\|E{P}u\|_{H_h^{\Lpexp}}+O(h^\infty)\|u\|_{L^2} \notag\\
&\leq Ch^{-\delta(p)+\frac{1}{2}}  h^{-1}\|{P}u\|_{H_h^{\Lpexp-2}}+O(h^\infty)\|u\|_{L^2}.
\end{align}

It remains to understand the terms $Op_h(\psi_{q_\ell})u$. Since there are finitely many such terms, 
\begin{equation}\label{e:paco}
\Big\|\sum_{\ell=1}^L Op_h(\psi_{q_\ell})u\Big\|_{L^p}\leq \sum_{\ell=1}^L \|Op_h(\psi_{q_\ell})u\|_{L^p},
\end{equation}
and consider each term $\|Op_h(\psi_{q_\ell})u\|_{L^p}$ individually.

{By Proposition \ref{l:nicePartition} for each $\ell=1,\dots, L$ there exist $\tau_1(q_\ell)>0$, $\e_1(q_\ell)>0$, and a family of cut-offs $\{\tilde{\chi}\sub{\T_{q_\ell,j}}\}_{j\in \J_{q_\ell}}$, with  $\tilde\chi\sub{\T_{q_\ell,j}}$ supported in $\Lambda_{\rho_j}^{\tau+{\e_1(q_\ell)}}(R(h))$ and such that for all $0<\tau<\tau_1(q_\ell)$
\begin{equation}\label{E:ellipticity condition}
\sum_{j\in \J_{\! q_\ell}}\tilde{\chi}\sub{\T_{q_\ell,j}}{\equiv} 1\qquad \text{on} \quad {\Lambda^\tau\sub{\Sigma_{H_{q_\ell}}}(\tfrac{1}{2}R(h))}.
\end{equation}}
{Let $\tau_0(q_\ell)$ from~\cite[Theorem 8]{CG19a}. Then, set 
$$
\tau_{_{\!M}}:=\min_{1\leq \ell \leq L} \Big\{\tfrac{\inj(M)}{4},\, \tau_0(q_\ell),\,\tau_1(q_\ell),\, \tfrac{1}{2}\tau_{_{\inj\! H_{q_\ell}}}\Big\}.
$$
From now on we work with tubes $\T_{q_\ell,j}=\Lambda_{_{\rho_j}}^\tau({R(h)})$ for some $0<\tau<\tau_{_{\!M}}$.
Next, we localize $u$ near and away from $\Lambda^\tau\sub{\Sigma_{H_{q_\ell}}}(h^\delta)$:}
$$
Op_h(\psi_{q_{\ell}})u=\sum_{{j\in \J_{\! q_\ell}}}Op_h(\tilde \chi\sub{\T_{q_\ell,j}})Op_h(\psi_{q_{\ell}})u+\Big(1-\sum_{j\in \J_{\! q_\ell}}Op_h(\tilde  \chi\sub{\T_{q_\ell,j}})\Big)Op_h(\psi_{q_\ell})u.
$$
\begin{remark}
\label{r:geodesicBeams}
We refer to functions of the form $Op_h(\tilde{\chi}\sub{\T_{q_\ell},j})u$ as \emph{geodesic beams}. One can check using Proposition~\ref{l:nicePartition}, that if $u$ solves $Pu=O(h)$, then the geodesic beams also solve $Pu=O(h)$ and are localized to an $R(h)$ neighborhood of a length$\sim$1 segment of a geodesic.
\end{remark}
In particular, by \eqref{E:ellipticity condition}, {$\tfrac{1}{2}R(h)\geq {\frac{1}{2}h^{\delta_2}}$}, and \cite[Lemma 3.6]{CG19a}, there is $E\in h^{-{\delta_2}}\Psi_{{\delta_2}}^{\comp}$ so that 
\begin{equation}\label{e:Sdelta parametrix}
\Big(1-\sum_{j\in \J_{\! q_\ell}}Op_h(\tilde{\chi}\sub{\T_{q_\ell,j}})\Big)Op_h(\psi_{q_{\ell}})=E{P}+O_{\Psi^{-\infty}}(h^\infty).
\end{equation}
Since $h^{n(\frac{1}{p}-\frac{1}{2})-\delta_{{2}}}=h^{-\delta(p)+\frac{1}{2}-\delta_{{2}}}h^{-1}$, combining \eqref{e:Sdelta parametrix} with Lemma \ref{L:lp bound} yields
\begin{align}\label{e:luis}
\Big\|\Big(1-\sum_{j\in \J_{\! q_\ell}} Op_h(\tilde{\chi}\sub{\T_{q_\ell,j}})\Big)Op_h(\psi_{q_{\ell}})u\Big\|_{L^p}
\leq
Ch^{-\delta(p)-\frac{1}{2}-\delta_{{2}}}\|{P}u\|_{H_h^{\Lpexp-2}}+O(h^\infty)\|u\|_{L^2}.
\end{align}
{Combining \eqref{e:hugo}, \eqref{e:paco} and \eqref{e:luis} we have proved that for $U\subset M$
\begin{align}\label{e:rico}
\|u\|_{L^p(U)}
&\leq 
\sum_{\ell=1}^L\Big\| \sum_{j\in \J_{\! q_\ell}}Op_h(\tilde{\chi}\sub{\T_{q_\ell,j}})Op_h(\psi_{q_\ell})u\Big\|_{L^p(U)} \notag\\
&\qquad \qquad \qquad+Ch^{-\delta(p)+\frac{1}{2}-\delta_{{2}}}  h^{-1}\|{P}u\|_{H_h^{\Lpexp-2}}+O(h^\infty)\|u\|_{L^2}.
\end{align}}
\subsection{Filtering tubes by $L^2$-mass} \label{s:estNear}
By \eqref{e:rico} it only remains to control terms of the form $\|\sum_{j\in \J_{\!q_\ell}}Op_h(\tilde{\chi}\sub{\T_{q_\ell,j}})Op_h(\psi_{q_\ell})u\|_{L^p}$, where  $u$ is localized to $V_{q_\ell}$ within the characteristic variety $S^*M$ and, more importantly, to the tubes $\T_{q_\ell,j}$.
We fix $\ell$ and, abusing notation slightly, write 
\begin{equation}
\label{e:individualized}
\begin{gathered}
\psi:=\psi_{q_{\ell}},\qquad \J=\J_{\! q_\ell}, \qquad {\T_j=\T_{q_\ell,j}},\qquad  \tilde\chi\sub{\T_j}:=\tilde\chi\sub{\T_{q_\ell,j}},\\
v:=\sum_{j\in \J}Op_h(\tilde{\chi}\sub{\T_j})Op_h(\psi)u.
\end{gathered}
\end{equation}

 Let $T=T(h)\geq 1$.  For each $j\in \J$ let 
  \begin{equation}\label{e:chiNoTilde}
     \chi\sub{\T_j}\in C_c^\infty(T^*M;[0,1]){\cap S_\delta}
 \end{equation}
 be a smooth cut-off function  with $\supp \chi\sub{\T_j}\subset \T_j$, 
 $\chi\sub{\T_j}\equiv 1$ on $\supp \tilde{\chi}\sub{\T_j}$, {and such that $\{\chi_j\}_j$ is bounded in $S_\delta$}. We shall work with the modified norm
$$
\|u\|\sub{P,T}:=\|u\|_{L^2}+\tfrac{T}{h}\|Pu\|_{L^2}.
$$
Note that this norm is the natural norm for obtaining $T^{-\frac{1}{2}}$ improved estimates in $L^p$ bounds since the fact that $u$ is an $o(T^{-1}h)$ quasimode implies, roughly, that $u$ is an accurate solution to $(hD_t+P)u=0$ for times $t\leq T$.
For each integer $k\geq -1$ we consider the set 
\begin{equation}\label{E:A_k}
\A_k=\Big\{j\in \J: \;\; \;\frac{1}{2^{k+1}}\|u\|\sub{P,T}\; \leq \|Op_h(\chi\sub{\T_j})u\|_{L^2}+h^{-1}\|Op_h(\chi\sub{\T_j})Pu\|_{L^2}\leq\; \frac{1}{2^{k}}\|u\|\sub{P,T}\Big\}.
\end{equation}
{It follows that $\A_k$ consists of those tubes $\T_j$ with $L^2$ mass comparable to $2^{-k}$.}

Observe that since $|\chi\sub{\T_j}|\leq 1$, {for $h$ small enough depending on finitely many seminorms of $\chi_j$,}
$
\|Op_h(\chi\sub{\T_j})\|_{L^2\to L^2}\leq 2.
$
In particular, this together with $T\geq 1$, implies that 
\begin{equation}\label{E:set J}
\J=\bigcup_{k\geq -1}\A_k.
\end{equation}

\begin{lemma}\label{L:|A_k| bound}
There exists $C_n>0$ so that for all $k \geq -1$
\begin{equation}\label{e:|A_k| bound}
|\A_k| \leq C_n 2^{2k}.
\end{equation}
\end{lemma}

\begin{proof}
According to \eqref{e:partition}, the collection $\{\T_j\}_{j \in \J}$ can be partitioned into $\mathfrak{D}_n$ sets of disjoint tubes. {Thus,} we have $\sum_{j\in \J} |\chi\sub{\T_j}|^2\leq \mathfrak{D}_n$ and  there is $C_n>0$ depending only on $n$ such that
$$
\Big\|\sum_{{j\in \J}} Op_h(\chi\sub{\T_j})^*Op_h(\chi\sub{\T_j})\Big\|_{L^2\to L^2}\leq C_n.
$$
In particular,
{
$$
\sum_{{j\in \J}} \|Op_h(\chi\sub{\T_j})u\|^2_{L^2}\leq C_n\|u\|_{L^2}^2 
\qquad \text{and} \qquad
\sum_{{j\in \J}} \|Op_h(\chi\sub{\T_j})Pu\|^2_{L^2}\leq C_n\|Pu\|_{L^2}^2 . 
$$
}
Therefore, 
\begin{align*}
|\A_k|2^{-2k-2}\|u\|\sub{P,T}^2
&\leq 2\Big(\sum_{j\in \A_k} \|Op_h(\chi\sub{\T_j})u\|^2_{L^2} +h^{-2}\|Op_h(\chi\sub{\T_j})Pu\|_{L^2}^2\Big)
\leq C_n\|u\|\sub{P,T}^2.
\end{align*}
\vspace{-1cm}

\end{proof}

Next, let 
\begin{equation}\label{e:w_k}
w_k:=\sum_{j\in \A_k}Op_h(\tilde{\chi}\sub{\T_j}){Op_h(\psi)}u.
\end{equation}
Then, {by \eqref{e:individualized} and \eqref{E:set J} we have}
\begin{equation}\label{e:v}
v=\sum_{k={-1}}^\infty w_k.
\end{equation}
The goal is therefore to control $\|w_k\|_{L^p(U)}$ for each $k$ since the triangle inequality yields
 $$
 \|v\|_{L^p(U)}\leq \sum_{k=-1}^\infty \|w_k\|_{L^p(U)}.
 $$


\subsection{Filtering tubes by $L^\infty$ weight on shrinking balls} 
By Lemma~\ref{l:cover1}, there are points $\{x_{\alpha}\}_{\alpha\in \I}\subset M$ such that there exists a partition $\{\I_i\}_{i=1}^{\mathfrak{D}_n}$ of $\I$ so that 

\begin{itemize}
    \item $M\subset \bigcup_{\alpha \in \I}B(x_{\alpha},R(h)),$ \medskip
    \item $B(x_{\alpha_1},3R(h))\cap B(x_{\alpha_2},3R(h))=\emptyset, \qquad \alpha_1,\alpha_2\in \mathcal I_i,\quad \alpha_1\neq \alpha_2.$
\end{itemize}
Then, for $m \in \mathbb Z$ define 
\begin{equation}\label{E: I_km}
\I_{k,m}:=\Big\{\alpha \in \I\sub{U}: \;\;\; 2^{m-1}\leq h^{\frac{n-1}{2}}R(h)^{\frac{1-n}{2}}2^k\frac{ \|w_k\|_{L^\infty(B(x_\alpha,R(h)))}}{\|u\|\sub{P,T}}\leq 2^m\Big\},
\end{equation}
where ${\I\sub{U}:=\{\alpha \in \I:\; B(x_{\alpha},R(h))\cap U\neq \emptyset\}.}$
For each $k\in \mathbb{Z}_+$ and  $\alpha \in \I$ consider the sets
\[
{\A_{k}(\alpha)}:=\{j\in \A_k:\; {\pi\sub{M}}(\T_j)\cap B(x_{\alpha},{2}R(h))\neq \emptyset\},
\]
{where $\pi\sub{M}:T^*M \to M$ is the standard projection.}
\noindent The indices in $\A_k$ are those that correspond to tubes with mass comparable to $\tfrac{1}{2^{k}}\|u\|\sub{P,T}$, while indices in $\A_{k}(\alpha)$ correspond to tubes of mass $\tfrac{1}{2^{k}}\|u\|_{\sub{P,T}}$ that run over the ball $B(x_{\alpha}, {2}R(h))$. In particular, Lemma \ref{L:|A_k| bound} and ~\cite[Lemma 3.7]{CG19a} yield the existence of $C_n, c_{_{\!M}}>0$ such that 
\begin{equation}
\label{e:bound1}
c_{_{\!M}}2^m\leq |\A_{k}(\alpha)|\leq C_n 2^{2k},\qquad \alpha \in \I_{k,m}.
\end{equation}
{Indeed, for $\alpha \in \I_{k,m}$,
\begin{equation}\label{e:lb infty}
    2^{m-1}h^{\frac{1-n}{2}}R(h)^{\frac{1-n}{2}}2^{-k}\|u\|\sub{P,T}\leq \|w_k\|_{L^\infty(B(x_\alpha,R(h)))}.
\end{equation}
In addition, \eqref{E:A_k} and Lemma \cite[Lemma 3.7]{CG19a} imply that there exist $c\sub{M}>0$, $\tau\sub{M}>0$, and $C_n>0$, depending on $M$ and $n$ respectively, such that for all $N>0$ there exists $C\sub{N}>0$ with
\begin{align*}
&\|w_k\|_{L^\infty(B(x_\alpha,R(h)))}\\
&\leq \frac{C_n R(h)^{\frac{n-1}{2}}}{\tau\sub{M}^{1/2}h^{\frac{n-1}{2}}}\!\!\!\!\!\!\sum_{j\in \A_k(\alpha)}\!\!\!\!\|Op_h(\tilde{\chi}\sub{\T_j})Op_h(\psi) u\|_{L^2}+\!h^{-1}\|Op_h(\tilde{\chi}\sub{\T_j})POp_h(\psi) u\|_{L^2}\!+\!C\sub{N}h^N\|u\|\sub{P,T}\\
&\leq c_{_{\!M}}^{{-1}}h^{-\frac{n-1}{2}}R(h)^{\frac{n-1}{2}}2^{-k}\|u\|\sub{P,T}|\A_k(\alpha)|+\!C\sub{N}h^N\|u\|\sub{P,T},
\end{align*}
which, combined with \eqref{e:lb infty}, proves the lower bound in ~\eqref{e:bound1}.}
%
%
To simplify notation, let 
\begin{equation}\label{E: A_km}
\A_{k,m}:=\bigcup_{\alpha\in \I_{k,m}}\A_{k}(\alpha).
\end{equation}

{\noindent Note that for each $\alpha\in \I_{k,m}$ there is $\tilde{x}_\alpha\in B(x_\alpha,R(h))$ such that 
\begin{equation}\label{E:x_alphas}
|w_k(\tilde{x}_\alpha)|\geq 2^{m-1}h^{\frac{1-n}{2}}R(h)^{\frac{n-1}{2}}2^{-k}\|u\|\sub{P,T}.
\end{equation}

{We finish this section with a result that controls the size of $\I_{k,m}$ in terms of that of $\A_{k,m}$.}
Let  
\begin{equation}\label{E:rho}
 \tfrac{1}{2}(\delta_2+1)< \rho < 1,
\end{equation}
 {$0<\e<\delta$,  $\tilde{\chi}\in C_c^\infty((-1,1))$, and define the operator $\chi\sub{h,\tilde x_\alpha}$ by
\begin{equation*}
\chi\sub{h,\tilde x_\alpha}u(x):=\tilde{\chi}(\tfrac{1}{\e}h^{-\rho}d(x,\tilde x_\alpha))\;[Op_h(\tilde\chi(\tfrac{1}{\e}(|\xi|_g-1)))u](x).
\end{equation*}
 In Lemma \ref{l:chi_h,y} we prove that $\chi\sub{h,\tilde x_\alpha} \in \Psi_{\Gamma_{\tilde x_\alpha},{L_{\tilde x_\alpha}} ,\rho}^{-\infty}$,
where
\[\Om_{\tilde x_\alpha}=\{\xi \in T_{\tilde x_\alpha}^*M: \; |1-|\xi|_{g(\tilde x_\alpha)}|<\delta\},
\qquad
\Gamma_{\tilde x_\alpha}= \bigcup_{|t|<\tfrac{1}{2}\inj(M)} \varphi_t(\Om_{\tilde x_\alpha}),
\]
and 
 $\Psi_{\Gamma_{\tilde{x}_\alpha}, L_{\tilde{x}_\alpha} ,\rho}^{-\infty}$ is a class of smoothing pseudodifferential operators that allows for localization to $h^\rho$ neighborhoods of $\Gamma_{\tilde{x}_\alpha}$ and is compatible with localization to $h^\rho$ neighborhoods of the foliation $L_{\tilde{x}_\alpha}$ of $\Gamma_{\tilde{x}_\alpha}$ generated by $\Om_{\tilde{x}_\alpha}$.}

{In Theorem~\ref{l:nice2ndCut} for $\e>0$ we explain how to build a cut-off operator 
$
X_{\tilde x_\alpha}\in \Psi_{\Gamma_{\tilde{x}_\alpha}, L_{\tilde{x}_\alpha} ,\rho}^{-\infty}
$
such that 
\begin{equation}\label{e:cutoffconds}
\begin{cases}
\chi\sub{h,\tilde x_\alpha} X_{\tilde x_\alpha}=\chi\sub{h,\tilde x_\alpha} + O(h^\infty)_{\Psi^{-\infty}},\\ 
\WFh'([P, X_{\tilde x_\alpha}])  \cap \{(x,\xi): x \in B(\tilde x_\alpha, \tfrac{1}{2}{\conj M}), \;\xi \in \Om_x\} = \emptyset,        
\end{cases}
\end{equation}
where $\conj M$ denotes the injectivity radius of $M$.
}}

\begin{lemma}\label{l: |I_km|}
Let $\frac{1}{2}(\delta_2+1)< \rho \leq 1$. There exists $C>0$ so that for every $k \geq -1$ and $m \in \mathbb Z$ the following holds. If  
$$
|\A_{k,m}|\leq C \,2^{2m}R(h)^{n-1}\Big(h^{\rho-\frac{1}{2}}{R(h)^{-\frac{1}{2}}}\Big)^{-\frac{2n(n-1)}{3n+1}},
$$
then
\begin{equation}
\label{e:counting}
|\I_{k,m}|\leq C|\A_{k,m}|2^{-2m}R(h)^{1-n}.
\end{equation}
\end{lemma}
\begin{proof}

We claim that by \eqref{e:w_k}, for $\alpha \in \I_{k,m}$,
\begin{equation}\label{E:w_km}
\chi\sub{h,\tilde x_\alpha} w_k=\chi\sub{h,\tilde x_\alpha} w_{k,m}+O(h^\infty\|u\|\sub{L^2}),
\qquad \quad
w_{k,m}:= \sum_{j\in  \A_{k,m}}Op_h(\tilde{\chi}\sub{\T_j}){Op_h(\psi)}u.
\end{equation}
Indeed, it suffices to show that $\chi\sub{h,\tilde x_\alpha} Op_h(\tilde{\chi}\sub{\T_j}){Op_h(\psi)}u=O(h^\infty\|u\|_{L^2})$ for $\alpha \in \I_{k,m}$ and $j \notin \A_{k,m}$. Note that for such indices $\pi\sub{M}(\T_j)\cap B(\tilde x_\alpha, {2}R(h))=\emptyset$ while 
$$\supp {\tilde{\chi}(\tfrac{1}{\e}h^{-\rho}d(x,\tilde x_\alpha))} \subset B(\tilde x_\alpha, C \e h^\rho){\subset B(x_\alpha,\tfrac{3}{2}R(h))}$$
for some $C>0$ {and all $h$ small enough}.

Our next goal is to produce a lower bound for $|\A_{k,m}|$ in terms of $|\I_{k,m}|$ by using 
the lower bound \eqref{E:x_alphas} on $\|\chi\sub{h,\tilde x_\alpha} w_{k,m}\|_{L^\infty}$ for indices  $\alpha \in \I_{k,m}$.  
{By~\eqref{e:cutoffconds},} we have
$$
\chi\sub{h,\tilde x_\alpha} w_{k,m}=\chi\sub{h,\tilde x_\alpha}X_{\tilde x_\alpha} w_{k,m}+O(h^\infty)_{L^\infty},
$$
for $\alpha \in \I_{k,m}$.
In particular, by \eqref{E:x_alphas} and \eqref{E:w_km},
\begin{equation}
\label{e:IhaveAsquid}
 2^{m-1}h^{\frac{1-n}{2}}R(h)^{\frac{n-1}{2}}2^{-k}\|u\|\sub{P,T}\leq  \|\chi\sub{h,\tilde x_\alpha} w_{k}\|_{L^\infty} \leq  \| X_{\tilde x_\alpha} w_{k,m}\|_{L^\infty}{+O(h^\infty)\|u\|\sub{P,T})}. 
\end{equation}
Therefore, applying the standard $L^\infty$ bound for quasimodes of the Laplacian (see e.g.~\cite[Theorem 7.12]{EZB}) and using {that by \eqref{e:cutoffconds}} we have that  $X_{\tilde x_\alpha}$ nearly commutes with $P$ on {$B(\tilde x_\alpha, \tfrac{1}{2}{\conj M})$,} 
\begin{equation}\label{E:bound for one alpha}
\begin{aligned}
2^{m-1}R(h)^{\frac{n-1}{2}}2^{-k}\|u\|\sub{P,T}&\leq C(\|X_{\tilde x_\alpha}w_{k,m}\|_{L^2}+h^{-1}\| PX_{\tilde x_\alpha} w_{k,m}\|_{L^2(B)})+O(h^\infty\|u\|\sub{P,T}).\\
&\leq C(\|X_{\tilde x_\alpha}w_{k,m}\|_{L^2}+h^{-1}\|X_{\tilde x_\alpha} P w_{k,m}\|_{L^2})+O(h^\infty\|u\|\sub{P,T}).
\end{aligned}
\end{equation}
Note that we have canceled the factor $h^{\frac{1-n}{2}}$ which appears both in~\eqref{e:IhaveAsquid} and the standard $L^\infty$ bounds for quasimodes.
Using that $h^{2\rho-1}R(h)^{-1}=o(1)$, Proposition \ref{P:orthogonality} proves  that for all $\tilde \I  \subset \I_{k,m}$ {and $v\in L^2(M)$}
$$\sum_{\alpha\in\tilde I}\| X_{\tilde x_\alpha} v\|^2_{L^2}\leq C\Big(1+a_h|\tilde I|^{\frac{3n+1}{2n}}\Big)\|v\|_{L^2}^2,$$
where $a_h=(h^{\rho-\frac{1}{2}}R(h)^{-\frac{1}{2}})^{n-1}$.
As a consequence, \eqref{E:bound for one alpha} gives 
\begin{align*}
|\tilde \I|R(h)^{n-1}2^{-2k}2^{2(m-1)}\|u\|^2\sub{P,T}
& \leq C \Big(\sum_{\alpha \in \tilde I}\|X_{\tilde x_\alpha}  w_{k,m}\|_{L^2}^2+h^{-2}\sum_{\alpha \in \tilde I}\|X_{\tilde x_\alpha} P w_{k,m}\|_{L^2}^2\Big)\\
&\leq C\Big(1+a_h|\tilde \I|^{\frac{3n+1}{2n}}\Big)(\| w_{k,m}\|_{L^2}^2+h^{-2}\|P w_{k,m}\|_{L^2}^2)\\
&\leq C\Big(1+a_h|\tilde \I|^{\frac{3n+1}{2n}}\Big)2^{-2k}|\A_{k,m}|\|u\|^2\sub{P,T}.
\end{align*}
The last inequality follows from the definition of $w_{k,m}$ together with the definition \eqref{E:A_k} of $\A_k$. 

In particular,  we have proved that there is $C>0$ such that for all $\tilde \I  \subset \I_{k,m}$
\begin{equation}\label{E:max bound}
|\tilde \I|R(h)^{n-1}2^{2m}\leq C\max \Big(1\;,\; a_h|\tilde \I|^{\frac{3n+1}{2n}}\Big)|\A_{k,m}|.
\end{equation}
Suppose that $a_h|\I_{k,m}|^{\frac{3n+1}{2n}} \geq 1$. Then, there exists $\tilde \I  \subset \I_{k,m}$ such that  $a_h|\tilde \I|^{\frac{3n+1}{2n}}= 1$. In particular, 
$|\tilde \I|R(h)^{n-1}2^{2m}\leq C|\A_{k,m}|.$ This implies that if 
$|\A_{k,m}|\leq \tfrac{1}{C}  a_h^{-\frac{2n}{3n+1}} R(h)^{n-1}2^{2m},$
then $a_h|\I_{k,m}|^{\frac{3n+1}{2n}} \leq 1$ and so by \eqref{E:max bound}
\[|\I_{k,m}|R(h)^{n-1}2^{2m} \leq C |\A_{k,m}|.\]

\vspace{-.9cm}

\end{proof}


Note that for $w_{k,m}$ defined as in  \eqref{E:w_km}, 
\begin{equation}\label{E:using U_km}
    \|w_{k}\|_{L^p(U)}^p 
    \leq {\mathfrak{D}_n}\sum_{m=-\infty}^{\infty}\|w_{k}\|_{L^p(U_{k,m})}^p
    = {\mathfrak{D}_n}\sum_{m=-\infty}^{\infty}\|w_{k,m}\|_{L^p(U_{k,m})}^p+ O(h^\infty \|u\|\sub{P,T}),
\end{equation}
where 
\begin{equation}\label{E:U_km}
U_{k,m}:=\bigcup_{\alpha\in \I_{k,m}}B(x_\alpha,R(h)).
\end{equation}

Finally, we split the study of $\|w_{k}\|_{L^p(U)}$ into two regimes: tubes with low or high $L^\infty$ mass. {Fix $N>0$ large, to be determined later.} {(Indeed, we will see that it suffices to take $N{> \frac{1}{2}}(1-\frac{p_c}{p})^{-1}$.)} Then, {we claim that} for each $k \geq -1$,
\begin{equation}\label{e:splitMe}
\begin{aligned}
\|w_{k}\|_{L^p(U)}^p
&\leq {\mathfrak{D}_n}\sum_{m=-\infty}^{m_{1,k}}\|w_{k,m}\|_{L^p(U_{k,m})}^p+{\mathfrak{D}_n}\sum_{m=m_{1,k}+1}^{m_{2,k}}\|w_{k,m}\|^p_{L^p(U_{k,m})} + O(h^\infty \|u\|\sub{P,T}),
\end{aligned}
\end{equation}
where 
${m_{1,k}}$ and ${m_{2,k}}$ are defined by 
$$
2^{m_{1,k}}=\min\Bigg(\frac{2^kR(h)^{\frac{1-n}{2}}}{{\Ti}^N}\;,\;c_n 2^{2k}\;,\; c_0R(h)^{1-n}\Bigg),\quad 2^{m_{2,k}}=\min\Big(c_n 2^{2k}\;,\; c_0R(h)^{1-n}\Big),
$$
where $c_0, c_n$ are described in what follows. 
Indeed, note that the  bound \eqref{e:bound1} yields that $2^m$ is bounded by $|\A_k(\alpha)|$ for all $\alpha \in \I_{k,m}$ and the latter is controlled by $c_0 R(h)^{n-1}$ for some $c_0>0$, depending only on $(M,g)$. Also, note that by \eqref{e:bound1} the  $w_{k,m}$ are  only defined for $m$ satisfying $2^m \leq c_n 2^{2k}$. These observations justify that the second sum in \eqref{e:splitMe} runs only up to $m_{2,k}$.

\subsection{Control of the low $L^\infty$ mass term,  $m\leq m_{1,k}$}\label{S:k_1, k_2}
We first estimate the small $m$ term in \eqref{e:splitMe}. The estimates here essentially amount to interpolation between $L^{p_c}$ and $L^\infty$. From the definition \eqref{E: I_km} of $\I_{km}$, together with $\frac{1-n}{2}(p-p_c){-1}=-p\delta(p)$ and $\|w_{k,m}\|_{L^{p_c}(U_{k,m})} \leq {h^{-\frac{1}{p_c}}}\|u\|\sub{P,T}$,
\begin{align*}
\sum_{m=-\infty}^{m_{1,k}}\|w_{k,m}\|_{L^p(U_{k,m})}^p
&\leq C\sum_{m=-\infty}^{m_{1,k}} \|w_{k,m}\|_{L^\infty(U_{k,m})}^{p-p_c}\|w_{k,m}\|^{p_c}_{L^{p_c}(U_{k,m})}\notag\\
&  \leq C h^{-p\delta(p)}R(h)^{\frac{n-1}{2}(p-p_c)}2^{-k(p-p_c)}\sum_{m=-\infty}^{m_{1,k}}2^{m(p-p_c)}\|u\|^p\sub{P,T}\notag\\
& \leq Ch^{-p\delta(p)}R(h)^{\frac{n-1}{2}(p-p_c)}2^{(m_{1,k}-k)(p-p_c)}\|u\|^p\sub{P,T}. 
\end{align*}
It follows that 
 \begin{align}\label{E:pluto}
 \sum_{k\geq -1} \Big(\sum_{m=-\infty}^{m_{1,k}}\|{w_{k,m}}\|_{L^p(U_{k,m})}^p\Big)^{\frac{1}{p}}
 &\leq Ch^{-\delta(p)}R(h)^{\frac{n-1}{2}(1-\frac{p_c}{p})}\|u\|\sub{P,T} \sum_{k\geq -1} 2^{(m_{1,k}-k)(1-\frac{p_c}{p})}.
 \end{align}
Finally, define $k_1, k_2$ such that
 \begin{equation}\label{E:defn of k}
 2^{k_1}=\frac{R(h)^{\frac{1-n}{2}}}{{c_n}\Ti^N},\qquad 2^{k_2}={c_0}R(h)^{\frac{1-n}{2}}\Ti^N.
 \end{equation}
 If $k \leq k_1$, then $2^{m_{1,k}}=c_n2^{2k}$, so there exists {$C_{n,p}>0$} such that 
\[
\sum_{k=-1}^{k_1}2^{(m_{1,k}-k)(1-\frac{p_c}{p})} \leq {C_{n,p}}\frac{R(h)^{\frac{1-n}{2}(1-\frac{p_c}{p})}}{\Ti^{N(1-\frac{p_c}{p})}}.
\]
 If $ k_1\leq k \leq k_2$, then $2^{m_{1,k}}=\frac{2^kR(h)^{\frac{1-n}{2}}}{\Ti^N}$. Therefore, since ${|k_2-k_1|} \leq c N\log \Ti$ for some $c>0$,  there exists $C>0$ such that
 \[
 \sum_{k=k_1}^{k_2}2^{(m_{1,k}-k)(1-\frac{p_c}{p})}
 {\leq CN\log \Ti\frac{R(h)^{\frac{1-n}{2}(1-\frac{p_c}{p})}}{\Ti^{N(1-\frac{p_c}{p})}}}.
\]
Last, if $k \geq k_2$, then $2^{m_{1,k}}=c_0R(h)^{1-n}$, so there exists $C_p>0$ such that 
\[
\sum_{k=k_2}^{\infty}2^{(m_{1,k}-k)(1-\frac{p_c}{p})} \leq C_p\frac{R(h)^{\frac{1-n}{2}(1-\frac{p_c}{p})}}{\Ti^{N(1-\frac{p_c}{p})}}.
\]
Putting these three bounds together with \eqref{E:pluto}, we obtain
 \begin{equation}
 \label{e:squid1}
 \sum_{k\geq -1} \Big(\sum_{m=-\infty}^{m_{1,k}}\|{w_{k,m}}\|_{L^p(U_{k,m})}^p\Big)^{\frac{1}{p}}\leq Ch^{-\delta(p)}\frac{{N\log \Ti}}{\Ti^{N(1-\frac{p_c}{p})}}\|u\|\sub{P,T}.
 \end{equation}
 

\subsection{Control of the high $L^\infty$ mass term,  $m\geq m_{1,k}$}\label{s:highLinf}
In this section we estimate the large $m$ term in \eqref{e:splitMe}. To do this we split 
$$
\A_{k,m}=\mc{G}_{k,m}\sqcup \mc{B}_{k,m},
$$
where the set of `good' tubes $\bigcup_{j\in \mc{G}_{k,m}}\T_j$ is $[t_0,T]$ non-self looping and the number of `bad' tubes $|\mc{B}_{k,m}|$ is small.
To do this, let
\begin{equation} 
\label{e:indB}
{\mc{B}\sub{U}(\alpha,\beta)}:=\Bigg\{j\in \bigcup_k\A_k(\alpha)\,:\; \bigcup_{t=t_0}^T\varphi_t(\T_j)\cap S^*_{B(x_\beta,{2}R(h))}M\neq \emptyset\Bigg\}.
\end{equation}
Then, we define 
$$
\mc{B}_{k,m}:=\bigcup_{\alpha,\beta \in \mc{I}_{k,m}}\mc{B}\sub{U}(\alpha,\beta)\cap \mc{A}_{k}(\alpha).
$$
Let $\mc{G}_{k,m}:=\mc{A}_{k,m}\setminus\mc{B}_{k,m}$. Then, by construction, $\bigcup_{j\in \mc{G}_{k,m}}\T_j$ is $[t_0,T]$ non-self looping and we have
 \begin{equation}\label{E:bound on B_km}
 |\mc{B}_{k,m}|\leq c|\I_{k,m}|^2|\mc{B}\sub{U}|
 \end{equation}
for some $c>0$, where 
\begin{equation} 
\label{e:globB}
|\mc{B}\sub{U}|:=\sup\{ |\mc{B}\sub{U}(\alpha,\beta)|:\; \alpha, \beta \in \I\},
\end{equation}
That is, $|\mc{B}\sub{U}|$ is the maximum number of loops of length in $[t_0,T]$ joining any two points {in $U$}.

Then, define 
\begin{gather}\label{e:theGoodTheBadAndTheUgly}
w_{k,m}^{\mc{G}}:=\sum_{j\in \mc{G}_{k,m}}Op_h(\tilde{\chi}\sub{\T_j})Op_h(\psi)u,\qquad w_{k,m}^{\mc{B}}:=\sum_{j\in \mc{B}_{k,m}}Op_h(\tilde{\chi}\sub{\T_j})Op_h(\psi)u.
\end{gather}
Next, consider
\begin{align}\label{e:lps}
    \Big(\sum_{m=m_{1,k}}^{m_{2,k}}\|w_{k,m}\|_{L^p(U_{k,m})}^p\Big)^{\frac{1}{p}}&\leq \Big( \sum_{m=m_{1,k}}^{m_{2,k}}\|w_{k,m}^{\mc{G}}\|_{L^p(U_{k,m})}^p\Big)^{\frac{1}{p}}+\Big(\sum_{m=m_{1,k}}^{m_{2,k}}\|w_{k,m}^{\mc{B}}\|_{L^p(U_{k,m})}^p\Big)^{\frac{1}{p}}.
\end{align}


\subsubsection{Bound on the looping piece.}
We start by estimating the `bad' piece
$$\sum_{k\geq -1}\Big(\sum_{m=m_{1,k}}^{m_{2,k}}\|w_{k,m}^{\mc{B}}\|_{L^p(U_{k,m})}^p\Big)^{\frac{1}{p}}.$$
Observe that if 
$
2^{m_{1,k}}=\min(c_0R(h)^{1-n}, c_n2^{2k}), 
$
then $m_{1,k}=m_{2,k}$ and we need not consider this part of the sum. Therefore, the high $L^\infty$ mass term has
\begin{equation}
\label{e:defM1}
2^{m_{1,k}}=\frac{2^kR(h)^{\frac{1-n}{2}}}{\Ti^N}
\end{equation}
and $k_1\leq k\leq k_2$. 
Hence, for $m_{1,k}<m\leq m_{2,k}$,  Lemma \ref{L:|A_k| bound} gives that there is ${C_n}>0$ with
$$
|\A_{k,m}| 
\leq {C_n} 2^{2k} 
\leq {C_n}  R(h)^{n-1}2^{2m} \Ti^{2N}.
$$
Furthermore, since $R(h)\geq h^{\delta_2}$ with $\delta_2 <\tfrac{1}{2}$, \eqref{E:rho} yields that there is {$\e=\e(n, N)>0$} such that $h^{\rho-\frac{1}{2}}R(h)^{-\frac{1}{2}}<h^\e$, and hence, {since $\Ti=O(\log h^{-1})$},
\[
|\A_{k,m}|{=o\Big(  R(h)^{n-1}2^{2m} \left(h^{\rho-\frac{1}{2}}R(h)^{-\frac{1}{2}}\right)^{-\frac{2n(n-1)}{3n+1}}\Big)}.
\]
In particular,  a consequence of Lemma \ref{l: |I_km|} is the existence of {$h_0>0$} and $C>0$ such that 
\begin{align}
\label{e:boundCount}
|\I_{k,m}|&\leq C R(h)^{1-n}2^{-2m}|\A_{k,m}|\\
&\leq CR(h)^{1-n}2^{2k-2m},&\label{e:boundCount2}
\end{align}
for all $0<h \leq h_0$,
where we have used again  Lemma \ref{L:|A_k| bound} to bound $|\A_{k,m}|$.


Next, note that for  each point in $\I_{k,m}$ there are at most  $c|\I_{k,m}||\mc{B}\sub{U}|$ tubes  in $\mc{B}_{k,m}$  touching it. {Therefore,} we may apply~\cite[Lemma 3.7]{CG19a} to obtain $C>0$ such that 
\begin{equation} \label{e:LinfBad}
\|w_{k,m}^\mc{B}\|_{L^\infty({U_{k,m}})}\leq C h^{\frac{1-n}{2}}R(h)^{\frac{n-1}{2}}|\I_{k,m}||\mc{B}\sub{U}|2^{-k}\|u\|\sub{P,T}.
\end{equation}
Using \eqref{e:LinfBad}  and interpolating between $L^\infty$ and $L^{p_c}$ we obtain
\begin{equation}\label{E: Bad piece L^p norm}
\|w_{k,m}^{\mc{B}}\|_{L^p(U_{k,m})}^p \leq C h^{-p\delta(p)}\left(R(h)^{\frac{n-1}{2}}|\I_{k,m}||\mc{B}\sub{U}|2^{-k}\|u\|\sub{P,T}\right)^{p-p_c}\|w_{k,m}^{\mc{B}}\|_{L^2(U_{k,m})}^{p_c}.
\end{equation}
In addition, since combining \eqref{E:A_k} with  \eqref{E:bound on B_km} yields
\[
\|w_{k,m}^{\mc{B}}\|_{L^2(U_{k,m})} \leq C |\mc{B}_{k,m}|^{\frac{1}{2}}2^{-k}\|u\|\sub{P,T}  \leq C 2^{-k}|\mc{I}_{k,m}||\mc{B}\sub{U}|^{\frac{1}{2}}\|u\|\sub{P,T},
\]
the bounds in \eqref{E: Bad piece L^p norm} and \eqref{e:boundCount2}, together with the definition of $m_{1,k}$~\eqref{e:defM1} yield
\begin{align*}
\sum_{m=m_{1,k}}^{m_{2,k}}\|w_{k,m}^{\mc{B}}\|_{L^p(U_{k,m})}^p
&\leq C h^{-p\delta(p)}R(h)^{\frac{n-1}{2}(p-p_c)}
\sum_{m=m_{1,k}}^{m_{2,k}}|\I_{k,m}|^{p}|\mc{B}\sub{U}|^{p-\frac{p_c}{2}}2^{-kp}\|u\|^{p}\sub{P,T}\\
&\leq C h^{-p\delta(p)}R(h)^{\frac{n-1}{2}(-p-p_c)}
2^{kp}|\mc{B}\sub{U}|^{p-\frac{p_c}{2}}\|u\|^{p}\sub{P,T}\sum_{m=m_{1,k}}^{m_{2,k}}2^{-2mp}\\
&\leq C h^{-p\delta(p)}R(h)^{\frac{n-1}{2}(p-p_c)}|\mc{B}\sub{U}|^{p-\frac{p_c}{2}}\Ti^{2Np}\,2^{-kp}\|u\|^{p}\sub{P,T}.
\end{align*}
Then, with $k_1, k_2$ defined as in \eqref{E:defn of k}, we have that 
\begin{align*}
\sum_{k=k_1}^{k_2}\Big(\sum_{m=m_{1,k}}^{m_{2,k}}\|w_{k,m}^{\mc{B}}\|_{L^p(U_{k,m})}^p\Big)^{\frac{1}{p}}&\leq C h^{-\delta(p)}R(h)^{\frac{n-1}{2}(1-\frac{p_c}{p})}|\mc{B}\sub{U}|^{1-\frac{p_c}{2p}}\Ti^{2N}\|u\|\sub{P,T}\sum_{k=k_1}^{k_2}2^{-k}\notag\\
&\qquad \leq C h^{-\delta(p)}(R(h)^{n-1}|\mc{B}\sub{U}|)^{1-\frac{p_c}{2p}}\Ti^{3N}\|u\|\sub{P,T}.
\end{align*}
Finally, since we only need to consider $k_1\leq k\leq k_2$,
\begin{equation}
\sum_{k\geq -1}\Big(\sum_{m=m_{1,k}}^{m_{2,k}}\|w_{k,m}^{\mc{B}}\|_{L^p(U_{k,m})}^p\Big)^{\frac{1}{p}}
\leq C h^{-\delta(p)}(R(h)^{n-1}|\mc{B}\sub{U}|)^{1-\frac{p_c}{2p}}\Ti^{3N}\|u\|\sub{P,T}.\label{e:squid2}
\end{equation}
\subsubsection{Bound on the non self-looping piece.}\label{s:loopMe}In this section we aim to control the `good' piece
\begin{equation}\label{E:non-looping part}
    \sum_{k\geq -1}\Big(\sum_{m=m_{1,k}}^{m_{2,k}}\|w_{k,m}^{\mc{G}}\|_{L^p(U_{k,m})}^p\Big)^{\frac{1}{p}}.
\end{equation}
So far all $L^p$ bounds appearing have been $\ll h^{\frac{1-n}{2}}/\sqrt{T}$. The reason for this is that the bounds were obtained by interpolation with an $L^\infty$ estimate which is substantially stronger than $h^{\frac{1-n}{2}}/\sqrt{T}$.

We now estimate the number of non-self looping tubes $\T_j$ with $j \in \A_k$. That is, tubes on which the $L^2$ mass of $u$ is comparable $2^{-k}\|u\|\sub{P,T}$.
\begin{lemma}
\label{l:propLowerBound}
{Let $k \in \mathbb Z$, $k\geq -1$,} and $t_0>1$. Suppose that $\mc{G}\subset \A_k$ is such that 
$$
\bigcup_{j\in\mc{G}}\T_j\;\;\text{ is }[t_0,T]\text{ non-self looping}.
$$
Then, there exists a constant {$C_n>0$}, {depending only on $n$,} such that
$
|\mathcal{G}|\leq \frac{C_n{t_0}}{T}\, 2^{2k}.
$
\end{lemma}

\begin{proof}
Using that $\G\subset \A_k$, we have
\begin{equation}
\label{e:Gbd}
|\G| \frac{\|u\|\sub{P,T}^2}{2^{2(k+1)}}\leq 2\sum_{j\in \G} \Big(\|Op_h(\chi\sub{\T_j})u\|^2+h^{-2}\|Op_h(\chi\sub{\T_j})Pu\|_{L^2}^2\Big).
\end{equation}
Since $\{\T_j\}_{j \in \mc{G}}$ is $(\mathfrak{D}_n,\tau,R(h))$-good, there are $\{\mc{G}_i\}_{i=1}^{ \mathfrak{D}_n}\subset \mc{G}$, such that for each $i=1, \dots, \mathfrak{D}_n$,
$$
\T_{j}\cap \T_k=\emptyset,\qquad j,k \in \mc{G}_i, \;\;\;j\neq k.
$$
By~\cite[Lemma 4.1]{CG19a} with $t_\ell=t_0$ and $T_\ell=T$ for all $\ell$,
\begin{equation}
\label{e:propEst}
 \sum_{j\in \G} \|Op_h(\chi\sub{\T_j})u\|_{L^2}^2
 \leq \sum_{i=1}^{\mathfrak{D}_n}\sum_{j\in \mc{G}_i}\|Op_h(\chi\sub{\T_j})u\|_{L^2}^2
 \leq \frac{\mathfrak{D}_n4 {t_0}}{T}\|u\|\sub{P,T}^2.
\end{equation}
On the other hand, since $\sum_{j\in \mc{G}_i}\|Op_h(\chi\sub{\T_j})\|^2 \leq 2$ for each $i$,
\begin{equation}
\label{e:Pbd}
    \sum_{j\in \G} \|Op_h(\chi\sub{\T_j})Pu\|_{L^2}^2\leq {2\mathfrak{D}_n}\|Pu\|_{L^2}^2.
\end{equation}
Combining~\eqref{e:Gbd},~\eqref{e:propEst}, and~\eqref{e:Pbd} yields
\begin{equation*}
|\G| \frac{\|u\|\sub{P,T}^2}{2^{2(k+1)}}\leq \frac{8\mathfrak{D}_n{t_0}}{T}\|u\|\sub{P,T}^2 +\frac{4 \mathfrak{D}_n}{h^{2}}  \|Pu\|_{L^2}^2  \leq   \frac{8\mathfrak{D}_n{t_0}+{\tfrac{4\mathfrak{D}_n}{T}}}{T}\|u\|\sub{P,T}^2.
\end{equation*}

\vspace{-1cm}
\end{proof}


We may now proceed to estimate the $L^p$-norm of the non-looping piece \eqref{E:non-looping part}. 
The first step is to notice that we only need to sum up to $m \leq m_{3,k}$, where $m_{3,k}$ is defined by
$$
2^{m_{3,k}}:=\min\Bigg( \frac{{C_n}t_02^{2k}}{{c_{_{\!M}}} T}\;,\; c_0R(h)^{1-n}\Bigg),
$$
and $c_{_{\!M}}>0$ is as defined in \eqref{e:bound1} {and $C_n>0$ is the constant in Lemma \ref{l:propLowerBound}}.
To see this, first observe that, using \eqref{E: I_km},~\eqref{e:boundCount} and~\eqref{e:LinfBad}, for each $\alpha \in \I_{k,m}$
\begin{equation}
\label{e:LinfGood1}
\begin{aligned}
\|w_{k,m}^{\mc{G}}\|\sub{L^\infty (B((x_\alpha, R(h))))}&\leq \|w_{k,m}\|\sub{L^\infty (B((x_\alpha, R(h))))}+\|w_{k,m}^{\mc{B}}\|\sub{L^\infty (B((x_\alpha, R(h))))}\\
&\leq C(2^{m}+|\I_{k,m}||\mc{B}\sub{U}|)2^{-k}h^{\frac{1-n}{2}}R(h)^{\frac{n-1}{2}}\|u\|\sub{P,T}\\
&\leq C(1+R(h)^{1-n}2^{-3m}|\A_{k,m}||\mc{B}\sub{U}|)2^{m-k}h^{\frac{1-n}{2}}R(h)^{\frac{n-1}{2}}\|u\|\sub{P,T}.
\end{aligned}
\end{equation}
Furthermore,  since $|\mc{G}_{k,m}|\geq |\A_{k,m}|-|\I_{k,m}|^2|\mc{B}\sub{U}|$ and $\mc{G}_{k,m}$ is $[t_0,T]$ non-self looping, Lemma~\ref{l:propLowerBound} yields {the existence of $C_n>0$ such that}
$$ |\A_{k,m}|-|\I_{k,m}|^2|\mc{B}\sub{U}| \leq C_n\frac{t_0}{T}2^{2k} .$$

Next, since $m_{1,k} \leq m \leq m_{2,k}$, we may apply Lemma \ref{l: |I_km|} to bound $|\I_{k,m}|$ as in \eqref{e:boundCount} to obtain that for some $C>0$ 

\begin{equation}
\label{e:nonLoopCount}
|\A_{k,m}|(1-CR(h)^{2(1-n)}2^{-4m}|\A_{k,m}||\mc{B}\sub{U}|)\leq C_n\frac{t_0}{T}2^{2k}.
\end{equation}
{In addition}, provided 
\begin{equation}\label{e: bound on Bu}
|\mc{B}\sub{U}|R(h)^{n-1}\ll \Ti^{-6N},
\end{equation}
we have, that for $m\geq m_{1,k}$ and  $k_1\leq k\leq k_2$
\begin{align}\label{E:noname bound}
R(h)^{2(1-n)}2^{-4m}|\A_{k,m}||\mc{B}\sub{U}|&\leq R(h)^{2(1-n)}2^{-4m+2k}|\mc{B}\sub{U}|\leq  2^{-2k}T^{4N}|\mc{B}\sub{U}|\notag\\
&\leq  {R(h)}^{n-1}T^{6N}|\mc{B}\sub{U}|\ll 1,\end{align}
where we used that by~\eqref{e:bound1}, $|A_{k,m}|$ is controlled by $2^{2k}$ to get the first inequality, that $m\geq m_{1,k}$ to get the second, and that $k\geq k_1$ to get the third.
Combining~\eqref{e:nonLoopCount} and the bound in \eqref{E:noname bound} we obtain
$
|\A_{k,m}|\leq {C_n}\frac{t_02^{2k}}{T},
$
and so, by  \eqref{e:bound1}, 
$
2^m\leq {C_n}\frac{t_02^{2k}}{{c_{_{\!M}}} T}.
$
{As claimed, this shows that to deal with \eqref{E:non-looping part} we only need to sum up to $m \leq m_{3,k}$.

The next step is to use interpolation to control the first sum in \eqref{E:non-looping part} by
\begin{equation}\label{E:interpolation good piece}
{\sum_{m=m_{1,k}}^{m_{2,k}}\|w_{k,m}^\mc{G}\|_{L^p(U_{k,m})}^p=}\sum_{m=m_{1,k}}^{m_{3,k}}\|w_{k,m}^\mc{G}\|_{L^p(U_{k,m})}^p
\leq \!\!\!\sum_{m=m_{1,k}}^{m_{3,k}} \|w_{k,m}^\mc{G}\|_{L^\infty(U_{k,m})}^{p-p_c}\|w_{k,m}^\mc{G}\|^{p_c}_{L^{p_c}(U_{k,m})}.
\end{equation}

We claim that \eqref{e:LinfGood1} yields
\begin{equation}\label{e:Linfty good piece}
\|w_{k,m}^{\mc{G}}\|\sub{L^\infty (B(x_\alpha, R(h)))}\leq C2^{m-k}h^{\frac{1-n}{2}}R(h)^{\frac{n-1}{2}}\|u\|\sub{P,T}.
\end{equation}
Indeed, using the  bound \eqref{e: bound on Bu} on $|\mc{B}\sub{U}|$, that $|\A_{k,m}|$ is controlled by $2^{2k}$, that $m\geq m_{1,k}$ as in~\eqref{e:defM1}, and that  $k_1\leq k\leq k_2$, we have
\begin{align*}
R(h)^{1-n}2^{-3m}|\A_{k,m}||\mc{B}\sub{U}|
\ll R(h)^{2(1-n)}2^{-3m+2k}\Ti^{-6N}
\leq \Ti^{-2N}.
\end{align*}

Using \eqref{e:Linfty good piece}, the standard bound on $\|w_{k,m}^\mc{G}\|^{p_c}_{L^{p_c}(U_{k,m})}$, and $\|w_{k,m}^\G\|_{L^2}^2 \leq C \frac{t_0}{T}$,  we obtain
\begin{equation}
\label{e:OhMarcus}
\|w_{k,m}^{\mc{G}}\|_{L^p(U_{k,m})}^2\leq Ch^{-p\delta(p)}(R(h)^{\frac{n-1}{2}}2^{m-k})^{p-p_c}\frac{t_0^{\frac{p_c}{2p}}}{T^{\frac{p_c}{2p}}} +O(h^\infty\|u\|\sub{P,T}^p).
\end{equation}
Using this, we estimate \eqref{E:interpolation good piece} 
\begin{align}\label{e:before358}
&\sum_{m=m_{1,k}}^{{m_{2,k}}}\|w_{k,m}^\mc{G}\|_{L^p(U_{k,m})}^p
\leq C h^{-p\delta(p)}(R(h)^{\frac{n-1}{2}}2^{(m_{3,k}-k)})^{p-p_c}\|u\|\sub{P,T}^{p}
\frac{t_0^{\frac{p_c}{2}}}{T^{\frac{p_c}{2}}} +O(h^\infty\|u\|\sub{P,T}^p).
\end{align}

Then, summing in $k$, and again using that only $k_1\leq k\leq k_2$ contribute,
\begin{align}
    {\sum_{k=-1}^{\infty}}\Big(\sum_{m=m_{1,k}}^{{m_{2,k}}}\|w_{k,m}^\mc{G}\|_{L^p(U_m)}^p\Big)^{\frac{1}{p}}&\leq C h^{-\delta(p)}\|u\|\sub{P,T}
\frac{t_0^{\frac{p_c}{2p}}}{T^{\frac{p_c}{2p}}}
    \sum_{k=k_1}^{k_2} \big(R(h)^{\frac{n-1}{2}}2^{(m_{3,k}-k)}\big)^{1-\frac{p_c}{p}} +O(h^\infty\|u\|\sub{P,T})\notag\\
&\qquad\leq C h^{-\delta(p)}\frac{t_0^{\frac{1}{2}}}{T^{\frac{1}{2}}}\|u\|\sub{P,T}{+O(h^\infty\|u\|\sub{P,T}).}\label{e:squid3}
\end{align}
}
Note that the sum over $k$ in~\eqref{e:squid3} is controlled by the value of $k$ for which  
$
\frac{C_nt_0 2^{2k}}{c_MT}=c_0R(h)^{1-n},
$
since the sum is geometrically increasing before such $k$ and geometrically decreasing afterward.


\subsection{Wrapping up the proof of Theorem~\ref{t:main bound} }
Combining~\eqref{e:squid1},~\eqref{e:squid2},~\eqref{e:squid3}, with \eqref{e:lps} and \eqref{e:splitMe}, and taking $N{> \frac{1}{2}}(1-\frac{p_c}{p})^{-1},$ provided
$
R(h)^{n-1}|\mc{B}\sub{U}|\leq C \Ti^{-6N},
$
{for some $C>0$,} we obtain
$$
\|v\|_{L^p(U)}\leq \sum_{k=-1}^\infty \|w_k\|_{L^p(U)}\leq C h^{-\delta(p)}\Bigg(\frac{t_0^{\frac{1}{2}}}{{T}^{\frac{1}{2}}}+{(R(h)^{n-1}|\mc{B}\sub{U}|)^{1-\frac{p_c}{2p}}\Ti^{3N}}\Bigg)\|u\|\sub{P,T}
$$
 as requested in \eqref{e: bound on Bu}. {Since this estimate holds only when $|\mc{B}\sub{U}|R(h)^{n-1}\leq {C}T^{-6N}$, we replace $\Ti$ by $\Ti_0:=\min\{{\tfrac{1}{C}}(R(h)^{n-1}|\mc{B}\sub{U}|)^{-\frac{1}{6N}}\,,\, \Ti\}$ so that 
\begin{equation}
\label{e:almostThere}
\begin{aligned}
\|v\|_{L^p(U)}&\leq C h^{-\delta(p)}\Bigg(\frac{t_0^{\frac{1}{2}}}{{\Ti_0}^{\frac{1}{2}}}+{(R(h)^{n-1}|\mc{B}\sub{U}|)^{1-\frac{p_c}{2p}}\Ti_0^{3N}}\Bigg)\|u\|\sub{P,T}\\
&\leq Ch^{-\delta(p)}\Bigg(\frac{t_0^{\frac{1}{2}}}{{\Ti}^{\frac{1}{2}}}+{t_0^{\frac{1}{2}}}(R(h)^{n-1}|\mc{B}\sub{U}|)^{\frac{1}{12N}}+{(R(h)^{n-1}|\mc{B}\sub{U}|)^{\frac{1}{2}(1-\frac{p_c}{p})}}\Bigg)\|u\|\sub{P,T}\\
&\leq Ch^{-\delta(p)}\Bigg(\frac{t_0^{\frac{1}{2}}}{{\Ti}^{\frac{1}{2}}}+{(R(h)^{n-1}|\mc{B}\sub{U}|)^{\frac{1}{12N}}}\Bigg)\|u\|\sub{P,T},
\end{aligned}
\end{equation}
{where the constant $C$ is adjusted from line to line.}

Next, combining~\eqref{e:almostThere} with~\eqref{e:rico} and the definition of $v$ in~\eqref{e:individualized}, we obtain
\begin{multline*}
\|u\|_{L^p(U)}\leq Ch^{-\delta(p)}\Bigg(\frac{t_0^{\frac{1}{2}}}{{\Ti}^{\frac{1}{2}}}+{(R(h)^{n-1}|\mc{B}\sub{U}|)^{\frac{1}{12N}}}\Bigg)\|u\|\sub{P,T} 
+Ch^{-\delta(p)+\frac{1}{2}-\delta_2}h^{-1}\|Pu\|_{H_h^{\Lpexp-2}}.
\end{multline*}
Putting $\e=\frac{1}{2}$ and setting $N=\frac{1}{2}(1+\frac{\e_0}{6})(1-\frac{p_c}{p})^{-1}$, the estimate~\eqref{e:LpestFinal} will follow once we relate $|\mc{B}\sub{U}|$ for a given $(\tau, R(h))$ cover to $|\mc{B}\sub{U}|$ for the $(\mathfrak{D},\tau,R(h))$ cover used in our proof.

Finally, to finish the proof of Theorem~\ref{t:main bound}, we need to show that for any $(\tau,R(h))$ cover $\{{\T}_j\}_j$ of $\SM$, up to a constant depending only on $M$, {$|\mc{B}\sub{U}|$ can be bounded by $|{\tilde{\mc{B}\sub{U}}}|$ where ${\tilde{\mc{B}}\sub{U}}$ is defined as in~\eqref{e:globB} using a $(\tilde{\mathfrak{D}},\tau,R(h))$ good cover $\{\tilde{\T}_k\}_k$ of $\SM$.}

\begin{lemma}
There exists $C_{_{\!M}}>0$ depending only on $M$ so that if $\{\T_j\}_{j\in \J}$ and $\{\tilde{\T}_k\}_{k\in \mc{K}}$ are respectively a $(\tau,R(h))$ cover $\SM$ and a $(\tilde{\mathfrak{D}},\tau,R(h))$ good cover of $\SM$, and {$|\mc{B}\sub{U}|$, $|\tilde{\mc{B}}\sub{U}|$} are defined as in~\eqref{e:globB} for respectively the covers $\{\T_j\}_{j\in \J}$,$\{\tilde{\T}_k\}_k$, then 
$$
|\tilde{\mc{B}}\sub{U}|\leq C_{_{\!M}}\tilde{\mathfrak{D}}|\mc{{B}}\sub{U}|.
$$
\end{lemma}
\begin{proof}
Fix $\alpha,\beta$ such that $x_\alpha,x_\beta \in U$. Suppose that $j\in {\mc{B}\sub{U}(\alpha,\beta)}$ where $\mc{B}\sub{U}(\alpha,\beta)$ is as in~\eqref{e:indB}. Then, there is $k\in \tilde {\mc{B}}\sub{U}(\alpha,\beta)$ such that $\tilde{\T}_k\cap \T_j\neq \emptyset.$ Now, fix $j\in \J$ and let 
$$\mc{C}_j:=\{ k\in \mc{K}:\;\T_j\cap \tilde{\T}_k\neq \emptyset\}.$$
We claim that there is $c\sub{M}>0$ such  that for each  $k\in \mc{C}_j$
\begin{equation}
\label{e:inside}
\tilde{\T}_k\subset \Lambda_{\rho_j}^{c\sub{M}\tau}(c\sub{M}R(h)).
\end{equation}
Assuming~\eqref{e:inside} for now, {there exists $C\sub{M}>0$ such that}
$$
|\mc{C}_j|\leq \tilde{\mathfrak{D}}\frac{\vol(\Lambda_{\rho_j}^{c\sub{M}\tau}(c\sub{M}R(h))}{\inf_{k\in \mc{K}}\vol(\tilde{\T}_k)}\leq \tilde{\mathfrak{D}}C_{_{\!M}}.
$$
Thus, for each $j\in \mc{B}\sub{U}({\alpha ,\beta})$, there are at most $C_{_{\!M}}\mathfrak{\tilde{D}}$ elements in $\tilde{\mc{B}}\sub{U}({\alpha,\beta})$ and hence $|\mc{B}\sub{U}(\alpha,\beta)|\geq |\tilde{\mc{B}}\sub{U}(\alpha,\beta)|/(C_{_{\!M}}\tilde{\mathfrak{D}})$ as claimed.

We now prove~\eqref{e:inside}. Let $q\in \tilde{\T}_k$. Then, there are $\rho'_k,\rho'_j,q'\in \SM$ and $t_k,t_j,s\in[\tau-R(h),\tau+R(h)]$ such that 
$$
\begin{gathered}
d(\rho_k,\rho'_k)<R(h),\qquad \quad d(\rho_j,\rho'_j)<R(h),\qquad \quad d(\rho_k,q')<R(h), \\
\varphi_{t_k}(\rho'_k)=\varphi_{t_j}(\rho'_j),\qquad \qquad \varphi_{s}(q')=q.
\end{gathered}
$$
In particular, $d(q',\rho'_k)<2R(h)$, so there is $c\sub{M}>0$ such that
$
d(\varphi_{t_k}(\rho'_k),\varphi_{t_k-s}(q))<c\sub{M}R(h).
$
Applying $\varphi_{-t_j}$, and adjusting $c\sub{M}$ in a way depending only on $M$,
$
d(\rho'_j,\varphi_{t_k-t_j-s}(q))<c\sub{M}R(h).
$
In particular, adjusting $c\sub{M}$ again,
$
d(\rho_j,\varphi_{t_k-t_j-s}(q))<c\sub{M}R(h)
$
and the claim follows.
\end{proof}

\subsection{Proof of Theorem \ref{t:JeffsFavorite}} \label{s:JeffsFavorite}
As explained in the introduction, {Theorem \ref{t:JeffsFavorite} actually holds under the more general assumptions of Theorem \ref{t:main bound}}. Let $p>p_c$ and assume that there is $\delta>0$ such that  
$$
T=T(h)\to \infty,\qquad\qquad |\mc{B}\sub{U}|R(h)^{n-1}T^{\frac{3p}{p-p_c}+\delta}=o(1).
$$
{In the general setup we work with
$$
\mc{S}\sub{U}(h, \e,u):=\Big\{\alpha \in \mc{I}(h): \|u\|_{L^\infty (B(x_\alpha,R(h)))} \geq  \frac{\e h^{\frac{1-n}{2}}{\sqrt{t_0}}}{\sqrt{T(h)}}\|u\|_{\LM}, \;\; B(x_\alpha,R(h))\cap U\neq \emptyset\Big\}.
$$}
We proceed to prove Theorem \ref{t:JeffsFavorite} in this setup, using the decompositions introduced in the previous sections. Throughout this proof we assume that 
\begin{equation}\label{e:quas}
\|Pu\|_{H_h^{\frac{n-3}{2}}}=o\Big(\frac{h}{T}\|u\|_{L^2}\Big).
\end{equation}

\subsubsection{Proof of the bound on $|\mc{S}\sub{U}(h,\e,u)|$}
{We claim that there is $c>0$ such that for  $\alpha\in S\sub{U}(h,\e,u)$ 
\begin{equation}\label{e:LB1}
\frac{c\e\sqrt{t_0}}{\sqrt{T}}h^{-\frac{1}{p}}\|u\|_{\sub{P,T}}\leq\|u\|_{L^{p}(B(x_\alpha,2R(h)))}.
\end{equation}}
{To see \eqref{e:LB1}, first let $\chi_0,\chi_1 \in C_c^\infty(-2,2)$, $\chi\equiv 1$ on $[-3/2,3/2]$, $\chi_1\equiv 1$ on $\supp \chi_0$ and  note that by Lemma~\ref{L:lp bound}, the elliptic parametrix construction for $P$, and~\eqref{e:quas}
\begin{equation}
\label{e:highFreq}
\|(1-\chi_0(-h^2\Delta_g))u\|_{L^p} \leq Ch^{-\delta(p)-\frac{1}{2}}\|Pu\|_{H_h^{\frac{n-3}{2}}}=o\Big(\frac{h^{-\delta(p)+\frac{1}{2}}}{T}\Big){\|u\|_{L^2}}.
\end{equation}}
Therefore, for $\alpha\in S\sub{U}(h,\e,u)$ 
we have
\begin{equation}\label{e:C_h}
\|\chi_0(-h^2\Delta_g) u\|_{L^\infty(B(x_\alpha,R(h)))}\geq    {\frac{\e h^{\frac{1-n}{2}}}{2\sqrt{T}}\|u\|_{\LM}}
\end{equation}
for $h$ small enough.  Next, set $\chi_{\alpha,h}(x):=\chi(R(h)^{-1}d(x,x_\alpha))$ and note
$$
\chi_1(-h^2\Delta_g) \chi_{\alpha,h}\chi_0(-h^2\Delta_g)u=\chi_{\alpha,h}\chi_0(-h^2\Delta_g)u+O(h^\infty\|u\|_{L^2})_{C^\infty}.
$$
Then, by \eqref{e:C_h} and \cite[Theorem 7.15]{EZB}
\begin{align}\label{e:LB2}
{\frac{\e h^{\frac{1-n}{2}}}{2\sqrt{T}}\|u\|_{\LM}}\leq \|\chi_0(-h^2\Delta_g)u&\|_{L^\infty(B(x_\alpha,R(h)))}
\leq \|\chi_{\alpha,h}\chi_0(-h^2\Delta_g)u\|_{{L^\infty(B(x_\alpha,R(h)))}} \notag\\
&\leq Ch^{-\frac{n}{p}}\Big({\|\chi_0(-h^2\Delta_g)u\|}_{L^{p}(B(x_\alpha,2R(h))}+O(h^\infty)\|u\|_{L^2}\Big),
\end{align}

Combining \eqref{e:LB2} and \eqref{e:highFreq} yields the claim in \eqref{e:LB1}. It then follows that, if $\{\alpha_i\}_{i=1}^N\subset S\sub{U}(h,\e,u)$ with $B(x_{\alpha_i},2R(h))\cap B(x_{\alpha_j},2R(h))=\emptyset$ for $i\neq j$, then {using Theorem~\ref{t:main bound},}
$$
N^{\frac{1}{p}}\frac{c\e\sqrt{t_0}}{\sqrt{T}}h^{-\frac{1}{p}}\|u\|_{\sub{P,T}}\leq \|u\|_{L^{p}}\leq Ch^{-\frac{1}{p}}\|u\|_{L^{2}}\leq C h^{-\frac{1}{p}}\frac{\sqrt{t_0}}{\sqrt{T}}\|u\|_{\sub{P,T}}.
$$
Then, 
$
N^{\frac{1}{p}}\leq C\e^{-1}.
$
Since  at most $\mathfrak{D}_n$ balls $B(x_\alpha,2R(h))$ intersect, 
$
|S\sub{U}(h,\e,u)|\leq C\mathfrak{D}_n\e^{-p}.
$

\subsubsection{Preliminaries for the decomposition of $u$}
Let $q\in \mathbb{R}$ such that $p\leq q\leq \infty$. Below, all implicit constants are uniform for $p\leq q\leq \infty$.
As above, it suffices to prove the statement for $v$ as in~\eqref{e:individualized} instead of $u$. Then, we decompose $v=\sum_{k=-1}^\infty w_k$ as in~\eqref{e:v}. For $V\subset U$, by {the same analysis that led to}~\eqref{e:splitMe},
$$
\|w_k\|^q_{L^q(V)}\leq \mathfrak{D}_n\sum_{m=-\infty}^{m_{2,k}}\|w_{k,m}\|^q_{L^q(V\cap U_{k,m})} +O(h^\infty)\|u\|\sub{P,T},
$$
where $w_{k,m}$ is as in \eqref{E:w_km}. Then, by~\eqref{e:squid1}, with  $N=\frac{q}{2(q-p_c)}+\frac{\delta}{6}$
\begin{equation}
\label{e:marcus0}
\sum_{k\geq -1}\Big(\sum_{m=-\infty}^{m_1,k}\|w_{k,m}\|_{L^q( U_{k,m})}^{{q}}\Big)^{{\frac{1}{q}}}\leq Ch^{-\delta(q)}\frac{\log T}{T^{\frac{1}{2}+ {\frac{\delta(q-p_c)}{6q}}}}\|u\|_{\sub{P,T}},
\end{equation}
for $h$ small enough.
Then, splitting $w_{k,m}=w_{k,m}^{\mc{B}}+w_{k,m}^{\mc{G}}$, as in~\eqref{e:theGoodTheBadAndTheUgly}, we have by~\eqref{e:squid2}
\begin{equation}
\label{e:marcus1}
\sum_{k\geq -1}\Big(\sum_{m=m_{1,k}}^{m_{2,k}}\|w_{k,m}^\mc{B}\|_{L^q( U_{k,m})}^{{q}}\Big)^{{\frac{1}{q}}}\leq Ch^{-\delta(q)}(R(h)^{n-1}|\mc{B}\sub{U}|)^{1-\frac{p_c}{2q}}T^{\frac{3q}{2(q-p_c)}+\frac{\delta}{2}}\|u\|\sub{P,T}.
\end{equation}

Define ${k_1^\e}$ and ${k_2^\e}$,  by
\begin{equation}\label{e:Kep}
2^{2k_1^\e}=\frac{C^{-2}\mathfrak{D}_n^{-2}\e^2 R(h)^{1-n} c\sub{M} T}{4C_n t_0},\qquad\quad 2^{2k_2^\e}=\frac{C^2\mathfrak{D}_n^2\e^{-2} R(h)^{1-n} c\sub{M} T}{4C_n t_0},
\end{equation}
where $C$ is as in ~\eqref{e:squid3}. Then, define 
$
\mc{K}(\e):=\{k: k_1^\e\leq k\leq k_2^\e\}
$
and note that, {since $2^{(k_2^\e-k_1^\e)}=C^2\mathfrak{D}_n^2\e^{-2}$,} $|\mc{K}(\e)|\leq \log_2(4C^2\mathfrak{D}_n^2\e^{-2})=:K_\e$. Using~\eqref{e:OhMarcus} and summing over $k\notin \mc{K}(\e)$, it follows that we have
\begin{equation}
\label{e:marcus2}
\sum_{k\notin \mc{K}(\e)}\Big(\sum_{m=m_{1,k}}^{m_{3,k}} \|w_{k,m}^\mc{G}\|_{L^q(U_{k,m})}^q\Big)^{\frac{1}{q}}\leq \frac{\e}{4\mathfrak{D}_n} \frac{h^{-\delta(q)}\sqrt{t_0}}{\sqrt{T}}\|u\|\sub{P,T}.
\end{equation}
Next, for $k\in \mc{K}(\e)$ let
$$
\mc{M}(k,\e):=\{ m\,:\,m^\e_{3,k}\leq m\leq m_{3,k}\}, \qquad\qquad {m^\e_{3,k}:=m_{3,k}-\tfrac{p}{p-p_c}\log_2 (\e^{-1}2C\mathfrak{D}_n),}
$$
and note $|\mc{M}(k,\e)|\leq \frac{p}{p-p_c}\log_2 (\e^{-1}2C\mathfrak{D}_n):=M_\e$. Using~\eqref{e:OhMarcus} and summing over $k\in \mc{K}(\e)$, $m\notin \mc{M}(k,\e)$, it follows that
\begin{align}\label{e:marcus3}
\sum_{k\in \mc{K}(\e)}\!\!\Big(\sum_{{m\notin \mc{M}(k,\e)}}\!\!\|w_{k,m}^{\mc{G}}\|_{L^q(U_{k,m})}^{{q}}\Big)^{{\frac{1}{q}}}
\!\!
&\leq Ch^{-\delta(q)}\frac{t_0^{\frac{p_c}{2q}}}{T^{\frac{p_c}{2q}}}\!\!\sum_{k\in \mc{K}(\e)}\!(R(h)^{\frac{n-1}{2}}2^{m^\e_{3,k}-k})^{1-\frac{p_c}{q}}\|u\|\sub{P,T}+O(h^\infty\|u\|\sub{P,T}) \notag\\
&\leq \frac{\e}{4\mathfrak{D}_n} \frac{h^{-\delta(q)}t_0^{\frac{1}{2}}}{T^{\frac{1}{2}}}\|u\|\sub{P,T}.
\end{align}
Let
\begin{equation}
\label{e:marcus4}
\mc{N}_{k,m}(\e):=\Big\{\alpha \in \mc{I}_{k,m}\,:\, \|w_{k,m}^{\mc{G}}\|_{L^\infty(B(x_\alpha,R(h)))}\geq \frac{\e}{4\mathfrak{D}_n M_\e K_\e}\frac{h^{\frac{1-n}{2}}\sqrt{t_0}}{\sqrt{T}}\|u\|\sub{P,T}\Big\}.
\end{equation}
We claim 
\begin{equation}\label{e:cliamNe0}
\mc{S}\sub{U}(h,\e,u) \subset \bigcup_{k\in \mc{K}(\e)}\bigcup_{m\in \mc{M}(k,\e)}\mc{N}_{k,m}(\e).
\end{equation}
To prove the claim \eqref{e:cliamNe0},
suppose $\alpha\notin \bigcup_{k\in \mc{K}(\e)}\bigcup_{m\in \mc{M}(k,\e)}\mc{N}_{k,m}(\e)$. Then, using ~\eqref{e:marcus0} {with $q=\infty$ and $N=\tfrac{1}{2}+\frac{\delta}{6}$}, 
\begin{align}\label{e:rain}
&\frac{1}{\mathfrak{D}_n}\|v\|_{L^\infty(B(x_\alpha,R(h)))}
\leq \frac{Ch^{\frac{1-n}{2}} \log T}{T^{\frac{1}{2}+\frac{\delta}{6}}}\|u\|_{\sub{P,T}} + \sum_{k\geq -1}\sum_{m=m_{1,k}}^{m_{2,k}}\|w_{k,m}\|_{L^\infty(U_{k,m})}.
\end{align}
Next, we decompose the second term in the RHS of \eqref{e:rain} as 
\begin{align}\label{e:cloud}
 \sum_{k\geq -1}\sum_{m=m_{1,k}}^{m_{2,k}}\|w_{k,m}^{\mc{B}}\|_{L^\infty(U_{k,m})}+\sum_{k\notin \mc{K}(\e)}\sum_{m=m_{1,k}}^{m_{3,k}}\|w_{k,m}^{\mc{G}}\|_{L^\infty(U_{k,m})}+\sum_{k\in \mc{K}(\e)}\sum_{m=m_{1,k}}^{m_{2,k}}\|w_{k,m}^{\mc{G}}\|_{L^\infty(U_{k,m})}
\end{align}
Note that in the term with the sum over $k\notin \mc{K}(\e)$ we only sum in $m\leq m_{3,k}$ for the same reason as in \eqref{E:interpolation good piece}.
We bound the three terms in \eqref{e:cloud} using ~\eqref{e:marcus1}, ~\eqref{e:marcus2}, ~\eqref{e:marcus3}, and ~\eqref{e:marcus4}, {with $q=\infty$ and $N=\tfrac{1}{2}+\frac{\delta}{6}$}. 
Combining it with \eqref{e:rain} this yields
\begin{align*}
&\frac{1}{\mathfrak{D}_n}\|v\|_{L^\infty(B(x_\alpha,R(h)))}
\leq Ch^{\frac{1-n}{2}}\|u\|_{\sub{P,T}}\Big( \frac{\log T}{T^{\frac{1}{2}+\frac{\delta}{6}}}+R(h)^{n-1}|\mc{B}\sub{U}|T^{\frac{3}{2}+\frac{\delta}{2}}+\frac{3\e}{4\mathfrak{D}_n}\frac{\sqrt{t_0}}{\sqrt{T}}+O(h^\infty)\Big).
\end{align*}
Thus, if $\alpha\notin \bigcup_{k\in \mc{K}(\e)}\bigcup_{m\in \mc{M}(k,\e)}\mc{N}_{k,m}(\e)$, then
$
\|v\|_{L^\infty(B(x_\alpha,R(h)))}\leq \e h^{\frac{1-n}{2}}\frac{\sqrt{t_0}}{\sqrt{T}}\|u\|\sub{P,T}
$
for $h$ small enough.
In particular, $\alpha\notin \mc{S}\sub{U}(h,\e,u)$. This proves the claim \eqref{e:cliamNe0}


\subsubsection{Decomposition of $u$} We next decompose $u$ as described in the theorem.
First, put 
$$
u_{e,1}:=\sum_{k\geq -1}\sum_{m=-\infty}^{m_1,k}w_{k,m}+\sum_{k\geq -1}\sum_{m=m_{1,k}}^{m_{2,k}}w_{k,m}^{\mc{B}}+\sum_{k\notin \mc{K}(\e)}\sum_{m=m_{1,k}}^{m_{3,k}}w_{k,m}^{\mc{G}}+\sum_{k\in \mc{K}(\e)}\sum_{m\notin \mc{M}(k,\e)}w_{k,m}^{\mc{G}},
$$
$$
u_{big}:=\sum_{k\in \mc{K}(\e)}\sum_{m\in \mc{M}(k,\e)}w_{k,m}^{\mc{G}}.
$$
and $u_{e,2}:=u-u_{big}-u_{e,1}$. Note that
$$
\|u_{e,1}\|_{L^q}\leq \frac{3\e}{4}h^{-\delta(q)}\frac{\sqrt{t_0}}{\sqrt{T}}\|u\|\sub{P,T},\qquad\qquad \|u_{e,2}\|_{L^q}\leq  Ch^{-\delta(q)+\frac{1}{2}-\delta_{{2}}}  h^{-1}\|{P}u\|_{H_h^{\frac{n-3}{2}}},
$$
where we use~\eqref{e:marcus1}, ~\eqref{e:marcus2}, ~\eqref{e:marcus3},~\eqref{e:rain}, and~\eqref{e:cloud} to obtain the frist estimate, and~\eqref{e:rico} to obtain the second. These two estimates prove the claim on $\|u_\e\|_{L^q}$ after combining them with \eqref{e:quas}.
Next, observe that
$$
u_{big}= \sum_{j\in \mc{L}(\e)}u_j,\qquad \qquad u_j:=Op_h(\tilde{\chi}\sub{\T_j})Op_h(\psi)u,\qquad \qquad \mc{L}(\e):=\bigcup_{k\in \mc{K(\e)}}\bigcup_{m\in \mc{M}(k,\e)}\mc{G}_{k,m}.
$$
We claim that the statement of the theorem holds with $v_j=\sqrt{T}u_j$. Note that $v_j$ are manifestly microsupported inside $\T_j$.

Let $\alpha \in \mc{S}\sub{U}(h,\e,u)$, then by definition,
\begin{equation}\label{e:uBigLB}
\|u_{big}\|_{L^\infty(B(x_\alpha,R(h)))}\geq \frac{\e}{4}h^{\frac{1-n}{2}}\frac{\sqrt{t_0}}{\sqrt{T}}\|u\|\sub{P,T}. 
\end{equation}
Note that for all $j\in \mc{L}(\e)$, the estimate
\begin{equation}
\label{e:upperA_k}
\|Op_h(\tilde{\chi}\sub{\T_j})Op_h(\psi)u\|+h^{-1}\|Op_h(\tilde{\chi}\sub{\T_j})Op_h(\psi)Pu\|_{L^2}\leq {2^{-k_1^\e+1}}\|u\|\sub{P,T}
\end{equation}
follows from the definition,~\eqref{E:A_k}, of $\A_k$ and the fact that $\chi\sub{\T_j}\equiv 1$ on $\supp \tilde{\chi}\sub{\T_j}$. To see that $u_j$ is a quasimode, we use the definition of $\A_k$ again, together with Proposition~\ref{l:nicePartition}, and obtain
\begin{equation}
\label{e:upperA_kP}
\|Pu_j\|_{L^2} \leq \|[-h^2\Delta_g,Op_h(\tilde{\chi}\sub{\T_j})]u_j\|_{L^2}+\|Op_h(\tilde{\chi}\sub{\T_j})Pu\|_{L^2}\leq C{2^{-k_1^\e}}h\|u\|\sub{P,T}.
\end{equation}
The definition of $k_1^\e$, together with~\eqref{e:upperA_k} and~\eqref{e:upperA_kP} give the required bounds on $v_j$ and $Pv_j$.

Next, define 
$$
\mc{L}(\e, u,\alpha):=\{j\in \mc{L}\,:\,\pi\sub{M}(\T_j)\cap B(x_\alpha,3R(h))\neq\emptyset\},
$$
and note that by~\cite[Lemma 3.7]{CG19a}

\begin{align}
&\|u_{big}\|_{L^\infty(B(x_\alpha,R(h)))}\notag\\
&\leq Ch^{\frac{1-n}{2}}R(h)^{\frac{n-1}{2}}\sum_{j\in \mc{L}(\e, u, \alpha)}\|Op_h(\tilde{\chi}\sub{\T_j})Op_h(\psi)u\|+h^{-1}\|Op_h(\tilde{\chi}\sub{\T_j})Op_h(\psi)Pu\|_{L^2}+O(h^\infty)\|u\|_{L^2}\notag\\
&\leq Ch^{\frac{1-n}{2}}R(h)^{\frac{n-1}{2}}2^{-k_1^\e}|\mc{L}(\e,u,\alpha)|\|u\|\sub{P,T}  +O(h^\infty)\|u\|\sub{P,T}.\label{e:usefulLower}
\end{align}
Therefore, {combining \eqref{e:uBigLB} with \eqref{e:usefulLower} yields}
$$
\e\frac{\sqrt{t_0}}{\sqrt{T}}\leq CR(h)^{\frac{n-1}{2}}2^{-k_1^\e}|\mc{L}(\e,\alpha,u)| +O(h^\infty).
$$
Moreover, $\bigcup_{j\in \mc{L}(\e,u)}\T_j$ is $[t_0,T]$ non-self looping and so by Lemma~\ref{l:propLowerBound}
$
|\mc{L}(\e,u)|\leq \frac{C_nt_0}{T}2^{2k_2^\e}.
$
Using the definition \eqref{e:Kep} of $k^\e_1,k^\e_2$ we have for $h$ small enough,
\begin{equation*}
c\e^2 R(h)^{1-n}=\e \frac{\sqrt{t_0}}{\sqrt{T}}R(h)^{\frac{1-n}{2}}2^{k_1^\e}\leq |\mc{L}(\e,u,\alpha)|\leq |\mc{L}(\e,u)|\leq \frac{C_n t_0}{T}2^{2k_2^\e}\leq  C\e^{-2}R(h)^{1-n},
\end{equation*}
which yields the upper bound on $|\mc{L}(\e,u)|$ and the lower bound on $|\mc{L}(\e,u,\alpha)|$. Note that, the upper bound on $|\mc{L}(\e,u,\alpha)|$ follows from the fact that the total number of tubes over $B(x_\alpha,3R(h))$ is bounded by $CR(h)^{1-n}$.  
Next, we note that the fact that at most $\mathfrak{D}_n$ tubes $\T_j$ overlap implies
$$
\sum_{j\in \mc{L}(\e, u,\alpha)}\|Op_h(\tilde{\chi}\sub{\T_j})Op_h(\psi)Pu\|_{L^2}^2\leq C\|Pu\|^2_{L^2}+O(h^\infty\|u\|_{L^2}).
$$
Therefore, using the first inequality in~\eqref{e:usefulLower} again, applying Cauchy-Schwarz, and using that there is $C>0$ such that $|\mc{L}(\e,u,\alpha)|\leq CR(h)^{1-n}$ we have
$$
\begin{aligned}\frac{\e}{4}\frac{\sqrt{t_0}}{\sqrt{T}}\|u\|\sub{P,T}&\leq CR(h)^{\frac{n-1}{2}}|\mc{L}(\e,u,\alpha)|^{\frac{1}{2}}\Big(\sum_{j\in \mc{L}(\e, u,\alpha)}\|u_j\|_{L^2}^2\Big)^{^{\frac{1}{2}}}+ Ch^{-1}\|Pu\|_{L^2}+O(h^\infty)\|u\|_{L^2}\\
&\leq C\Big(\sum_{j\in \mc{L}(\e, u,\alpha)}\|u_j\|_{L^2}^2\Big)^{^{\frac{1}{2}}}+o(T^{-1}\|u\|_{L^2}).
\end{aligned}
$$
In particular, for $h$ small enough, 
$
c\frac{\sqrt{t_0}}{\sqrt{T}}\|u\|\sub{P,T}\leq \Big(\sum_{j\in\mc{L}(\e, u,\alpha)}\|u_j\|^2\Big)^{\frac{1}{2}}.
$
This completes the proof. 
\qed

\section{Proof of theorem 1}
\label{s:dynamical}

In order to finish the proof of Theorem~\ref{t:noConj}, we need to verify that the hypotheses of Theorem~\ref{t:main bound} hold with $T(h)=b\log h^{-1}$ for some $b>0$, and such that for all $x_1,x_2 \in U$ there is some splitting  $\J_{x_1}=\mc{G}\sub{x_1,x_2}\cup \mc{B}\sub{x_1,x_2}$ of the set of tubes over $x_1 \in M$ with a set of `bad' tubes $\mc{B}\sub{x_1,x_2}$ satisfying
$$(|\mc{B}\sub{x_1,x_2}|R(h)^{n-1})^{\frac{1}{6+\e_0}(1-\frac{p_c}{p})}\leq T(h)^{-\frac{1}{2}}$$ and ${\e_0>0}$. Fix $x_1,x_2\in U$ and let {$F_1,F_2:T^*M \to \re^{n+1}$ be smooth functions so that for $i=1,2$,}
\begin{equation}
 \label{e:defFunction}
\begin{gathered}
 S^*_{x_{i}}M=F_i^{-1}(0),\qquad 
 \tfrac{1}{2}d(q,S^*_{x_{i}}M)\leq |F_i(q)|\leq 2d(q,S^*_{x_i}M),\qquad \max_{|\alpha|\leq 2}(|\partial^\alpha F_i(q)|)\leq 2 \\
 dF_i(q)\text{ has a {right} inverse }{R}_{_{\!F_i}}(q)\text{ with }\|{R}_{_{\!F_i}}(q)\|\leq 2.
 \end{gathered}
 \end{equation}
Define also   $\psi_i:\re\times \TM \to \re^{n+1}$ by
$
\psi_i(t,\rho)=F_i\circ \varphi_t(\rho).
$

 To {find $\mc{B}_{x_1,x_2}$}, we apply the arguments from~\cite[Sections 2, 4]{CG19dyn}. In particular, fix $a>0$ and let $r_{t}:=a^{-1}e^{-a |t|}$. Suppose that
$
d(x_2\,,\,\mc{C}_{x_1}^{n-1,r_{t_0},t_0})>r_{t_0},
$
then for $\rho_0\in S^*_{x_1}M$ with 
$
d(S^*_{x_2}M,\varphi_{t_0}(\rho_0))<r_{t_0}
$
we have by~\cite[Lemma 4.1]{CG19dyn},  that there exists ${\bf w}\in T_{\rho_0}S^*_{x_1}M$ so that 
$$
d(\psi_{2})_{(t_0,\rho_0)}: \re\partial_t\times \re {\bf w} \to T_{\psi_2(t_0,\rho_0)}\re^{n+1}
$$
has a left inverse $L_{(t_0,\rho_0)}$ with
$$
\|L_{(t_0,\rho_0)}\|\leq C\sub{M}\max( a e^{C\sub{M}(a+\Lambda)|t_0|},1).
$$ 

Next, let $\{\Lambda_{\rho_j}^\tau(r_1)$ be a $(\mathfrak{D}\sub{M},\tau,r_1)$-good cover for $\SM$. We apply~\cite[Proposition 2.2]{CG19dyn} to construct $\mc{B}\sub{x_1,x_2}$ and $\mc{G}\sub{x_1,x_2}$. 
\begin{remark}We must point out that we are applying the proof of that proposition rather than the proposition as stated. {The only difference here is that the loops we are interested in go from a point $x_1$ to a point $x_2$ where $x_1$ and $x_2$ are not necessarily equal. This does not affect the proof.}
\end{remark}

{We use \cite[Proposition 2.2]{CG19dyn} to see that} there exist $\alpha_1=\alpha_1(M)>0$, $ \alpha_2=\alpha_2(M,a)$, and ${{\bf{C}_0}={\bf{C}}_0(M,a)}$ so that the following holds.
Let $r_0,r_1,r_2 >0$ satisfy
\begin{equation}
\label{e:assumeR}
r_0 < r_1, \qquad r_1< \alpha_1\, r_2,  \qquad  r_2 \leq \min\{{R_0},{1},  \alpha_2\, e^{-\gamma T}\}, \qquad r_0 < {\tfrac{1}{3}}\, e^{-\Lambda T} r_2,
\end{equation}
where  {$\gamma= 5\Lambda+2a$ and $\Lambda>\Lambda_{\max}$ where $\Lambda_{\max}$ is as in~\eqref{e:LambdaMax}}}. {Then,}  for all balls $B\subset S^*_{x_1}M$ of radius ${R_0}>0$, there is a family {of points}
 $\{\rho_j\}_{j\in \mc{B}\sub{B}}\subset S^*_{x_1}M$ such that
$$
|\mc{B}\sub{B}|\leq  {\bf{C}}_0{\mathfrak{D}_n} \;r_2 \;\frac{R_0^{n-1}}{r_1^{n-1}}\;  T\,e^{4({2}\Lambda+{a})T},
$$ 
and for $j\in \mc{G}\sub{B}:= \{j\in \J_{x_1}:\; {B(\rho_j,2r_1)}\cap B\neq \emptyset \}\setminus \mc{B}\sub{B}\}$ 
$$
\bigcup_{t\in[t_0,T]}\varphi_t\Big(\Lambda_{\rho_j}^\tau(r_1)\Big)\cap \Lambda_{S^*_{x_2}M}^\tau(r_1)=\emptyset.
$$

{We proceed to apply \cite[Proposition 2.2]{CG19dyn}.}
There is ${c\sub{M}r^{1-n}}\geq N\sub{\!{r}}>0$ so that for all $x_1\in M$, $S^*_{x_1}M$ can be covered by $N\sub{\!{r}}$ balls. Let {$0<R_0<1$ and} $\{B_{i}\}_{i=1}^{N\sub{\!{R_0}}}$ be such a cover. Fix ${0}<\e{<\e_1}<{\frac{1}{{4}}}$ and set
 \[r_0:=h^{{\e_1}}, \qquad r_1:=h^\e, \qquad r_2:=\tfrac{2}{\alpha_1}h^\e.\]
Let {$T(h)=b \log h^{-1}$ with $0<b{<\frac{1}{4\Lambda_{\max}}<\frac{1-2\e_1}{2\Lambda_{\max}}}$ to be chosen later}. Then, the assumptions in~\eqref{e:assumeR} hold provided
$$
h^{\e}< \min \Big\{\tfrac{\alpha_1\alpha_2}{2} e^{- \gamma T}, {\tfrac{\alpha_1 {R_0}}{2}}\Big\},\qquad {h^{\e_1-\e}<{\tfrac{2}{3\alpha_1}} e^{-\Lambda T}.}
$$
 In particular, if we set $\alpha_3:= \tfrac{\alpha_1\alpha_2}{2}$, $\alpha_4=\frac{2}{3\alpha_1}$, the assumptions in~\eqref{e:assumeR} hold provided {$h<\big(\frac{\alpha_1 R(h)}{2}\big)^{\frac{1}{\e}}$} and
\begin{equation}
\label{e:t0Temp}
T(h)< {\min\Big\{}\frac{ \e}{ \gamma} \log h^{-1} + \frac{\log \alpha_3}{ \gamma}{\;,\; \frac{\e_1-\e}{\Lambda}\log h^{-1}+\frac{\log(\alpha_4)}{\Lambda}\Big\}}.
\end{equation}

Fix {$b>0$} and ${h_0>0}$ so that 
{$
b < \frac{{\min(\e,\e_1-\e)}}{12(2\Lambda+a)}
$}
{and  \eqref{e:t0Temp} is satisfied for all $h<h_0$.} {Note that this implies that $b=b(M,a)$, {$h_0=h_0(M,a)$}.} Let $\mc{B}\sub{x_1,x_2}:=\bigcup_{i=1}^{N\sub{{R_0}}}\mc{B}\sub{B_i}$. Then, for $j\in \mc{G}\sub{x_1,x_2}:= \J_{x_1}\setminus \mc{B}\sub{x_1,x_2}$,
$$
\bigcup_{t\in[t_0,T_]}\varphi_t\Big(\Lambda_{\rho_j}^\tau(r_1)\Big)\cap \Lambda_{S^*_{x_2}M}^\tau(r_1)=\emptyset.
$$
Moreover, shrinking $h_0$ in a way depending only on $(M,a,\e)$, we have for $0<h<h_0$,
$$
r_1^{n-1}|\mc{B}_{x_1,x_2}|\leq C\sub{M}{\bf{C}}_0{\mathfrak{D}_n}r_2 Te^{4(2\Lambda+a)T}\leq h^\frac{\e}{3}.
$$
Therefore, putting $R(h)=r_1=h^\e$ and $T=T(h)=b\log h^{-1}$ in Theorem~\ref{t:main bound} proves Theorem~\ref{t:noConj}.


\section{Anisotropic Pseudodifferential calculus}
\label{s:anisotropic}

In this section, we develop the second microlocal calculi necessary to understand `effective sharing' of $L^2$ mass between two nearby points. That is, to answer the question: how much $L^2$ mass is necessary to produce high $L^\infty$ growth at two nearby points? To that end, we develop a calculus associated to the co-isotropic
$$
\Gamma_x:=\bigcup_{{|t|<\tfrac{1}{2}\inj(M)}} \varphi_t(\Omega_x),\qquad \Omega_x:=\{\xi\in T^*_xM:\;  |1-|\xi|_g|<\delta\},
$$
which allows for localization to a Lagrangian leaves $\varphi_t(\Omega_x)$. In Section~\ref{S:uncertainty} we will see, using a type of uncertainty principle, that the calculi associated to two distinct points, {$x_\alpha,\, x_\beta \in M$}, are incompatible in the sense that, despite the fact that $\Gamma_{x_\alpha}$ and $\Gamma_{x_\beta}$ intersect in a dimension 2 submanifold, for {operators ${X_{x_\alpha}}$ and $X_{x_\beta}$ localizing to  $\Gamma_{x_\alpha}$ and $\Gamma_{x_\beta}$ respectively}, 
$$
\|X_{x_\alpha}X_{x_\beta}\|_{L^2\to L^2}\ll \|X_{x_\alpha}\|_{L^2\to L^2}\|X_{x_\beta}\|_{L^2\to L^2}.
$$

Let $\Gamma\subset T^*M$ be a co-isotropic submanifold and {$L=\{L_{{q}}\}_{q\in \Gamma}$ be a family of Lagrangian subspaces $L_{{q}}\subset T_{{q}}\Gamma$} that is integrable in the sense that if $U$ is a neighborhood of $\Gamma$, and $V,W$ are smooth vector fields {on $T^*M$} such that $V_{q}, W_{q}\in L_{q}$ for all $q\in \Gamma$, then $[V,W]_{q}\in L_{q}$ for all ${q} \in \Gamma$. The aim of this section is to introduce a calculus of pseudodifferential operators associated to $(L,\Gamma)$ that allows for localization to $h^\rho$ neighborhoods of $\Gamma$ with $0\leq \rho<1$ and is compatible with localization to $h^\rho$ neighborhoods of the foliation of $\Gamma$ generated by $L$. This calculus is close in spirit to those developed in~\cite{SjZw:99} and~\cite{DyZa}. To see the relationships between these calculi, note that the calculus in~\cite{DyZa} allows for localization to any leaf of a Lagrangian foliation defined over an open subset of $T^*M$ and that in~\cite{SjZw:99} allows for localization to a single hypersurface. The calculus developed in this paper is designed to allow localization along leaves of a Lagrangian foliation defined only over a co-isotropic submanifold of $T^*M$. In the case that the co-istropic is a whole open set, this calculus is the same as the one developed in~\cite{DyZa}. Similarly, in the case that the co-isotropic is a hypersurface and no Lagrangian foliation is prescribed, the calculus becomes that developed in~\cite{SjZw:99}.

\begin{definition}
Let $\Gamma$ be a co-isotropic submanifold and $L$ a Lagrangian foliation on $\Gamma$. {Fix  $0\leq \rho<1$ and let $k$ be a positive integer}.  We say that $a\in S^{{k}}_{\Gamma,L,\rho}$ if $a\in C^\infty(T^*M)$, $a$ is supported in an $h$-independent compact set, and 
\begin{equation}
\label{e:anSymbEst}
V_1\dots V_{\ell_1}W_1\dots W_{\ell_2}a=O(h^{-\rho\ell_2}\langle h^{-\rho}d(\Gamma,\cdot)\rangle^{{k}-\ell_2})
\end{equation}
{where $W_1,\dots W_{\ell_2}$ are any vector fields on $T^*M$,  $V_1,\dots V_{\ell_1}$ are vector fields on $T^*M$ with $(V_1)_q,\dots (V_{\ell_1})_q\in L_q$ for $q\in \Gamma$}, {and $q \mapsto d(\Gamma, q)$ is the distance from $q$ to $\Gamma$ induced by the Sakai metric on $T^*M$}.  
\end{definition}

We also define symbol classes associated to only the co-isotropic
\begin{definition}
Let $\Gamma$ be a co-isotropic submanifold. We say that $a\in S^k_{\Gamma,\rho}$ if $a\in C^\infty(T^*M)$, $a$ is supported in an $h$-independent compact set, and 
$$
V_1\dots V_{\ell_1}W_1\dots W_{\ell_2}a=O(h^{-\rho\ell_2}\langle h^{-\rho}d(\Gamma,\cdot)\rangle^{k-\ell_2})
$$
where $V_1,\dots V_{\ell_1}$ are tangent vector fields to $\Gamma$, and $W_1,\dots W_{\ell_2}$ are any vector fields. 
\end{definition}

\subsection{Model case}
The goal of this section is to define the quantization of symbols in $S^{{k}}_{\Gamma_0,L_0,\rho}$, where $\Gamma_0, L_0$ are a model pair of co-isotropic and Lagrangian foliation defined below.
The model co-isotropic submanifolds of dimension $2n-r$ is
{$$
\Gamma_0:=\{(x',x'',\xi', \xi'') \in \R^r\times \R^{n-r} \times \R^r \times \R^{n-r}:\; x'=0\}
$$
with Lagrangian foliation
$$
L_0:= \{L_{0,q}\}_{q \in \Gamma_0}, \qquad L_{0,q}=\text{span}\{\partial_{\xi_i},\,i=1,\dots n\}\subset T_q\Gamma_0.
$$ }
 Note that in this model case the distance from a point $(x,\xi)$ to $\Gamma_0$ is controlled by  $|x'|$. Therefore, $a\in S^{k}_{\Gamma_0,L_0,\rho}$ if and only if $a$ is supported in an $h$-independent compact set and for all $(\alpha, \beta) \in \mathbb N^n \times \mathbb N^n$ there exists $C_{\alpha, \beta}>0$ such that 
{$$
|\partial_{x}^{\alpha} \partial_{\xi}^\beta a|\leq C_{\alpha,\beta} h^{-\rho|\alpha|}\langle h^{-\rho}|x'|\rangle^{k-|\alpha|}.
$$}
In the model case, it will be convenient to define $\tilde{a}\in C^\infty(\re^n_x\times \re^n_\xi\times \re^r_\lambda)$ such that 
$$
a(x,\xi)=\tilde{a}(x,\xi,h^{-\rho}x'),
$$
and for all $(\alpha', \alpha'', \beta, \gamma ) \in \mathbb N^r \times \mathbb N^{n-r} \times \mathbb N^n\times \mathbb N^r$ there exists $C_{\alpha, \beta, \gamma}>0$ such that 
\begin{equation}\label{E: symbol model L}
|\partial_{x'}^{\alpha'}\partial_{x''}^{\alpha''} \partial_\xi^{\beta} \partial_\lambda^{\gamma} \tilde{a}(x,\xi,\lambda)|\leq C_{\alpha, \beta, \gamma} h^{-\rho|\alpha''|}\langle \lambda\rangle^{k-|\gamma|-|\alpha''|}.
\end{equation}
Similarly, if $a\in S^k_{\Gamma_0,\rho}$, then 
for  $(\alpha', \alpha'', \beta, \gamma ) \in \mathbb N^r \times \mathbb N^{n-r} \times \mathbb N^n\times \mathbb N^r$ there exists $C_{\alpha, \beta, \gamma}>0$
\begin{equation}\label{E: symbol model no L}
|\partial_{x'}^{\alpha'}\partial_{x''}^{\alpha''} \partial_\xi^{\beta} \partial_\lambda^{\gamma} \tilde{a}(x,\xi,\lambda)|\leq C_{\alpha, \beta, \gamma} \langle \lambda\rangle^{k-|\gamma|}.
\end{equation}

\begin{definition}
The symbols associated with this submanifold are as follows. We say $a\in \widetilde{S^{k}_{\Gamma_0,L_0,\rho}}$ if $a\in C^\infty(\re_x^n\times \re_\xi^{n}\times \re_\lambda^{r})$ {satisfies \eqref{E: symbol model L}}
and $a$ is supported in an $h$-independent set in $(x,\xi)$. {If we have the improved estimates} \eqref{E: symbol model no L}
then we say that $a\in \widetilde{S^k_{\Gamma_0,\rho}}$.
\end{definition}
\begin{remark}
{While there is no $\rho$ in the definition of $\widetilde{S^k_{\Gamma_0,\rho}}$, we keep it in the notation for consistency.}
\end{remark}

{Let $a\in {\widetilde{S^{k}_{\Gamma_0,L_0,\rho}}}$.} We then define 
$$
[\Op_{h}(a)] u(x):=\frac{1}{(2\pi h)^n}\int e^{\frac{i}{h}\langle x-y,\xi\rangle}a(x,\xi,h^{-\rho}x'){u(y)}{dy}d\xi.
$$
Since $a\in {\widetilde{S^{k}_{\Gamma_0,L_0,\rho}}}$ is compactly supported in $x$, {there exists $C>0$ such that} on the support of the integrand $\lambda \leq {C} h^{-\rho}$ and hence 
$
h\leq {C} h^{1-\rho}\langle \lambda\rangle^{-1}.
$
This will be important when computing certain asymptotic expansions.

\begin{lemma}\label{E: boundedness}
Let $k\in \re$ and $a\in \widetilde{S^k_{\Gamma_0,L_0,\rho}}$. Then,
$$
\|\Op_h(a)\|_{L^2\to L^2}\leq C \sup_{\re^{2n}}|a(x,\xi,{h^{-\rho}x'})|+O(h^{{-\rho}\max(k,0)+\frac{1-\rho}{2}})
$$
\end{lemma}
\begin{proof}
Define $T_\delta:L^2(\re^n)\to L^2(\re^n)$ by
\begin{equation}
\label{e:rescaling}
T_\delta u(x):=h^{\frac{n}{2}\delta}u(h^{\delta}x).
\end{equation}
Then $T_\delta$ is unitary and, for $a\in\widetilde{S^k_{\Gamma_0,L_0,\rho}}$,
$$
\Op_h(a)u=T_{\frac{1+\rho}{2}}^{-1}Op_1(a_h)T_{\frac{1+\rho}{2}}u, \qquad 
a_h{(x,\xi)}:=a(h^{\frac{1+\rho}{2}}x,h^{\frac{1-\rho}{2}}\xi,h^{\frac{1-\rho}{2}}x').
$$
Then, {for all $\alpha, \beta \in \mathbb N^n$ there exists $C_{\alpha, \beta}$ such that}
$
|\partial_x^\alpha \partial_\xi^\beta a_h|\leq {C_{\alpha, \beta}}h^{\frac{1-\rho}{2}(|\alpha|+|\beta|)}\langle h^{\frac{1-\rho}{2}}x'\rangle^{k-|\alpha|}.
$
Now, {since $a_h,b_h\in S_{\frac{1-\rho}{2}}$,} by~\cite[Theorem 4.23]{EZB} {there is a universal constant $M>0$ with}
$$
\|\Op_1(a_h)\|_{L^2\to L^2}\leq C\sum_{|\alpha|\leq Mn}\sup_{{\re^{2n}}}|\partial^\alpha a_h|\leq C\sup|a|+C_ah^{{-}\max({\rho} k,0)+\frac{1-\rho}{2}}.
$$
\vspace{-.8cm}

\end{proof}

\begin{lemma}
\label{l:compose}
Suppose that $a\in \widetilde{S^{k_1}_{\Gamma_0,L_0,\rho}}$, $b\in \widetilde{S^{k_2}_{\Gamma_0,L_0,\rho}}$. Then, 
$
\Op_h(a)\Op_h(b)=\Op_h(c)+O(h^\infty)_{L^2\to L^2}
$
where $c\in \widetilde{S^{k_1+k_2}_{\Gamma_0,L_0,\rho}}$ satisfies
\begin{equation}
\label{e:composed}
c=ab+O(h^{1-\rho})_{\widetilde{S^{k_1+k_2-1}_{\Gamma_0,L_0,\rho}}}.
\end{equation}
In particular, 
\begin{equation}
\label{e:fullAsymptotic}
c\sim \sum_{j}\sum_{|\alpha|=j}\frac{i^j}{j!}\big((hD_{x''})^{\alpha''}(hD_{x'}+h^{1-\rho}D_\lambda)^{{\alpha'}}b\big) D^\alpha_{\xi}a.
\end{equation}
If instead, $a\in \widetilde{S^{k_1}_{\Gamma_0,\rho}}$ and $b\in \widetilde{S^{k_2}_{\Gamma_0,\rho}}$, then the remainder in~\eqref{e:composed} lies in $h^{1-\rho}\widetilde{S^{k_1+k_2-1}_{\Gamma_0,\rho}}$.
\end{lemma}
\begin{proof}
With $T_\delta$ as in~\eqref{e:rescaling}, we have 
$
\Op_h(a)\Op_h(b)=T^{-1}_{\rho/2}Op_h(a_h)Op_h(b_h)T_{\rho/2}
$
where
$$
a_h=a(h^{\frac{\rho}{2}}x,h^{-\frac{\rho}{2}}\xi,h^{-\frac{\rho}{2}}x'),\qquad b_h=b(h^{\frac{\rho}{2}}x,h^{-\frac{\rho}{2}}\xi,h^{-\frac{\rho}{2}}x').
$$
Now, {for all $\alpha, \beta \in \mathbb N^n$ there exists $C_{\alpha, \beta}$ such that}
\begin{gather*}
|\partial_x^\alpha \partial_\xi^\beta a_h|\leq {C_{\alpha, \beta}}h^{-\frac{\rho}{2}(|\alpha|+|\beta|)}\langle h^{-\frac{\rho}{2}}x'\rangle^{k_1-|\alpha|},\quad
|\partial_x^\alpha \partial_\xi^\beta b_h|\leq {C_{\alpha, \beta}}h^{-\frac{\rho}{2}(|\alpha|+|\beta|)}\langle h^{-\frac{\rho}{2}}x'\rangle^{k_2-|\alpha|}.
\end{gather*}
In particular, using that $a$ and $b$ are compactly supported, $a_h\in h^{-\max(\rho k_1,0)}S_{\rho/2}$ and $b_h\in h^{-\max(\rho k_2,0)}S_{\rho/2}$
and hence~\cite[Theorems 4.14,4.17]{EZB} apply. 
 In particular, {if we let $M>0$ and $\tilde k:= \max(k_1,0)+\max(k_2,0)$, we obtain} 
$
Op_h(a_h)Op_h(b_h)=Op_h(c_h)
$
where, for any $N>0$,
\begin{align*}
&c_h(x,\xi)= \sum_{j=0}^{N-1}\sum_{|\alpha|=j}\frac{h^ji^j}{j!}(D_{\xi}^\alpha a_h(x,\xi))(D_x^\alpha b_h(x,\xi)) +O(h^{-\rho{\tilde k}+N(1-\rho)})_{S_{\rho/2}}\\
&=\sum_{j=0}^{N-1}\sum_{|\alpha|=j}\sum_{\alpha'+\alpha''=\alpha}\!\! \!\!\frac{h^{(1-\rho){j}}i^j}{j!}(D^\alpha_\xi a)_h[(h^{\rho}D_{x''})^{\alpha''}(h^{\rho}D_{x'}+D_\lambda)^{\alpha'}b]_h+O(h^{-\rho {\tilde k}+N(1-\rho)})_{S_{\rho/2}}.
\end{align*}
Choosing 
$
N=\max\Big(k_1+k_2,\frac{\rho\tilde k+M}{1-\rho}\Big),
$
the remainder is $O(h^{M})_{S_{\rho/2}}$. Moreover, since $a$ and $b$ were compactly supported, we may assume {introducing an $h^\infty$ error,} that the remainder is supported in $\{(x, \xi): \, |(x,\xi)|\leq Ch^{-\frac{\rho}{2}}\}.$
Putting
$$c=\sum_{j=0}^{N-1}\sum_{|\alpha|=j}\sum_{\alpha'+\alpha''=\alpha} \frac{i^j}{j!}(D^\alpha_\xi a)[(hD_{x''})^{\alpha''}(hD_{x'}+h^{1-\rho}D_\lambda)^{\alpha'}b],
$$
we thus have
$T_{\rho/2}^{-1}Op_h(c_h)T_{\rho/2}=\Op_h(c)+O(h^M)_{\mc{D}'\to C^\infty}$
as claimed.
\end{proof}

\begin{lemma}
\label{l:commutator}
Suppose that $a\in \widetilde{S^{m_1}_{\Gamma_0,L_0,\rho}}$, $b\in \widetilde{S^{m_2}_{\Gamma_0,L_0,\rho}}$. Then, 
$$
[\Op_h(a),\Op_h(b)]=-ih^{1-\rho}\Op_h(c)+O(h^\infty)_{L^2\to L^2}
$$
where $c\in \widetilde{S^{m_1+m_2-2}_{\Gamma_0,L_0,\rho}}$ satisfies
\begin{equation*}
\label{e:commutator1}
c=h^\rho\sum_{i=1}^n(\partial_{\xi_i}a\partial_{x_i}b-\partial_{\xi_i}b\partial_{x_i}a)+\sum_{i=1}^r(\partial_{\xi_i}a\partial_{\lambda_i}b-\partial_{\lambda_i}a\partial_{\xi_i}b)+O(h^{1-\rho})_{\widetilde{S^{m_1+m_2-2}_{\Gamma_0,L_0,\rho}}}.
\end{equation*}
If instead, $a\in \widetilde{S^{m_1}_{\Gamma_0,\rho}}$ and $b\in \widetilde{S^{m_2}_{\Gamma_0,\rho}}$, then the remainder lies in $h^{1-\rho}\widetilde{S^{m_1+m_2-2}_{\Gamma_0,\rho}}$.
Moreover, if $a\in S^{\comp}(\re^{2n})$ is independent of $\lambda$ and ${\partial_{\xi'}}a=e(x,\xi)x'$ {with $e(x,\xi):\R^r \to \R^r$ for all $(x,\xi)$,} then 
$$
[\Op_h(a),\Op_h(b)]=-ih\Op_h(c)+O(h^\infty)_{{\Psi^{-\infty}}}
$$
with 
$
c=H_ab+\sum_{i=1}^r(e\lambda)_i\partial_{\lambda_i}b+O(h^{1-\rho})_{\widetilde{S^{m_2-1}_{\Gamma_0,L_0,\rho}}}.
$
Similarly, the same conclusion holds if $b\in\widetilde{ S^{m_2}_{\Gamma_0,\rho}}$ with the error term in $c$ being $O(h^{1-\rho})_{\widetilde{S^{m_2-1}_{\Gamma_0,\rho}}}$.
\end{lemma}
\begin{proof}
In each case, we need only apply the formula~\eqref{e:fullAsymptotic}.  
\end{proof}

\subsection{Reduction to normal form}
{In order to define the quantization of symbols in $S_{\Gamma,L, \rho}$ for general $(\Gamma, L)$, we first explain how to reduce the problem to the model case $(\Gamma_0, L_0)$.} 

\begin{lemma}
Let $L$ be a Lagrangian foliation over a co-isotropic submanifold $\Gamma {\subset \R^{2n}}$ of dimension $2n-r$. Then for each $(x_0,\xi_0)\in \Gamma$ there is a neighborhood $U_0$ of $(x_0,\xi_0)$ and a symplectomorphism $\kappa:U_0\to V_0\subset T^*\re^n$ such that 
$$  \kappa(\Gamma\cap U_0)=\Gamma_0\cap V_0 \qquad \text{and}\qquad {(\kappa_*)_q L_q=L_{0,q} \;\;\text{for}\;\; q\in \Gamma \cap U_0}.$$
\end{lemma}
\begin{proof}
We first put $\Gamma$ in normal form. That is, we build symplectic coordinates {$(y, \eta)$} such that 
\begin{equation}
\label{e:coisoForm}
\Gamma=\{{(y, \eta):\,}y_1=\dots =y_r=0\}.
\end{equation}

First, assume $r=1$ and let $f_1\in C^\infty(T^*M)$ define $\Gamma$. By Darboux's theorem there are symplectic coordinates such that $y_1=f_1$ and the proof { of \eqref{e:coisoForm}} is complete for $r=1$. 

Next, assume that we can put any coisotropic of co-dimension $r-1$ in normal form. Let $f_1,\dots, f_r\in C^\infty(T^*M)$ define $\Gamma$. Then,{ since $\Gamma$ is co-isotropic, for $X\in T\Gamma$ and $i=1, \dots, r$}
$$
\sigma(X ,H_{f_i})=df_i(X) =0. 
$$
In addition, since $\Gamma$ is co-isotropic,  $(T\Gamma)^\perp\subset T\Gamma$ and so $H_{f_i}\in T\Gamma$ for all $i=1, \dots, r$. In particular, 
$$
\{f_i,f_j\}= H_{f_i}f_j=df_j(H_{f_i})=0,
$$
on $\Gamma$. Now, using Darboux' theorem, choose symplectic coordinates ${(y, \eta)=}(y_1,y',\eta_1,\eta')$ so that $y_1=f_1$ and $(x_0,\xi_0)\mapsto (0,0)$. Then, 
$\partial_{\eta_1} f_j=\{f_j,y_1\}=0$ on $\Gamma,$
for $j=2,\dots, r$.
Next, observe that  
$
\Gamma=\{{(y, \eta):\,}y_1=f_2=\dots ={f_{r}}=0\},
$
and $dy_1,\{df_j\}_{j=2}^{{r}}$ are independent. Thus,  {since $\partial_{\eta_1} f_j=0$ on $\Gamma$,
$$
\Gamma=\{{(y, \eta):\;}y_1=0, \;\; f_j(0,y',0,\eta')=0,\,\;\;j=2,\dots, {r}\}.
$$
}
Now, $\{y_1=\eta_1=0\}\cap \Gamma$ is a co-isotropic submanifold of co-dimension $r-1$ in $T^*\{y_1=0\}$. Hence, by induction, there are symplectic coordinates {$(y_2, \dots, y_n, \eta_2, \dots, \eta_n)$} on $T^*\{y_1=0\}$ such that 
$$
\Gamma\cap \{y_1=\eta_1=0\}=\{y_1=\eta_1=0, \;\; y_2=\dots =y_r=0\}.
$$
In particular, 
$$
{\{ (y', \eta'): \; f_j(0,y',0,\eta')=0,\,\;\;j=2,\dots, {r}\}}=\{y_2=\dots = y_r=0\}. 
$$
Thus, extending {$(y_2, \dots, y_n, \eta_2, \dots, \eta_n)$} to be independent of $(y_1,\eta_1)$ puts $\Gamma$ in the form~\eqref{e:coisoForm}.

Next, we adjust the coordinates  to be adapted to $L$ along $\Gamma$. First, define $\tilde{y}_i:=y_i$ for $i=1,\dots r$. {Then, since $L\subset T\Gamma$, for every $i=1, \dots, r$ we have that $d\tilde{y_i}(X)|_{\Gamma}$ is well defined for $X\in L$ and $d\tilde{y_i}(X)|_{\Gamma}=0$.} Next, since $L$ is integrable, the Frobenius theorem~\cite[Theorem 19.21]{LeeBook13} shows that there are coordinates  {$(\tilde y_{r+1}, \dots, \tilde y_n, \tilde \xi_1, \dots, \tilde \xi_n)$ on $\Gamma$, defined in a neighborhood of $(0,0)$,} such that $L$ is the annihilator of $d\tilde{y}$. 
{Since  we know that for every  $X\in L$
$$
\sigma(X,H_{\tilde{y}_i})={d\tilde y_i}(X)=0,
$$
and $L$ is Lagrangian, we conclude that $H_{\tilde{y}_i}\in L$.} In particular, since $L$ is the annihilator of $d\tilde{y}$, $$\{\tilde{y}_i,\tilde{y}_j\}=H_{\tilde{y}_i}\tilde{y}_j=d\tilde{y}_j(H_{\tilde{y}_i})=0.$$
Now, extend {$(\tilde y_{r+1}, \dots, \tilde y_n, \tilde \xi_1, \dots, \tilde \xi_n)$ outside $\Gamma$ to be independent of $(\tilde y_{1}, \dots, \tilde y_r)$}. Then, 
$
\{\tilde{y}_i,\tilde{y}_j\}=0
$
in a neighborhood of $(x_0,\xi_0)$ and hence, by Darboux's theorem, there are functions $\{\tilde{\eta}_j\}_{j=1}^n$, such that 
$\{\tilde{y}_i,\tilde{\eta}_j\}=\delta_{ij}$ and $\{\tilde{\eta}_i,\tilde{\eta}_j\}=0.$
In particular, in the $(\tilde{y},\tilde{\eta})$ coordinates, 
$
\Gamma=\{(\tilde y, \tilde \eta):\;\; \tilde y_1=\dots =\tilde y_r=0\},
$
and $d\tilde{y}(L)|_{\Gamma}=0$. In particular, $L=\operatorname{span} \{{\partial}\tilde{\eta}_i\}$ as claimed.
\end{proof}

In order to create a well-defined global calculus of psuedodifferential operators associated to $(\Gamma,L)$, we will need to show invariance under conjugation by FIOs preserving the pair $(L_0,\Gamma_0)$. 
\begin{proposition}
\label{p:invariance}
Suppose that $U_0,V_0$ are neighborhoods of $(0,0)$ in $T^*\re^n$ and $\kappa:U_0\to V_0$ is a symplectomorphism such that 
\begin{equation} \label{e:specialSymplectic}
\kappa(0,0)=(0,0),\qquad \kappa(\Gamma_0\cap U_0)=\Gamma_0\cap V_0,\qquad \kappa_*|_{\Gamma_0} L_0=L_0|_{\Gamma_0}.
\end{equation}
Next, let $T$ be a semiclassically elliptic FIO microlocally defined in a neighborhood of $$\big((0,0),(0,0)\big){\in T^*\re^n \times T^*\re^n}$$ quantizing $\kappa$. Then, for $a\in \widetilde{S^k_{\Gamma_0,L_0,\rho}}$, there {are $b\in \widetilde{S^k_{\Gamma_0,L_0,\rho}}$ and $c\in \widetilde{S^{k-1}_{\Gamma_0,L_0,\rho}}$} such that 
$$
T^{-1}\Op_h(a)T=\Op_h(b),\qquad b=a\circ K_{\kappa}+ h^{1-\rho}c
$$
where {$K_{\kappa}:T^*\R^n \times \R^r \to T^*\R^n \times \R^r$ is defined by}
{\begin{equation*}
K_{\kappa}(y,\eta,\mu)=\Big(\kappa(y,\eta), \pi_{x'}(\kappa(y, \eta)) \frac{|\mu|}{|y'|}\Big), 
\end{equation*}
and $\pi_{x'}: T^*\R^n \to \R^r$ is the projection onto the first $r$-spatial coordinates.} In addition, if $a\in \widetilde{S^k_{\Gamma_0,\rho}}$, then $c\in \widetilde{S^{k-1}_{\Gamma_0,\rho}}$ and $b\in \widetilde{S^k_{\Gamma_0,\rho}}$.
\end{proposition}
To prove Proposition~\ref{p:invariance}, we follow~\cite{SjZw:99}. First, observe that the {proposition} holds  with $\kappa=\Id$ since then {$T$} is a standard pseudodifferential operator. In addition, {the proposition also} holds {whenever for a given $j \in\{1, \dots, n\}$ we work with} 
$$
\kappa(y,\eta):=(y_1,\dots,y_{j-1},-y_j,y_{j+1},\dots,y_n\;,\;\eta_1,\dots,{\eta_{j-1}, -\eta_j},\eta_{j+1},\dots, \eta_n).
$$
{Indeed, this follows from the fact that in this case an FIO quantizing $\kappa$ is} $$Tu(x)=u(x_1,\dots x_{j-1},-x_j,x_{j+1},\dots, x_n)$$ {and so the conclusion of the proposition follows from a direct computation together with the identity case.}
Thus, we may assume that 
\begin{equation}
\label{e:positivity}
\kappa(y,\eta)=(x,\xi)\qquad\text{ implies }\qquad x_iy_i\geq 0,\qquad i=1,\dots n.
\end{equation}

\begin{lemma}\label{L: deformation of kappa}
Let $\kappa$ be a symplectomorphism satisfying~\eqref{e:specialSymplectic} and~\eqref{e:positivity}. Then, there is a piecewise smooth family of symplectomorphisms $[0,1]\ni t\mapsto\kappa_t$ such that $\kappa_t$ satisfies~\eqref{e:specialSymplectic}, \eqref{e:positivity}, $\kappa_0=\Id$, and $\kappa_1=\kappa$.
\end{lemma}
\begin{proof}
{In what follows we assume that $\kappa(y, \eta)= (x,\xi)$ but
reorder the coordinates $(y',y'',\eta',\eta'')\in T^*\re^n$ is represented as $(y',\eta',y'',\eta'')\in \R^{2r}\times \R^{2(n-r)}$. Let $\xi'$ and $\kappa''{=(x''(y',\eta),\xi''(y',\eta))}$ with}
$$
\kappa|_{\Gamma_0}:(0, {\eta'}, y'',{\eta''})\mapsto (0,\xi'(y'',\eta),\kappa''(y'',\eta)).
$$
{Now, since $(\kappa_*)|_{\Gamma_0}L_0=L_0$, we have for $i=1,\dots n$,
\begin{equation}\label{e:Tmapping}
\kappa_* \partial_{\eta_i}=\frac{\partial x_j}{\partial \eta_i}\partial_{x_j}+\frac{\partial \xi_j}{\partial\eta_i}\partial_{\xi_j}{\in} L_0
\end{equation}
and hence,
\begin{equation} \label{e:xindep}
\partial_{\eta}x|_{\Gamma_0}\equiv 0.
\end{equation}}
{Next, since $\kappa$ preserves $\Gamma_0$, $\{\kappa^*x_i\}_{i=1}^r$ defines $\Gamma_0$ and 
$
\Span\{d\kappa^*x_i|_{\Gamma_0}\}_{i=1}^r=\Span\{dy_i|_{\Gamma_0}\}_{i=1}^r,
$
$$
\Span\{H_{\kappa^*x_i}|_{\Gamma_0}\}_{i=1}^r=\Span\{H_{y_i}|_{\Gamma_0}\}_{i=1}^r.
$$
By Jacobi's theorem, $\kappa_*H_{\kappa^*x^i}=H_{x_i}$. Therefore,
$$(\kappa|_{\Gamma_0})_* \Big(\operatorname{span} \{H_{y_i}\}_{i=1}^r\big|_{\Gamma_0}\Big)=\operatorname{span} \{H_{x_i}\}_{i=1}^r\big|_{\Gamma_0},
$$
 and we conclude from~\eqref{e:Tmapping} $\xi''|_{\Gamma_0}$ is independent of $\eta'$ and hence that $\kappa''$ is independent of $\eta'$. In particular, $\kappa''$ is a symplectomorphism on $T^*\re^{n-r}$.}
This also implies that for each fixed $(y'',\eta'')$, the map $\eta'\mapsto \xi'(y'', {\eta', \eta''})$ is a diffeomorphism. Writing $$\kappa''(y'',\eta'')=(x''(y'',\eta''),\xi''(y'',\eta'')),$$
we have {by~\eqref{e:xindep} that} 
$\partial_{\eta''}x''=0$, and hence $x''=x''(y'')$. Now, since $\kappa''$ is symplectic, 
$$
(\partial_{\eta''}\xi'' d\eta''+\partial_{y''}\xi''dy'')\wedge \partial_{y''}x''dy''=d\eta''\wedge dy'',
$$
{and so we conclude that}  
\begin{equation}\label{E:1 and 2}
(\partial_{y''}x'')^t\partial_{\eta''}\xi''=\Id, \qquad  \qquad (\partial_{y''}x'')^t\partial_{y''}\xi'' \;\;\text{ is diagonal.}
\end{equation} 
The first equality in \eqref{E:1 and 2} {gives that $\partial_{\eta''}\xi''$ is a function of $y''$ only, and hence there exists a function $F=F(y'')$ such that} 
$\xi''(y'', \eta'')=[({\partial x''}(y''))^t]^{-1}(\eta''-F(y'')).$
Therefore, calculating on $\eta''=F(y'')$, the second {statement in \eqref{E:1 and 2}} implies that 
$
-\partial_{y''}F(y'')dy''\wedge dy''=0. 
$
In particular, 
$
d(F(y'')\cdot dy'')=0.
$
It follows from the Poincar\'e lemma that, shrinking the neighborhood of {$(0,0)$} to be simply connected if necessary, $F(y'')\cdot dy''=d\psi(y'')$ for {some function $\psi=\psi(y'')$}. Hence, 
\begin{equation} 
\label{e:GammaSymp}
\kappa''(y'',\eta'')=\Big(x''(y'')\;,\; [(dx''(y''))^t]^{-1}(\eta''-{\partial \psi}(y''))\Big).
\end{equation}
Now, every symplectomorphism of the form~\eqref{e:GammaSymp} preserves $L_0$. Hence, we can deform $\kappa''$ to the identity by putting $\psi_t=t\psi$ and deforming $x''$ to the identity. Since the assumption in~\eqref{e:positivity} implies $\partial_{y''}x''>0$, this can be done simply by taking 
$
x''_t=(1-t)\Id+t x''.
$
Putting $\kappa''_t=(x''_t,\xi''_t)$, there is $\kappa''_t$ such that $\kappa''_0=\Id$ and $\kappa''_1=\kappa''$. Now, composing $\kappa$ with 
$$
(y',\eta';y'',\eta'')\mapsto (y',\eta';(\kappa''_t)^{-1}(y'',\eta''))
$$
we reduce to the case that $\kappa''=\Id$. 
In particular, we need only consider the case in which
\begin{equation} 
\label{e:reducedKappa}  
\kappa(y',\eta',y'',\eta'')
=\Big(f(y,\eta)y'\;,\;\xi'(y'',\eta)+h_0(y,\eta)y' \;,\; (y'',\eta'')+h_1(y,\eta)y'\Big).
\end{equation}
where $f(y,\eta)\in \mathbb{GL}_r$, {$h_0(y, \eta)$ is an $r\times r$ matrix, and $h_1(y, \eta)$ is an $ 2(n-r)\times r$ matrix}. 
Next, we claim that {the projection map from $\operatorname{graph}(\kappa)$ to $\R^{2n}$ defined as }
$
 (x,\xi;y,\eta)\mapsto (x,\eta)
$
is a local diffeomorphism. To see this, note that for $|y'|$ small the map $(x'',\eta'')\mapsto (y'',\xi'')$ is a diffeomorphism,  that for each fixed $(y'',\eta'')$ the map $\eta'\mapsto \xi'$ is a diffeomorphism, and that $\det \partial_{y'}x'|_{\Gamma_0}\neq 0$. 
Thus, $\kappa$ has a generating function $\phi$: 
$
\kappa: ({\partial_\eta\phi}(x,\eta),\eta)\mapsto (x, {\partial_x \phi}(x,\eta)),
$
such that 
$$
\det {\partial_{x \eta}^2\phi}(0,0)\neq 0,\qquad \partial_{\eta'}\phi(0,x'',\eta)=0.$$

Now, {writing $\kappa=(\kappa', \kappa'')$, we have} $\kappa''=\Id$ at $x'=0$. Therefore, 
$$
\partial_{\eta''}\phi(0,x'',\eta)=x'',\qquad \partial_{x''}\phi(0,x'',\eta)=\eta''
$$
and we have $\phi(0,x'',\eta)=\langle x'',\eta''\rangle +C$ {for some $C\in \re$}. We may choose $C=0$ to obtain
\begin{equation}
\label{e:generate}
\phi(x,\eta)=\langle x'',\eta''\rangle +g(x,\eta)x',
\end{equation}
for some  $g:\R^{2n} \to\mathbb{M}_{1\times r} $.
Finally, since $\kappa(0,0)=(0,0)$ and $\partial_{x \eta}^2\phi$ is non-degenerate, we have $\partial_{x'}\phi(0,0)=g(0,0)=0$ and $\partial_{\eta'}g$ is non-degenerate. In fact \eqref{e:positivity} implies that as a quadratic form \begin{equation} 
\label{e:posGen}
\partial_{\eta'}g>0.
\end{equation}

Observe next that for every $\phi$, such that~\eqref{e:generate} holds for some $g$ satisfying~\eqref{e:posGen} and $g(0,0)=0$ generates a canonical transformation satisfying~\eqref{e:reducedKappa} and~\eqref{e:positivity}. In particular, the symplectomorphism satisfies~\eqref{e:specialSymplectic}. Thus, we can deform from the identity by putting $g_t=(1-t)\eta'+tg$. 
\end{proof}

Finally, we proceed with the proof of Proposition~\ref{p:invariance}.

\begin{proof}[Proof of Proposition~\ref{p:invariance}] 
Let $\kappa_t$ be {as in Lemma \ref{L: deformation of kappa}. That is,  a piecewise smooth deformation from $\kappa_0=\Id$ to $\kappa_1=\kappa$} such that $\kappa_t$ preserves $\Gamma_0$ and $(\kappa_t)_*|_{\Gamma_0}$ perserves $L_0$. Let $T_t$ be piecewise smooth family of elliptic FIOs defined microlocally near $(0,0)$, quantizing $\kappa_t$, and satisfying 
\begin{equation}\label{E:FIO T}
hD_tT_t+T_tQ_t=0,\qquad T_0=\Id.
\end{equation}
Here, $Q_t$ is a smooth family of pseudodifferential operator with symbol $q_t$ satisfying 
$
\partial_t\kappa_t=(\kappa_t)_*H_{q_t}.
$
(Such an FIO exists, for example, by~\cite[Chapter 10]{EZB} and $q_t$ exists by~\cite[Thoerems 11.3, 11.4]{EZB})
Next, define $$A_t:=T_t^{-1}\Op_h(a)T_t.$$

Note that 
$
T^{-1}\Op(a)T=T^{-1}T_1T_1^{-1}\Op(a)T_1T_1^{-1}T +O(h^\infty)_{\Psi^{-\infty}}.
$
Hence, since the Proposition follows by direct calculation when $\kappa=\Id$, we may assume that $T=T_1$.

{In that case, our goal is to find a symbol $b$ such  that $A_1=Op_h(b)$}. First, observe that \eqref{E:FIO T} implies that {$hD_tT_t^{-1}-Q_tT_t^{-1}=0$} and so
\begin{equation}\label{E: A_t}
hD_t A_t=[Q_t,A_t],\qquad A_0=\Op_h(a).
\end{equation}

We will construct $b_t\in \widetilde{S^k_{\Gamma_0,L_0,\rho}}$ such that $B_t:=\Op_h(b_t)$ satisfies
\begin{equation}\label{E: claim b_t}
hD_t B_t=[Q_t,B_t]+O(h^\infty)_{{\Psi^{-\infty}}},\qquad B_0=\Op_h(a).
\end{equation}
This would yield that $B_t-A_t=O(h^\infty)_{L^2\to L^2}$ and the argument would then be finished by setting $b=b_1$. Indeed, that $B_t-A_t=O(h^\infty)_{L^2\to L^2}$ would follow from the fact that {by \eqref{E: claim b_t}}
$$
{hD_t(T_t B_tT_t^{-1})}=O(h^\infty)_{{\Psi^{-\infty}}},
$$
and hence, since $T_0=\Id$ and $B_0=\Op_h(a)$, we have
$
{T_tB_tT_t^{-1}}-\Op_h(a)=O(h^\infty)_{\Psi^{-\infty}}.
$
Combining this with the fact that both $T_t$ and $T_t^{-1}$ are bounded on $H^k_h$ completes the proof.

To find $b_t$ as in \eqref{E: claim b_t}, note that since $\kappa_t$ preserves $\Gamma_0$ and $L_0$,  $\partial_t \kappa_t=H_{q_t}$, and $H_{q_t}$ is tangent to $L_0$ on $\Gamma_0$. Therefore, $\partial_{\eta'}q_t=0$ on $y'=0$ {and so there exists $r_t(y, \eta)$ such that $\partial_{\eta'}q_t(y, \eta)=r_t(y, \eta)y'$}. 
Hence, by Lemma~\ref{l:commutator} for any $b\in \widetilde{S^k_{\Gamma_0,L_0,\rho}}$ 
$$
[Q_t,\Op_h(b)]= -ih{\Op_h(f) \!+O(h^\infty)_{{\Psi^{-\infty}}}, \;\;\;f\!=\! H_{q_t}b+\sum_{j=1}^r(r_t\lambda)_j(\partial_\lambda b)_j +\!O(h^{1-\rho})_{\widetilde{S^{k-2}_{\Gamma_0,L_0,\rho}}}}\!.
$$
Then, letting  $b_t^0:=a\circ K_{\kappa_t} {\in \widetilde{S^k_{\Gamma_0,L_0,\rho}}}$ and $B_t^0=\Op_h(b_t^0)$ yields
$$
hD_t{B_t^0}={-ih\Op_h\big(H_{q_t}b_t^0+ (r_t \mu)\cdot\partial_\mu b_t^0 \big)}=[Q_t,B_t^0]+h^{2-\rho}\Op_h(e^0_t)
$$
where $e^0_t\in \widetilde{S^{k-2}_{\Gamma_0,L_0,\rho}}$. {This follows from the fact that if we set $\mu(y)=y' h^{-\rho}$, then 
$$\partial_t (b_t^0(y, \eta, \mu(y)))=H_{q_t}b_t^0(y, \eta, \mu(y))+ \partial_\mu b_t^0(y, \eta, \mu(y))H_{q_t}(\mu(y))$$
and $H_{q_t}\mu(y)= r_t(y, \eta)\mu(y)$. 

}
Iterating this procedure and solving away successive errors finishes the proof of Proposition~\ref{p:invariance}. If $a\in \widetilde{S^{k}_{\Gamma_0,\rho}}$, then we need only use that $\partial_{\xi'}q_t=r_tx'$ and we obtain the remaining results.
\end{proof}

Our next lemma follows~\cite[Lemma 4.1]{SjZw:99} and gives a characterization of our second microlocal calculus in terms of the action of an operator. In what follows, given {operators  $A$ and $B$, we define the operator $\ad_A$ by
$\ad_A B =[A,B].$}

\begin{lemma}[Beal's criteria]
\label{l:beal}
Let $A_h:\mc{S}(\re^n)\to \mc{S}'(\re^n)$ {and $k\in \mathbb Z$}.  Then, $A_h=\Op_h(a)$ for some $a\in \widetilde{S^{k}_{\Gamma_0,L_0,\rho}}$ if and only if  {for any $\alpha, \beta \in \mathbb R^n$ there exists $C>0$ with}
\begin{equation*}
\|\ad_{h^{-\rho}x}^\alpha  \ad_{hD_{x}}^{\beta}A_hu \|_{|\beta|-\min(k,0)}
\leq Ch^{(1-\rho)(|\alpha|+|\beta|)}\|u\|_{\max(k,0)}
\end{equation*}
where 
$
\|u\|_{r}:=\|u\|_{L^2}+\|h^{-\rho r}|x'|^ru\|_{L^2},
$
{for $r\geq 0$}.
Similarly, $A_h=Op_h(a)$ for some $a\in \widetilde{S^k_{\Gamma_0,\rho}}$ if and only if
\begin{equation*}
\|\!\ad_{h^{-\rho}x'}^{\alpha'} \ad_{x''}^{\alpha''}   \ad_{hD_{x'}}^{\beta'} \ad_{hD_{x''}}^{\beta''}\!\!A_hu \|_{|\beta'|-\min(k,0)}
\!\leq \!Ch^{(1-\rho)(|\alpha'|+|\beta'|)+|\alpha''|+|\beta''|}\|u\|_{\max(k,0)},
\end{equation*}
then $A_h=Op_h(a)$ for some $a\in \widetilde{S^k_{\Gamma_0,\rho}}.$
\end{lemma}
\begin{proof}
The fact that $A_h=\Op_h(a)$ for some $a\in \widetilde{S^k_{\Gamma_0,L_0,\rho}}$ implies the estimates above follows directly from the model calculus.
Let $U_h$ be the unitary (on $L^2$) operator,
$
U_hu(x)=h^{\frac{n}{2}}u(hx),
$
and note that 
$$
\|U_h^{-1} u\|_{r}=\|u\|_{L^2}+\|h^{(1-\rho)r}|x'|^r u\|_{L^2}.
$$

Then, consider 
$
\tilde{A}_h:=U_h A_h U_h^{-1}.
$
For fixed $h$, we can use Beal's criteria (see e.g.~\cite[Theorem 8.3]{EZB}) to see that there is $a_h$ such that
$
\tilde{A}_h=a_h(x,D).
$
Define $a$ such that 
$a(hx,\xi;h)=a_h(x,\xi)$
and hence, 
$
A_h=Op_h(a).
$
Note that for $\phi,\psi\in \mc{S}(\re^n)$,
\begin{equation}
\label{e:clownfish}
\langle \tilde{A}_h\psi, \phi\rangle =\frac{1}{(2\pi)^n} \iint e^{i\langle x,\xi\rangle }a_h(x,\xi)\hat{\psi}(\xi)\overline{\phi(x)}dxd\xi,
\end{equation}
where $\hat{\psi}(\xi)=(\mc{F}\psi)(\xi)=\int e^{-i\langle y,\xi\rangle }\psi(y)dy$.
Next, define $$B_h:=U_h\,{\ad_{h^{-\rho}x}^\alpha (\ad _{hD_x}^\beta (A_h))U_h^{-1}}.$$ 
{Since $D_x U_h= U_h hD_x$ and $U_h^{-1}D_x=hD_x U_h^{-1}$, we have}
\begin{align*}
B_h=\ad_{h^{1-\rho}x}^\alpha \ad _{D_x}^\beta \tilde{A}_h
=(-i)^{|\alpha|+|\beta|}h^{(1-\rho) |\alpha|}b_h(x,D),
\end{align*}
where
$
b_h(x,\xi)=(-\partial_{\xi})^\alpha \partial_x^\beta a_h(x,\xi).
$
{Our goal is then to understand the behavior of $b_h(x,\xi)$ in terms of $h$ and $\langle h^{1-\rho} x' \rangle$.} Let $\tau_{x_0}$ and $\hat{\tau}_{\xi_0}$ be the physical and frequency shift operators 
$$
\tau_{x_0}u(x)=u(x-x_0),\qquad \hat{\tau}_{\xi_0}u(x)=e^{i\langle x,\xi_0\rangle }u(x),
$$
with 
$\mc{F}\hat{\tau}_{\xi_0}=\tau_{\xi_0}\mc{F}$ and $ \mc{F}\tau_{x_0}=\hat{\tau}_{-x_0}.$
In addition, write $\|u\|_{(-r)}:=\|\langle h^{1-\rho}x'\rangle ^{-r}u\|_{L^2}$ for the dual norm to $\|u\|_{(r)}:=\|U_h^{-1}u \|_{r}$.

Assume that $k\geq 0$. Then, the definition of $B_h$ combined with the assumptions yield
\begin{equation*}
\label{e:aardvark}
|\langle B\tau_{x_0}\hat{\tau}_{\xi_0}\psi,\tau_{y_0}\hat\tau_{\eta_0}\phi\rangle| \leq h^{(1-\rho)(|\alpha|+|\beta|)} \|\tau_{x_0}\hat\tau_{\xi_0}\psi\|_{(k)}\|\tau_{y_0}\hat{\tau}_{\eta_0}\phi\|_{(-|\beta|)}.
\end{equation*}
In addition, note that for fixed $\psi,\phi\in \mc{S}$, 
$$
\|\tau_{x_0}\hat{\tau}_{\xi_0}\psi\|_{(k)}\sim \langle h^{1-\rho}(x_0)'\rangle^k,\qquad \|\tau_{y_0}\hat{\tau}_{\eta_0}\psi\|_{(-|\beta|)}\sim \langle h^{1-\rho}(y_0)'\rangle^{-|\beta|},
$$
Therefore, \eqref{e:aardvark} leads to
\begin{equation}
\label{e:aardvarkpro}
|\langle B\tau_{x_0}\hat{\tau}_{\xi_0}\psi,\tau_{y_0}\hat\tau_{\eta_0}\phi\rangle|\leq C h^{(1-\rho)(|\alpha|+|\beta|)}\langle h^{1-\rho}(x_0)'\rangle^{k}\langle h^{1-\rho}(y_0)'\rangle^{-|\beta|}.
\end{equation}

{On the other hand,} we have by~\eqref{e:clownfish} that 
\begin{align}
|\langle B\tau_{x_0}\hat{\tau}_{\xi_0}\psi,\tau_{y_0}\hat\tau_{\eta_0}\phi\rangle| 
&=\frac{h^{(1-\rho)|\alpha|}}{(2\pi)^n}\Big| \iint e^{i\langle x,\xi\rangle} b_h(x,\xi)\hat{\psi}(\xi-\xi_0)e^{-i\langle x_0,\xi-\xi_0\rangle -i\langle \eta_0,x-y_0\rangle }\bar{\phi}(x-y_0)dxd\xi\Big| \notag\\
&=h^{(1-\rho)|\alpha|}|\mc{F}((\tau_{y_0,\xi_0}\chi) b_h)(\eta_0-\xi_0,x_0-y_0)|, \label{e:aardvarkproLHS}
\end{align}
where 
$\chi(x,\xi)=e^{i\langle x,\xi\rangle }\hat{\psi}(\xi)\bar{\phi}(x).$
{Combining \eqref{e:aardvarkproLHS} with \eqref{e:aardvarkpro}} we then have
$$
|\mc{F}((\tau_{y_0,\xi_0}\chi) \partial_\xi^\alpha\partial_x^\beta a_h)(\eta_0-\xi_0,x_0-y_0)|\leq C h^{(1-\rho)|\beta|}\langle h^{1-\rho}(x_0)'\rangle^{k}\langle h^{1-\rho}(y_0)'\rangle^{-|\beta|}.
$$

Next, note that $\chi$ can be replaced by any fixed  function in $C_c^\infty$ by taking $\psi,\phi$ with $\hat{\psi}(\xi)\phi(x)\neq 0$ on $\supp \chi$. 
Putting $\zeta=\eta_0-\xi_0$ and $z=x_0-y_0$, we obtain {that for every $\tilde \alpha, \tilde \beta \in \mathbb N^n$}
$$
|\mc{F}(\partial_\xi^{\tilde{\alpha}}\partial_x^{\tilde{\beta}}(\tau_{y_0,\xi_0}\chi)\partial_\xi^\alpha \partial_x^\beta a_h)(\zeta,z)|\leq C h^{(1-\rho)|\beta|}\langle h^{1-\rho}(x_0)'\rangle^{k}\langle h^{1-\rho}(x_0-z)'\rangle^{-|\beta|}.
$$
Hence, 
$$
|z^{\tilde\alpha}\zeta^{\tilde{\beta}}\mc{F}((\tau_{y_0,\xi_0}\chi)\partial_{\xi}^\alpha\partial_x^\beta a_h)(\zeta,z)|\leq  C h^{(1-\rho)|\beta|}\langle h^{1-\rho}(x_0)'\rangle^{k}\langle h^{1-\rho}(x_0-z)'\rangle^{-|\beta|}.
$$
{In particular, for every $N>0$}
$$
|\mc{F}((\tau_{y_0,\xi_0}\chi)\partial_{\xi}^\alpha\partial_x^\beta a_h)(\zeta,z)|\leq  C h^{(1-\rho)|\beta|}\langle h^{1-\rho}(x_0)'\rangle^{k-|\beta|}\langle \zeta\rangle^{-N}\langle z\rangle^{-N},
$$
and, as a consequence, we obtain
$$
\partial_\xi^\alpha\partial_x^\beta a_h(x,\xi)=\partial_\xi^\alpha\partial_x^\beta(a(hx,\xi))=O(h^{(1-\rho)|\beta|}\langle h^{1-\rho}x'\rangle ^{k-|\beta|}).
$$
This gives the first claim of the lemma for $k\geq 0$. For $k\leq 0$, we consider $\langle h^{-\rho}x'\rangle^{-k}A$ and use the composition formulae.
A nearly identical argument yields the second claim.
\end{proof}

\subsection{Definition of the second microlocal class}

With Proposition~\ref{p:invariance} in place, { we are now in a position to define the class of operators with symbols in $S^{k}_{\Gamma,L,\rho}$. }

\begin{definition}
Let $\Gamma\subset U\subset T^*M$ be a co-isotropic submanifold, {$U$ an open set}, and $L$ a Lagrangian folation on $\Gamma$. A \emph{chart for $(\Gamma, L)$} is a symplectomorphism
$$
\kappa:U_0\to V,\qquad U_0\subset U,\qquad V\subset T^*\re^n,
$$
such that 
$\kappa(U_0\cap \Gamma)\subset V\cap \Gamma_0$ and ${\kappa_{*, q} L_q=(L_0)_{\kappa(q)}}$
 for $q\in \Gamma\cap U$.  
\end{definition}
 
 We now define the pseudodifferential operators associated to $(\Gamma,L)$.
\begin{definition}
Let $M$ be a smooth, compact manifold and $U\subset {T^*M}$ open, $\Gamma\subset U$ a co-isotropic submanifold, $L$ a Lagrangian foliation on $\Gamma$ and $\rho\in [0,1)$. We say that 
$
A:\mc{D}'(M)\to C_c^\infty(M)
$
is a \emph{semiclassical pseudodifferential operator with symbol class $S^{k}_{\Gamma,L,\rho}(U)$} (and write $A\in \Psi^k_{\Gamma,L,\rho}(U)$) {if  there are charts $\{\kappa_\ell\}_{\ell=1}^N$ for $(\Gamma, L)$ and symbols $\{a_\ell\}_{\ell=1}^N \subset \widetilde{S_{\Gamma,L,\rho}^k}(U)$ such that $A$ can be written in the form}
\begin{equation}
    \label{e:standardRep}
A=\sum_{\ell=1}^N T_\ell' \,\Op_h(a_\ell)\,T_\ell+O(h^\infty)_{\mc{D'}\to C^\infty}
\end{equation}
{where $T_\ell$ and $T_\ell'$ are FIOs quantizing $\kappa_\ell$ and $\kappa_\ell^{-1}$ for $\ell=1, \dots, N$.}

We say that $A$ is a \emph{semiclassical pseudodifferential operator with symbol class $S^{k}_{\Gamma,\rho}(U)$}, {and write $A\in \Psi^k_{\Gamma,\rho}(U)$}, if {there are symbols $\{a_\ell\}_{\ell=1}^N \subset \widetilde{S_{\Gamma,\rho}^k}(U)$ such that $A$ can be written in the form \eqref{e:standardRep}}.
\end{definition}

\begin{lemma}
\label{l:chartCompos}
Suppose that $\kappa:U\to T^*\re^n$ is a chart for $(\Gamma, L)$, $T$ quantizes $\kappa$, and $T'$ quantizes $\kappa^{-1}$.
{If $A\in \Psi^k_{\Gamma,L,\rho}(U)$, then} there is $a\in \widetilde{S^k_{\Gamma,L,\rho}}(U)$, with $\supp a(\cdot,\cdot,\lambda)\subset \kappa(U)$, such that  
$
TAT'=\Op_h(a)+O(h^\infty)_{\mc{D}'\to C^\infty}.
$
Moreover, if $A$ is given by~\eqref{e:standardRep}, then  
$$
a\circ K_\kappa =\sigma(T'T)\sum_{\ell=1}^N\sigma(T_\ell'T_\ell)\,(a_\ell\circ K_{\kappa_\ell})+O(h^{1-\rho})_{\widetilde{S^{k-1}_{\Gamma,L,\rho}}}.
$$
\end{lemma}
\begin{proof}
Note that we can write
$
T A T'=\sum_{\ell=1}^N T T_\ell' \Op_h(a_\ell)T_\ell{T'} +O(h^\infty)_{\mc{D'}\to C^\infty}.
$
Next, note that $T T_\ell'$ quantizes $\kappa\circ \kappa_\ell^{-1}$ and that $T_\ell T'$ quantizes $\kappa_\ell \circ \kappa^{-1}$. Letting $F_\ell$ be a microlocally unitary FIO quantizing $\kappa_\ell \circ\kappa^{-1}$,  ${F_\ell}$ satisfies the hypotheses of {Proposition~\ref{p:invariance}} and we can write 
$$
T_\ell T'= C\sub{\!L}{F_\ell},\qquad {T}T_\ell' =  {F_\ell}^{-1}C\sub{\!R}
$$
with $C\sub{\!L},C\sub{\!R}\in \Psi(M)$ satisfying
$
\sigma(C\sub{\!R}C\sub{\!L})=\big(\sigma(T_\ell' T_\ell)\circ \kappa_\ell^{-1}\big) \big(\sigma(T'T)\circ  \kappa_\ell^{-1}\big).
$
Therefore,  
$$
\begin{gathered}T T_\ell' \Op_h(a_\ell)T_\ell T'= {F_\ell}^{-1}C\sub{\!R} \Op_h(a_\ell)C\sub{\!L} {F_\ell}= Op_h({b_\ell})+(h^\infty)_{\mc{D'}\to C^\infty},\\
{b_\ell}=\big(\sigma(C\sub{\!R}C\sub{\!L})\circ \kappa_\ell \circ\kappa^{-1}\big) \big(a_\ell \circ K_{\kappa_\ell \circ\kappa^{-1}}\big)+O(h^{1-\rho})_{\widetilde{S^{k-1}_{\Gamma,L,\rho}}}.
\end{gathered}
$$
The lemma follows.
\end{proof}

\begin{lemma}{Let $\Gamma\subset U\subset T^*M$ be a co-isotropic submanifold, {$U$ an open set}, and $L$ a Lagrangian foliation on $\Gamma$.}
There is a principal symbol map 
$${\sigma\sub{\Gamma, L}}:\Psi^k_{\Gamma,L,\rho}{(U)}\to S^k_{\Gamma,L,\rho}{(U)}/h^{1-\rho}S^{k-1}_{\Gamma,L,\rho}{(U)}$$ 
such that for $A\in \Psi^{k_1}_{\Gamma,L,\rho}{(U)}$, $B\in \Psi^{k_2}_{\Gamma,L,\rho}{(U)}$,
\begin{equation} 
\label{e:symbols}
{\sigma\sub{\Gamma, L}}(AB)={\sigma\sub{\Gamma, L}}(A){\sigma\sub{\Gamma, L}}(B),\quad {\sigma\sub{\Gamma, L}}([A,B])=-ih\{{\sigma\sub{\Gamma, L}}(A),{\sigma\sub{\Gamma, L}}(B)\}.
\end{equation}
Furthermore, {the sequence}
$$
0\;\mapsto\; h^{1-\rho}\Psi^{k-1}_{\Gamma,L,\rho}{(U)}\overset{{\sigma\sub{\Gamma, L}}}{\longrightarrow}\; S^{k}_{\Gamma,L,\rho}{(U)}/h^{1-\rho}S^{k-1}_{\Gamma,L,\rho}{(U)}\;\to\; 0
$$
is exact. {The same holds with  $\sigma\sub{\Gamma}$, $\Psi_{\Gamma,\rho}$ and $S^k_{\Gamma, \rho}$.}
\end{lemma}
\begin{proof}
For $A$ as in~\eqref{e:standardRep}, we define 
$$
{\sigma\sub{\Gamma, L}}(A)=\sum_{\ell=1}^N\sigma(T_\ell T_\ell')(\tilde{a}_\ell\circ\kappa)
$$
where $\tilde{a}_\ell(x,\xi):=a_\ell(x,\xi,h^{-\rho}x')$. The fact that $\sigma$ is well defined then follows from Lemma~\ref{l:chartCompos}, and the formulae~\eqref{e:symbols} follow from Lemma~\ref{l:compose}.

To see that the sequence is exact, we only need to check that if $A\in \Psi^k_{\Gamma, L,\rho}$ and  $\sigma\sub{\Gamma, L}(A)=0$, then $A\in h^{1-\rho}\Psi^{k-1}_{\Gamma,L,\rho}.$ To do this, we may assume that ${\WFh'}(A)\subset U$ such that there is a chart $(\kappa, U)$ for $(\Gamma,L)$. Let $T$ {be a  microlocally unitary FIO quantizing $\kappa$} and suppose that $\sigma\sub{\Gamma,L}(A)\in h^{1-\rho}S_{\Gamma,L,\rho}^{k-1}$. 
Then, by the first part of Lemma~\ref{l:chartCompos} we know
$
TAT^{-1}= \Op_h(a)+O(h^\infty)
$
for some $a\in \widetilde{S^{k}_{\Gamma,L,\rho}}$. Then, by the second part of Lemma~\ref{l:chartCompos}, since $\sigma\sub{\Gamma, L}(A)\in h^{1-\rho}S_{\Gamma,L,\rho}^{k-1}$, {$a\in h^{1-\rho}{\widetilde{S^{k-1}_{\Gamma,L,\rho}}}$} and in particular, $A\in h^{1-\rho}\Psi^{k-1}_{\Gamma,L,\rho}.$
\end{proof}

Note that if $A\in \Psi^{\comp}(M)$, then $A\in \Psi^0_{\Gamma,L,\rho}$ and 
$
\sigma(A)=\sigma\sub{{\Gamma}}(A).
$
Furthermore, if $A\in \Psi^k_{\Gamma,\rho}$, then $A\in \Psi^k_{\Gamma,L,\rho}$ and 
$
\sigma\sub{{\Gamma}}(A)=\sigma\sub{{\Gamma,L}}(A).
$

\begin{lemma}Let $\Gamma\subset U\subset T^*M$ be a co-isotropic submanifold, {$U$ an open set}, and $L$ a Lagrangian foliation on $\Gamma$.
There is a non-canonical quantization procedure 
$$ Op_h^{\Gamma,L}:S^k_{\Gamma,L,\rho}{(U)}\to \Psi^k_{\Gamma,L,\rho}{(U)}$$ 
such that for all $A\in \Psi^{k}_{\Gamma,L,\rho}{(U)}$ there is $a\in S^{k}_{\Gamma, L,\rho}{(U)}$ such that
$
Op_h^{\Gamma, L}(a)=A+O(h^\infty)_{\mc{D}'\to C^\infty}.
$
and ${\sigma\sub{\Gamma,L}}\circ Op_h^{\Gamma,L}:S^k_{\Gamma,L,\rho}{(U)}\to S^k_{\Gamma,L,\rho}{(U)}/h^{1-\rho}S^{k-1}_{\Gamma,L,\rho}{(U)}$ is the natural projection map.
\end{lemma}
\begin{proof}
{Let $\{(\kappa_\ell, U_\ell)\}_{\ell=1}^N$ be charts for $(\Gamma, L)$ such that $\{U_\ell\}_{\ell=1}^N$ is a locally finite cover for $U$,  $T_\ell$ and $T_\ell'$ quantize respectively $\kappa_\ell$ and $\kappa_\ell^{-1}$, and  $\sigma(T_\ell'T_\ell)\in C_c^\infty(U_\ell)$ is a partition of unity on $U$.}
{Let $a\in S^{k}_{\Gamma, L,\rho}{(U)}$.} Then, define $a_\ell \in \widetilde{S^k_{\Gamma_0,L_0,\rho}}$ such that
$
a_\ell(x,\xi,h^{-\rho}x'):= (\chi_\ell a)\circ \kappa^{-1}{(x,\xi)}
$
where $\chi_\ell\equiv 1$ on $\supp \sigma(T_\ell'T_\ell)$. 
 We then define the quantization map
$$
Op_h^{\Gamma,L}(a):=\sum_{\ell=1}^NT_\ell' \Op_h(a_\ell)T_\ell.
$$
The fact that ${\sigma\sub{\Gamma,L}}\circ Op_h^{\Gamma,L}$ is the natural projection follows immediately. Now, fix $A\in \Psi^{k}_{\Gamma, L,\rho}(U)$. Put $a_0={\sigma\sub{\Gamma,L}}(A)$. Then, 
$
A={Op_h^{\Gamma,L}}(a_0)+h^{1-\rho}A_1
$
where $A_1\in \Psi^{k-1}_{\Gamma,L,\rho}.$ We define  $a_k={\sigma\sub{\Gamma,L}}(A_k)$ inductively for $k\geq 1$ by  
$$
h^{(k+1)(1-\rho)}A_{k+1}=A-\sum_{k=0}^k h^{k(1-\rho)}Op_h^{\Gamma,L}(a_k).
$$
Then, letting $a\sim \sum_k h^{k(1-\rho)}a_k$, we have 
$
A=Op_h^{\Gamma,L}(a)+O(h^\infty)_{\mc{D}'\to C^\infty}
$
as claimed.
\end{proof}

\begin{remark}
{Note that $E:=\sum_{\ell=1}^NT_\ell T_\ell'$ is an elliptic pseudodifferential operator with symbol $1$. Therefore, there is $E'\in \Psi^0$ with $\sigma(E')=1$ such that $E' EE'=\Id$. Replacing $T_\ell$ by $E'T_\ell$ and $T_\ell'$ by $T_\ell'E'$,  we may (and will) ask for $\sum_{\ell=1}^N T_\ell T_\ell'=\Id$, and so ${Op_h^{\Gamma,L}}(1)=\Id$.}
\end{remark}

\begin{lemma}
{Let $\Gamma\subset U\subset T^*M$ be a co-isotropic submanifold.}
If $A\in \Psi^k_{\Gamma,\rho}(U)$ and $P\in \Psi^m(U)$ with symbol $p$ such that for every $q\in \Gamma$ we have $H_p(q)\in T_q\Gamma.$ Then, 
$$
\frac{i}{h}[P,A]=Op_h^{\Gamma}(H_p  a)+O(h^{1-\rho})_{\Psi^{k-1}_{\Gamma,\rho}},
$$
where $a(x,\xi;h)= \sigma\sub\Gamma(A)(x,\xi, h^{-\rho}x')$.
\end{lemma}
\begin{proof}
Suppose that {$WF_h'(A) \subset U_\ell$ for $U_\ell \subset U$ open, and suppose that $\kappa:U_\ell\to T^*\re^n$ is a chart for $(\Gamma, L)$. Note that
we may assume that $\WFh(A)' \subset U_\ell$ and then use a partition of unity to cover $U$ with a family $\{U_\ell\}_\ell$}. Therefore, there exist $ a\in \widetilde{S^k_{\Gamma,\rho}}$ {and a Fourier integral operator $T$ that is microlocally elliptic on {$U_\ell$} and quantizes $\kappa$}, such that 
$
A=T^{-1} \Op_h(a)T +O(h^\infty)_{\mc{D'}\to C^\infty}.
$
Then, on $\WFh'(A)$, 
$$
T[P,A]T^{-1}=[T PT^{-1},\Op_h( a)]+O(h^\infty)_{\mc{D}'\to C^\infty}. 
$$
Now, 
$
T PT^{-1}=Op_h(p\circ \kappa^{-1})+O(h)_{{\Psi^{m-1}}}.
$
Hence, {a direct computation using Lemma~\ref{l:commutator} gives}
$$
[TPT^{-1}, \Op_h( a)]=-ih\Op_h(c)+O(h^{2-\rho})_{\widetilde{\Psi^{k-2}_{\Gamma_0,\rho}}},
$$
with $c(x,\xi,h^{-\rho}x')=H_{p\circ \kappa^{-1}}{( a(x,\xi,h^{-\rho}x'))} {\in S^{k-1}\sub{\Gamma, \rho}(U_\ell)}.$
In particular, 
$$
[P,A]=-ihT^{-1}\Op_h(c)T+O(h^{2-\rho})_{\Psi^{k-2}_{\Gamma,\rho}}
$$
Therefore, 
$
[P,A]\in h\Psi^{k-1}_{\Gamma,\rho}.
$
with symbol $\sigma\sub{\Gamma}(ih^{-1}[P,A])=H_p {( a(x,\xi,h^{-\rho}x'))}$. 
\end{proof}

\section{An Uncertainty principle for co-isotropic localizers}
\label{s:uncertainMe}

The {first goal of this section is to build a family of cut-off operators $X_y$ with $y \in M$ that act as the identity on the shrinking ball $B(y, h^\rho)$ and such that they commute with $P$ in a fixed size neighborhood of $y$. This is the content of section \ref{S:co-isotrop}. The second goal is to control $\|X_{y_1} X_{y_2}\|_{L^2 \to L^2}$ in terms of the distance $d(y_1, y_2)$, as this distance shrinks to $0$. We do this in Section \ref{S:uncertainty}. Finally, in Section~\ref{s:almostOrthog}, we study the consequences of these estimates for the almost orthogonality of $X_{y_i}$.}

{In order to localize to the ball $B(y, h^\rho)$ in a way compatible with microlocalization we need to make sense of $$\chi_y(x)=\tilde{\chi}\big(\tfrac{1}{\e}h^{-\rho}d(x,y)\big)\qquad \tilde{\chi}\in C_c^\infty((-1,1)),$$ as an operator in some anisotropic pseudodifferential calculus. 
As a function, $\chi_y$ is in the symbol class $S^{-\infty}_{\Gamma_y, L_y}$, where $\Gamma_y, L_y$ are the co-isotropic submanifold and Lagrangian foliation defined as follows:  } Fix $\delta>0$, to be chosen small later, and {for each $x\in M$ let}
\begin{equation}
    \label{e:coiso}
\Gamma_y:=\bigcup_{|t|<\frac{1}{2}\inj(M)} \varphi_t(\Omega_y), \qquad \Omega_y:=\big\{{ \xi\in T^*_yM:\; \;\big|1-|\xi|_g\big|<\delta}\big\}.
\end{equation}

In this section, we construct localizers to $\Gamma_{y}$ adapted to the Laplacian and study the incompatibility between localization to $\Gamma_{y_1}$ and $\Gamma_{y_2}$ as a function of the distance between $y_1, y_2 \in M$.
{Let $y \in M$. In what follows we work with the Lagrangian foliation $L_y$ of $\Gamma_y$ given by 
$$
L_y=\{L_{y, \tilde q}\}_{\tilde q \in \Gamma_y}, \qquad 
L_{y, \tilde q}= (\varphi_t)_*(T_qT^*_yM),
$$
where $\tilde q= \varphi_t(q)$ for some $|t|<\frac{1}{2}\inj(M)$ and $q \in \Omega_y$.
}
{\begin{remark}
In fact, it will be enough for us to show that $\chi_y(x) \tilde{\chi}( \delta^{-1}(|hD|_g-1))\in \Psi_{\Gamma_y,L_y,\rho}$ since we will be working near the characteristic variety for the Laplacian.
\end{remark}
}

\subsection{Co-isotropic cutoffs adapted to the Laplacian}\label{S:co-isotrop}

\begin{lemma}\label{l:chi_h,y}
Let $y\in M$, $0<\e<\delta$, $0\leq \rho <1$, $\tilde{\chi}\in C_c^\infty((-1,1))$, and define the operator $\chi\sub{h,y}$ by
\begin{equation}\label{e:chi_y}
\chi\sub{h,y}u(x):=\tilde{\chi}(\tfrac{1}{\e}h^{-\rho}d(x,y))\;[Op_h(\tilde\chi(\tfrac{1}{\e}(|\xi|_g-1)))u](x).
\end{equation}
Then, $\chi\sub{h,y} \in \Psi_{\Gamma_{y},{L_{y}} ,\rho}^{-\infty}$.
\end{lemma}
\begin{proof}

We will use Lemma~\ref{l:beal} to prove the claim. First, observe that we may work in a single chart for $(\Gamma_y,L_y)$ by using a partition of unity. Therefore, suppose that $B\in \Psi^0$ and $\kappa: U_0\to T^*\mathbb{R}^n$ is a chart for $(\Gamma_y,L_y)$, $V_0\Subset U_0$, and $T$ is an FIO quantizing $\kappa$ that is microlocally unitary on $V_0$. Furthermore, since $\kappa_*L_y=L_0$, we may assume that $\kappa(U_0\cap T^*_yM)\subset T^*_{0}\mathbb{R}^n$. Denote the microlocal inverse of $T$ by $T'$. Then, observe that for $A$ and $B$ with wavefront set in $V_0$
$$
{\ad_A (TB T')=T \ad_{T'AT}(B)T'} +O(h^\infty)_{\mc{D'}\to C^\infty}.
$$

By a partition of unity, we will work as though $\chi\sub{h,y}$ were microsupported in $U_0$. We then {consider for all $N>0$, and $\alpha, \beta \in \mathbb N^n$,}
\begin{multline}
h^{-2N\rho}|x'|^{2N} \ad_{h^{-\rho}x}^\alpha \ad_{hD_{x}}^\beta {(T\chi\sub{h,y}T')} \notag\\
=h^{-2\rho N}T (T'|x'|^2T)^N \ad_{h^{-\rho}T' xT}^\alpha { (\ad_{T'hD_xT}^\beta {(\chi\sub{h,y})} )}T' +O(h^\infty)_{\mc{D}'\to C^\infty}. \label{E:ads}
\end{multline}

In order to prove the requisite estimates, we will actually view $\chi\sub{h,y}$ first as an element of the model microlocal class. In particular, {we work with $x\in M$ written in  geodesic normal coordinates centered at $y$, }so that
$$
\chi\sub{h,y}u(x)={\tilde{\chi}(\tfrac{1}{\e}h^{-\rho}{|x|})\;[Op_h(\tilde\chi(\tfrac{1}{\e}(|\xi|_g-1)))u](x)}.
$$
Then, $\chi\sub{h,y} {=\Op_h(\tfrac{1}{\e}\tilde \chi(\lambda)) \, Op_h(\tilde\chi(\tfrac{1}{\e}(|\xi|-1)))}$ is  an element of $\widetilde{\Psi^{-\infty}_{\Gamma_0,L_0,\rho}}$ {with $r=n$}, and so we can apply Lemma~\ref{l:commutator} to compute $\ad_A(\chi\sub{h,y})$ for $A\in \Psi^{-\infty}(M)$. In particular, 
\begin{equation}
\ad_{{T'hD_xT}}(\chi\sub{h,y})=\Op_h(c)+O(h^\infty) \label{E:Ads1}
\end{equation}
where $c\in h^{1-\rho}\widetilde{S^{-\infty}_{\Gamma_0,L_0,\rho}}$ is supported on {$\{(x, \xi, \lambda): \, |x|\leq \e h^\rho, \, |\lambda|\leq \e\}$}. Now, suppose $c\in\widetilde{S^{-\infty}_{\Gamma_0,L_0,\rho}}$ is supported on {$\{(x, \xi, \lambda): \, |x|\leq \e h^\rho, \, |\lambda|\leq \e\}$} and $B\in \Psi^{-\infty}$ with $\sigma(B)(0,\xi)=0$. Then, again using Lemma~\ref{l:commutator}, 
\begin{equation}
\ad_B (\Op_h(c)) =\Op_h(c')+O(h^\infty)\label{E:Ads2}
\end{equation}
where $c'\in h\widetilde{S^{-\infty}_{\Gamma_0,L_0,\rho}}$ is supported on {$\{(x, \xi, \lambda): \, |x|\leq \e h^\rho, \, |\lambda|\leq \e\}$}.

Now, note that since $\kappa (T^*_yM)\subset T^*_0\mathbb{R}^n$, 
then for all $i=1,\dots n$, $B=T'x_iT$ has symbol $\sigma(B)=[b(x,\xi)x]_i$ for some $b \in C^\infty(T^*M; {\mathbb{M}_{n\times n}})$.  Therefore,   \eqref{E:Ads1} and \eqref{E:Ads2} yield
$$
\ad_{h^{-\rho}T' xT}^\alpha (\ad_{T'hD_xT}^\beta (\chi\sub{h,y})) =h^{(1-\rho)(|\alpha|+|\beta|)}\Op_h(c')+O(h^\infty),
$$
where $c'\in \widetilde{S^{-\infty}_{\Gamma_0,L_0,\rho}}$ is supported on {$\{(x, \xi, \lambda): \, |x|\leq \e h^\rho, \, |\lambda|\leq \e\}$}. Finally, using again that $T'{x_i}T$ has symbol {$[b(x,\xi)x]_i$}, we have that \eqref{E:Ads2} gives
$$
\|h^{-2N\rho}|x'|^{2N}\ad_{h^{-\rho}T' xT}^\alpha (\ad_{T'hD_xT}^\beta (\chi\sub{h,y})) \|_{L^2\to L^2}\leq C h^{(1-\rho)(|\alpha|+|\beta|)}.
$$
\end{proof}

We next construct a pseudodifferential cutoff, $X_y\in \Psi^{-\infty}_{\Gamma_y,\rho}$ which is microlocally the identity near $S^*_yM$ and which essentially commutes with $P$ near $y$. In particular, we will have
$$\chi\sub{h,y}X_y=\chi\sub{h,y}+O(h^\infty)_{\Psi^{-\infty}}.$$
When considering the value of a quasimode $u$, that is $h^\rho$ close to the point $y$, this will allow us to effectively work with $X_yu$ instead. 
\begin{theorem}
\label{l:nice2ndCut}
Let $y\in M$, $0<\e<\delta$, $0\leq \rho <1$. 
{Then, there exists
$
X_{y}\in \Psi_{\Gamma_y,\rho}^{-\infty} \subset \Psi_{\Gamma_{y},{L_{y}} ,\rho}^{-\infty}
$
with 
\begin{enumerate}
    \item if $\chi\sub{h,y}$ is defined as in \eqref{e:chi_y}, then 
    \begin{equation}
    \label{e:2MicrolocalIdentity}
    \chi\sub{h,y}X_y=\chi\sub{h,y} + O(h^\infty)_{\Psi^{-\infty}}.
    \end{equation}
    \item 
    ${\WFh'}([P, X_{y}]) \cap \{(x,\xi): x \in B(y, \tfrac{1}{2}\conj(M)), \;\xi \in \Omega_x\} = \emptyset.$
\end{enumerate}
}
\end{theorem}

\begin{proof}
{First, we note that we will actually prove that $X_{y}\in  \Psi_{\Gamma_y,\rho}^{-\infty}$, and so the result will follow since $\Psi_{\Gamma_{y} ,\rho}^{-\infty} \subset \Psi_{\Gamma_{y},{L_{y}} ,\rho}^{-\infty}$.}
Let $\mc{H}\subset T^*M$ be transverse to {the Hamiltonian flow} $H_p$ such that $\Omega_y\subset \mc{H}$. Next, let $\varkappa \in C_c^\infty((-2,2))$ with $\varkappa \equiv 1$ on $[-1,1]$ and define ${\varkappa_0\in C_c^\infty(\mc{H})}$, by
$$
\varkappa_0={\varkappa}(h^{-\rho}d(x,y)) \varkappa(\tfrac{2}{\delta}(1-|\xi|_g)),
$$
{where $\delta$ is as in the definition of $\Omega_y$.} Let $\psi\in C_c^\infty(T^*M)$ with 
$$\psi \equiv 1\text{ on }B(y,\tfrac{1}{2}\conj(M))\cap \{|\xi|_g<2\},\qquad \supp \psi\subset B(y,\tfrac{3}{4}\conj(M)).$$
Then, let $\chi_0$ be defined locally by
$
H_p\chi_0\equiv 0$ and $\chi_0|_{\mc{H}}=\varkappa_0.
$
so that $\chi_0\in { S^{-\infty}_{\Gamma_y,\rho}}$.
That is,
$
\chi_0(\varphi_t(q))=\psi(\varphi_t(q))\chi_0(q)
$
 for $|t|<\conj(M)$ and $q\in \mc{H}$.
Next, observe that there is {$e_0\in S^{-\infty}_{\Gamma_y,\rho}$} such that
$$
-\tfrac{i}{h}[P,Op_h^{\Gamma_y}(\chi_0)]=h^{1-\rho}Op_h^{\Gamma_y}(e_0),\qquad
\supp e_0 \cap B(y,\tfrac{1}{2}\conj(M)) \subset \bigcup_{|t|< \tfrac{3}{4}\conj(M)} \varphi_t(\mc{H}\cap \supp \partial \varkappa_0).
$$
Suppose that {there exist  $\chi_{k-1}, e_{k-1}\in{ S^{-\infty}_{\Gamma_y,\rho}}$ such that}
$$
-\tfrac{i}{h}[P,Op_h^{\Gamma_y}(\chi_{k-1})]=h^{k(1-\rho)}Op_h(e_{k-1}),\qquad
\supp e_{k-1} \cap B(y,\tfrac{1}{2}\conj(M)) \subset  \hspace{-.4cm} \bigcup_{|t|< \tfrac{3}{4}\conj(M)}\hspace{-.5cm} \varphi_t(\mc{H}\cap \supp \partial \varkappa_0).
$$
 Then,  define $\tilde{\chi}_k\in{ S^{-\infty}_{\Gamma_y,\rho}}$ by solving locally
$
H_p\tilde{\chi}_k=e_{k-1}$ and  $\tilde{\chi}_{k}|_{\mc{H}}=0.$
Note that then 
{
$$
\supp \tilde{\chi}_k \cap B(y,\tfrac{1}{2}\conj(M)) \subset \bigcup_{|t|< \tfrac{3}{4}\conj(M)} \varphi_t(\mc{H}\cap \supp \partial \varkappa_0)
$$
and}
$$
h^{-k(1-\rho)}\sigma\Big(\tfrac{i}{h}\Big[P,Op_h^{\Gamma_y}\big(\chi_{k-1}+h^{k(1-\rho)}\tilde{\chi}_k\big)\Big]\Big)=H_p\tilde{\chi}_k-e_{k-1}=0.
$$
In particular, with $\chi_k:=\chi_{k-1}+h^{k(1-\rho)}\tilde{\chi}_k$, we obtain
$
-\tfrac{i}{h}[P,Op_h^{\Gamma_y}(\chi_{k})]=h^{(k+1)(1-\rho)}Op_h(e_{k})
$
with $e_k\in S^{-\infty}_{\Gamma_y,\rho}$ and
{
$$
\supp e_{k} \cap B(y,\tfrac{1}{2}\conj(M)) \subset \bigcup_{|t|< \tfrac{3}{4}\conj(M)} \varphi_t(\mc{H}\cap \supp \partial \varkappa_0).
$$} 
Setting
$$
{X_y=Op_h^{\Gamma_y}(\chi_\infty)}, \qquad \chi_\infty\sim \big(\chi_0 +\sum_k \chi_{k+1}-\chi_k)\big),
$$
we have that $X_y$ satisfies the second claim and, moreover, $\chi_\infty\equiv 1$ on 
$$
{\bigcup_{|t|\leq \frac{1}{4}\conj(M)}\varphi_t\Big(\mc{H}\cap \{d(x,y)<h^\rho\}\cap \big\{\big||\xi|_g-1\big|<\tfrac{\delta}{2}\big\}\Big).}
$$

{To see the first claim, observe that for $\e>0$ small enough,
$$
B(y,\e h^\rho)\cap \big\{\big||\xi|_g-1\big|<\delta\}\subset {\bigcup_{|t|\leq \frac{1}{4}\conj(M)}\varphi_t\Big(\mc{H}\cap \{d(x,y)< h^\rho\}\cap \big\{\big||\xi|_g-1\big|<\tfrac{\delta}{2}\big\}\Big).}
$$
and hence 
$
\chi\sub{h,y}X_y=\chi\sub{h,y}Op_{h}^{\Gamma,L}(1)+O(h^\infty)_{\Psi^{-\infty}}=\chi\sub{h,y}+O(h^\infty)_{\Psi^{-\infty}}.
$}
\end{proof}

\subsection{An uncertainty principle for co-isotropic localizers}\label{S:uncertainty}

Let $\Gamma(t)\subset T^*\re^n$, $t\in {(-\e_0,\e_0)}$ be a smooth family of co-isotropic submanifolds of dimension $n+1$ with 
$$
{\Gamma(0)}=\{(0,x_n,\xi',\xi_n):\;{ x_n}\in \re,\,\xi'\in \re^{n-1}, \,\xi_n \in \R\}.
$$
Assume that $\Gamma(t)$ is defined by $\{q_i(t)\}_{i=1}^{n-1}$ with $q_i(0)=x_i$. Moreover, assume that there are $c,C>0$ such that for $i=1,\dots,n-1$,
\begin{equation}
\label{e:uncertainAssume1}
|\{q_i(t),x_i\}|\geq c|t|\qquad\text{ on }\;\Gamma(0)\cap \Gamma(t),\qquad |t|>0,
\end{equation}
and for all $i,j=1,\dots, n-1$, {and all $t\in {(-\e_0,\e_0)}$,}
\begin{equation} 
\label{e:uncertainAssume2}
\{q_i(t),q_j(t)\}=0,\qquad \{q_i(t),\xi_n\}=0,\qquad |\{q_i(t),x_j\}|\leq Ct^2,\; i\neq j.
\end{equation}

The main goal of this section is to prove the following proposition
\begin{proposition}
\label{p:uncertainty}
Let $0<\rho<1$ {and $\{\Gamma(t):\; t\in (-\e_0,\e_0)\}$ be as above.} Suppose that $X(t)\in \Psi_{\Gamma(t),\rho}^{-\infty}$ {for all $t\in (-\e_0,\e_0)$}, and that there is $\e>0$ such that $h^{\rho-\e}\leq {|t|<\e_0}$. Then, 
$$
\|X(0)X(t)\|_{L^2\to L^2}\leq Ch^{\frac{n-1}{2}(2\rho-1)}t^{\frac{1-n}{2}}.
$$
\end{proposition}

\begin{proof}

We begin by finding a convenient chart for $\Gamma(t)$. By Darboux's theorem, there is a smooth family of sympletomorphisms $\kappa_t: V_1 \to V_2$ such that for $j=1,\dots,n-1$,
{\begin{equation}\label{E: kappa}
\kappa_t^*(q_j(t))=y_j,\qquad \kappa_t^*\xi_n=\eta_n,
\end{equation}
where $V_1, V_2$ are simply connected neighborhoods of $0$}. {Note that $\kappa_t(\Gamma(0))=\Gamma(t)$ with this setup, so $\kappa_t^{-1}$ is a chart for $\Gamma(t)$}.
By \cite[Theorem 11.4]{EZB}, the symplectomorphism $\kappa_t$ can be extended to a family of symplectomorphisms on $T^*\re^n$ that is the identity outside a compact set, and such that there is a smooth family of symbols ${p_t}\in C^\infty(T^*\re^{n})$ satisfying 
$
\partial_t\kappa_t=(\kappa_t)_*H_{p_t}.
$

Now, let $U(t):L^2\to L^2$ solve 
$$
(hD_t+Op_h({p_t}))U(t)=0,\qquad U(0)=\Id.
$$
Then, {$U(t)$} is microlocally unitary from $V_1$ to $V_2$ and quantizes $\kappa_t$. Moreover,
$$
U(t)=\frac{1}{(2\pi h)^{n}}\int_{\R^n} e^{\frac{i}{h}(\phi(t,x,\eta)-\langle y,\eta\rangle)}b(t,x,\eta;h)d\eta
$$
{where $b(t, \cdot)\in S^{\comp}(T^*\R^n)$   and the phase function $\phi(t,\cdot)\!\in C^\infty\!(T^*\R^n;\R)$  satisfies}
$$
\partial_t\phi+{p_t}(x,\partial_{x}\phi)=0,\qquad \phi(0,x,\eta)=\langle x,\eta\rangle,
$$
for all $t\in (-\e_0,\e_0)$.
Since $U(t)$ is microlocally unitary, it is enough to estimate the operator 
$${A(t)}:={X(0)}X(t)U(t).$$
First, note that since $X(t)\in \Psi_{\Gamma(t),\rho}^{-\infty}$, {and $U(t)$ quantizes $\kappa_t$, there exist}  $a(t)\in \widetilde{S^{-\infty}_{\Gamma_0,\rho}}$ {with $t\in {(-\e_0,\e_0)}$} such that {$X(t)=U(t)\Op_h(a(t))[U(t)]^* + O(h^\infty)_{L^2 \to L^2}$} and so
$$
A(t)=\Op_h(a(0)) U(t)\Op_h(a(t))+{O(h^\infty)_{L^2\to L^2}}.
$$

Fix $N>n-1$ and let $\chi=\chi(\lambda)\in \widetilde{S^{-N}_{\Gamma_0,\rho}}$ be such that $|\chi(\lambda)|\geq c\langle \lambda\rangle^{-N}$. 
Now, since $a(t)\in \widetilde{S^{-\infty}_{\Gamma_0,\rho}}$, by the elliptic parametrix construction there are $e\sub{L}(t),e\sub{R}(t)\in \widetilde{S^{-\infty}_{\Gamma_0,\rho}}$ such that
\begin{gather*}
\Op_h(e\sub{L}(t))\Op_h(\chi)=\Op_h(a(t))+O(h^\infty)_{L^2\to L^2},\quad \Op_h(\chi)\Op_h(e\sub{R}(t))=\Op_h(a(t))+O(h^\infty)_{L^2\to L^2},
\end{gather*}
{for all $t\in {(-\e_0,\e_0)}$}.
Note that we are implicitly using the fact that $a(t)$ is compactly supported in $(x,\xi)$ to handle the fact that $\chi$ is not compactly supported in $(x,\xi)$.  
Thus, 
$$
A(t)=\Op_h(e\sub{L}(0))\Op_h(\chi)U(t)\Op_h(\chi)\Op_h(e\sub{R}(t))+O(h^\infty)_{L^2\to L^2}.
$$
Since $\Op_h(e\sub{L}(t))$ and $\Op_h(e\sub{R}(t))$ are $L^2$ bounded uniformly in $t\in (-\e_0,\e_0)$, we estimate
$$
{\tilde{A}(t)}:=\Op_h(\chi)U(t)\Op_h(\chi)
$$
In fact, we estimate $B(t):=\tilde{A}(t)(\tilde{A}(t))^*$ by considering its kernel. 
\begin{align*}
{B(t;x,y)}&=\int U(t)(x,w)U(t)^*(w,y)\chi(h^{-\rho}x')\chi(h^{\rho}y')\chi(h^{-\rho}w')^2 dw\\
&\hspace{-0.3cm}=\!\!\frac{1}{(2\pi h)^{2n}}\!\!\int e^{\frac{i}{h}\Phi(t,x,w,y,\eta,\xi)}b(t,x,\eta)\bar{b}(t,y,\xi)\chi(h^{-\rho}x')\chi(h^{-\rho}y'){\chi(h^{-\rho}w')^2}dwd\eta d\xi, 
\end{align*}
with
$
\Phi(t,x,w,y,\eta,\xi)=\phi({t},x,\eta)-\phi({t},y,\xi)+\langle w,\xi-\eta\rangle.
$
First, performing stationary phase in $(w_n,\eta_n)$ yields
{
\begin{align*}
B(t;x,y)&=\frac{1}{(2\pi h)^{2n-1}}\int F(t,x,w', \xi_n) \overline{F(t,y,w', \xi_n)}  dw'd\xi_n,
\end{align*}
$$
F(t,x,w', \xi_n):=\int e^{\frac{i}{h} (\phi({t}, x, \eta', \xi_n)-\langle w', \eta'\rangle)} b_{{1}}(t,x,\eta', \xi_n)\chi(h^{-\rho}x')\chi(h^{-\rho}w')^2 d\eta'
$$
for some $b_1\in S^{\comp}(T^*\re^n)$.
}
Next, note that {since $\phi(0,x,\eta)=\langle x,\eta\rangle$,}
$$
\phi({t},x,\eta)-\langle x,\eta\rangle 
=t\tilde{\phi}(t,x,\eta)
$$
with {$\tilde \phi$ such that for every multi-index $\alpha$ there exists $C_\alpha>0$ with}
$
|\partial^\alpha_{t,x,\eta} \tilde{\phi}|\leq C_\alpha.
$

Next, we claim that {there exists $C>0$ such that  
\begin{equation}\label{E:claim on derivative}
\|(\partial^2_{\eta'} \tilde{\phi}(t, x, \eta))^{-1}\|\leq C \qquad \;\;\text{if}\quad (x,\eta)\in \Gamma(0),\; \,\partial_{\eta'}\phi(t, x, \eta)=0.
\end{equation}}
We postpone the proof of \eqref{E:claim on derivative} and proceed to finish the proof of the lemma.

To continue the proof, note that  modulo an $O(h^{N\e})$ error, we may assume that the integrand of $B(t;x,y)$ is supported, in $\{(x,y,w'): |x'|\leq h^{\rho-\e},\,|y'|\leq h^{\rho-\e},\,|w'|\leq h^{\rho-\e}\},$
{and $h^{\rho-\ep}\leq |t|$.} Therefore, the bound in \eqref{E:claim on derivative} continues to hold on the support of the integrand. By~\eqref{E:claim on derivative} and
\begin{equation}
\label{e:derVanish}
\partial^2_{\eta'}{\big(\phi(t,x,\eta)-t\tilde{\phi}(t,x,\eta) \big)}=0,
\end{equation} 
{there is a unique critical point $\eta'_c(t,x,w',\xi_n)$ for the map $\eta' \mapsto \phi(t,x,\eta', \xi_n)-\langle w',\eta'\rangle$,} in an $O(1)$ neighborhood of $\eta'_c$. In particular, {$\eta'_c$ is} the unique solution to 
$
\partial_{\eta'}\phi(t,x,\eta'_c,\xi_n)-w'=0.
$

Next, again using~\eqref{e:derVanish}, {by applying the method of stationary phase in $\eta'$ to $F$, with small parameter $h/t$}, we obtain 
\[
B(t; x,y)=\frac{1}{(2\pi h)^{n}t^{n-1}}\int e^{\frac{i}{h}\Phi_1(t,x,w',y,\xi_n)}{B_1(t;x,y,w', \eta_c', \xi)}dw' d\xi_n,
\]
\begin{gather*}
\Phi_1(t,x,w',y,\xi_n):=\Psi(t,x,w',\xi_n)-\Psi(t,y,w',\xi_n),\\
\Psi(t,x,w',\xi_n):=\phi(t,x,\eta'_c(t,x,w',\xi_n),\xi_n)-\langle w',\eta'_c(t,x,w',\xi_n)\rangle,\\
B_1(t;x,y,w', \eta', \xi):=b_{2}(t,x,\eta', \xi_n)\bar{b}(t,y,\xi',\xi_n)\chi(h^{-\rho}x')\chi(h^{-\rho}y'){\chi(h^{-\rho}w')^2}.
\end{gather*}
for some $b_2\in S^{\comp}(T^*\re^n)$.
Next, observe that 
$
\partial_{x_n}\partial_{\xi_n}\Psi(t,x,w',\xi_n)= 1+O(t), 
$
and therefore, {there exist $c>0$ and a function $g=g(x',y,w',\xi_n)$} so that
$
|\partial_{\xi_n}\Phi_1|\geq c |x_n-g_n|.
$
In particular, integration by parts in $\xi_n$ shows that for any $N>0$ there is $C_N>0$ such that 
$$
|B(t; x,y)|\leq C_N{h^{-n}}t^{1-n}h^{\rho(n-1)}\chi(h^{-\rho}y')\chi(h^{-\rho}x')\frac{ h^{2N}+h^N|x_n-g_n|}{(h^2+|x_n-g_n|^2)^N}.
$$

Applying Schur's lemma together with the fact that {there exists $C>0$ such that} {for all $t$}
$$
\sup_{x}\int |{B(t;x,y)}|dy+\sup_y\int |{B(t;x,y)})|dy\leq Ch^{(2\rho-1)(n-1)}t^{1-n},
$$
yields {that}
$
\|{B(t)}\|_{L^2\to L^2}\leq Ch^{(2\rho-1)(n-1)}t^{1-n},
$
for all $t\in (-\e_0,\e_0)$,
and hence 
$
\|X(0)X(t)\|_{L^2\to L^2}\leq Ch^{\frac{n-1}{2}(2\rho-1)}t^{\frac{1-n}{2}},
$
as claimed.\smallskip

\noindent{\emph{Proof of the bound in \eqref{E:claim on derivative}.}} 
{Let $\phi_t(x,\eta):=\phi(t,x, \eta)$ and $\varphi_t(x,y,\eta):=\phi_t(x,\eta)+\langle y, \eta\rangle$. Then, 
$C_{\varphi_t}=\{(x,y, \eta):\; \partial_\eta \phi_t(x,\eta)=y\}$ and so
$$\Lambda_{\varphi_t}=\{(x,\partial_x \phi_t(x,\eta), \partial_\eta \phi_t(x,\eta), -\eta) \}\subset T^*\R^n \times T^*\R^n.$$
In particular, since $\Lambda_{\varphi_t}$ is the twisted graph of $\kappa_t$, we have that $\kappa_t$ is characterized by 
$$ \kappa_t(\partial_\eta \phi_t(x,\eta), \eta)= (x,\partial_x \phi_t(x,\eta)).$$
Furthermore, since $\kappa_t(\Gamma(0))=\Gamma(t)$, we have
$$\Gamma(t)=\{(x, \xi): \; \kappa_t(y, \eta)=(x,\xi), \; y=\partial_\eta \phi_t(x,\eta), \;  \xi=\partial_x \phi_t(x,\eta),\;  (y, \eta) \in \Gamma(0)\}.$$
Then, using $\kappa_t^* \xi_n =\eta_n$, 
$$\Gamma(t)=\{(x, \xi): \; \xi'=\partial_{x'} \phi_t(x,\eta), \;  \partial_{\eta'} \phi_t(x,\eta)=0, \; \xi_n=\eta_n,\; \eta\in \R^n\}.$$

Next, let $\tilde p:=(\tilde x, \tilde \eta)\in \Gamma(0)$ be such that $\partial_{\eta'}\phi_t(\tilde x, \tilde \eta)=0$.
Without loss of generality, in what follows we assume that $\tilde x_n=0$.
Letting $\Gamma_0(t):=\Gamma(t)\big|_{\{ x_n=0\}}$
we have that 
$$\Gamma_0(t)=\{(x, \xi): \; \xi'=\partial_{x'} \phi_t(x,\eta), \;  \partial_{\eta'} \phi_t(x,\eta)=0, \;x_n=0,\; \xi_n=\eta_n,\; \eta\in \R^n\}.$$
In particular, letting $\tilde{\xi}:=(\partial_{x'}\phi_t(\tilde{p}),\tilde{\eta}_n)$, and $\tilde{p}_0:=(\tilde{x},\tilde{\xi})$, we have  $\tilde p_0\in \Gamma_0(t) \cap \Gamma_0(0),$
and 
\begin{align*}
    T_{\tilde p_0}\Gamma_0(t)
    =\{(\delta_x, \delta_\xi): 
    \; &\delta_{\xi'}=\partial_x\partial_{x'} \phi_t(\tilde p) \delta_x+\partial_\eta\partial_{x'} \phi_t(\tilde p) \delta_\eta , \\
    \;& \partial_x\partial_{\eta'} \phi_t(\tilde p)\delta_x+\partial_\eta\partial_{\eta'} \phi_t(\tilde p)\delta_\eta =0,
    \;\delta_{x_n}=0,\; \delta_{\xi_n}=\delta_{\eta_n},\; \delta_{\eta}\in \R^n\}.
\end{align*}
Next, we note that $\partial_{x_n} \in T_{\tilde p_0}\Gamma(t)$ and $H_{q_i(t)} \in T_{\tilde p_0}\Gamma(t)$ for all $i=1, \dots, n-1$. Therefore, since $\partial_{x_n}q_i(t)=0$, we also know that 
 $H_{q_i(t)}':=(\partial_{\xi'}q_i(t), 0, -\partial_{x'}q_i(t),0) \in T_{\tilde p_0}\Gamma_0(t)$ for all $i=1, \dots, n-1$.
  We claim that there exists $C>0$ such that for all  $v=(\delta_{x'},0, \delta_{\xi'}, 0)  \in \text{span}\{H_{q_i(t)}'\}_{i=1}^{n-1} \subset  T_{\tilde p_0}\Gamma(t)$ we have
 \begin{equation}\label{e: delta relation}
     \|\delta_{x'}\| \geq C t \| \delta_{\xi'}\|.
 \end{equation}
Suppose that the claim in \eqref{e: delta relation} holds. Then,  note that for each such $v$, since $\delta_{x_n}=0$ and $\delta_{\xi_n}=0$, we have that there is $\delta_{\eta'} \in \R^{n-1}$ such that
$$
 \delta_{\xi'}=\partial_{x'}^2 \phi_t(\tilde p) \delta_{x'}+\partial_{\eta'x'}^2 \phi_t(\tilde p) \delta_{\eta'}, \qquad 
    \; \partial_{x'\eta'}^2 \phi_t(\tilde p)\delta_{x'}+\partial_{\eta'}^2 \phi_t(\tilde p)\delta_{\eta'} =0.
$$
Using that $\partial_{x'\eta'}^2 \phi_t(\tilde p)=\Id +O(t)$ and $\partial_{x'}^2 \phi_t(\tilde p)=O(t)$, we conclude that 
$$
\partial_{\eta'}^2 \phi_t(\tilde p)[\partial_{\eta'x'}^2 \phi_t(\tilde p)]^{-1}\delta_{\xi'}= \left(\partial_{\eta'}^2 \phi_t(\tilde p)[\partial_{\eta'x'}^2 \phi_t(\tilde p)]^{-1}\partial_{x'}^2 \phi_t(\tilde p) -\partial_{x'\eta'}^2 \phi_t(\tilde p) \right)\delta_{x'},
$$
and so 
\begin{equation}\label{e: startfish}
\partial_{\eta'}^2 \phi_t(\tilde p)(\Id +O(t))\delta_{\xi'}={(-\Id +O(t))}\delta_{x'}.
\end{equation}

Let  $H_{q_i(t)}'=(\delta_{x'}^{(i)}, 0, \delta_{\xi'}^{(i)},0)$.
Since $\tilde p_0 \in \Gamma(t) \cap \Gamma(0)$, assumptions \eqref{e:uncertainAssume1} and \eqref{e:uncertainAssume2}  yield that the vectors $\{\delta_{x'}^{(i)}\}_{i=1}^{n-1}$ are linearly independent. Indeed, setting $e_i:=(\delta_{ij})_{j=1}^{n-1} \in \R^{n-1}$,
\begin{equation}\label{e: vectors are li}
\delta_{x'}^{(i)}=\partial_{\xi_i}q_i(t) e_i +O(t^2), \qquad |\partial_{\xi_i}q_i(t)|\geq Ct,
\end{equation}
for $t$ small enough.
Furthermore, \eqref{e: startfish} then yields that $\{\delta_{\xi'}^{(i)}\}_{i=1}^{n-1}$ are linearly independent. Then,  combining \eqref{e: startfish}  with \eqref{e: delta relation}  yields \eqref{E:claim on derivative} as claimed.

To finish it only remains to prove \eqref{e: delta relation}. Let $v=(\delta_{x'},0, \delta_{\xi'}, 0)  \in \text{span}\{H_{q_i(t)}'\}_{i=1}^{n-1}$. Then, there is $a \in \R^{n-1}$ such that $\delta_{x'}=\sum_{i=1}^{n-1} a_i \delta_{x'}^{(i)}$ and $\delta_{\xi'}=\sum_{i=1}^{n-1} a_i \delta_{\xi'}^{(i)}$. Next, note that by \eqref{e: vectors are li} we have $\|\delta_{x'}\| \geq \|a\| (Ct +O(t^2))$. Since $\|\delta_{\xi'}\| \leq C_0 \|a\|$ for some $C_0>0$, the claim in \eqref{e: delta relation} follows.
%
%
%
}
\end{proof}

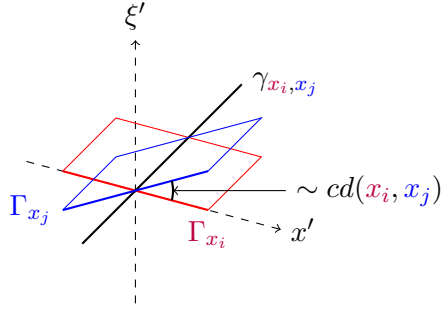
\begin{figure}
\label{f:twisting}
    \begin{tikzpicture}
        \begin{scope}[scale=1]
        \draw[->,dashed] (-15:-1.5)--(-15:2)node[right]{$x'$};
        \draw[->,dashed] (0,-1.5)--(0,2)node[above]{$\xi'$}; 
        \draw[red,thick] (-15:-1)--(-15:1) node[below]{$\red{\Gamma_{x_i}}$};
        \draw[red] (-15:1)--++(45:1)--++(-15:-2)--++(45:-1);
        \draw[thick](45:-1)--(45:2)node [right]{$\gamma_{\red{x_i},\blue{x_j}}$};
        \draw[blue, thick] (15:-1)node[left]{$\blue{\Gamma_{x_j}}$}--(15:1) ;
        \draw[blue](15:1)--++(45:1)--++(15:-2)--++(45:-1);
        \begin{scope}[rotate=-15]
        \draw[thick] (.5,0) arc (0:30:.5);
        \draw[->](15:2)node[right]{$\sim cd(\red{x_i},\blue{x_j})$}--(15:.5);
        \end{scope}
        \end{scope}
    \end{tikzpicture}
    \caption{A pictorial representation of the co-isotropics involved in Corollary~\ref{c:corUncertain} where $\gamma_{x_i,x_j}$ is the geodesic from $x_i$ to $x_j$. Localization to both $\Gamma_{x_i}$ and $\Gamma_{x_j}$ implies localization in the non-symplectically orthogonal directions, $x'$ and $\xi'$. The uncertainty principle then rules this behavior out.} 
\end{figure}

For each $x \in M$ let $\Gamma_x$ be as in~\eqref{e:coiso}. (See Figure~\ref{f:twisting} for a schematic representation of these two co-isotropic submanifolds.) Then we have the following result.
\begin{corollary}
\label{c:corUncertain}
Let $0<\rho<1$, {$0<\e<\rho$}, and $\gamma(t):(-\e_0,\e_0)\to M$ be a unit speed geodesic.  {Then, for $X(t)\in \Psi^{-\infty}_{\Gamma_{\gamma(t)},\rho}$ and $h$  such that $h^{\rho-\e}\leq {|t|<\e_0}$},
$$
\|X(0)X(t)\|_{L^2\to L^2}\leq Ch^{\frac{n-1}{2}(2\rho-1)}t^{\frac{1-n}{2}}.
$$
\end{corollary}
\begin{proof}
To do this, we study the geometry of the flow-out co-isotropics $\Gamma_{\gamma(t)}$. {Namely, we prove that $\Gamma_{\gamma(t)}$ is defined by some functions $\{q_i(t)\}_{i=1}^n$ with $q_i(0)=x_i$ and satisfying \eqref{e:uncertainAssume1} and \eqref{e:uncertainAssume2}. We then apply Proposition
\ref{p:uncertainty} to $\Gamma(t)=\kappa^{-1}(\Gamma_{\gamma(t)})$, for a suitable symplectomorphism $\kappa$.} 

Fix coordinates {$(x',x_n) \in \R^{n-1}\times \R$} on $M$ so that $\gamma(t)=( 0,t)$, and for each $t \in (-\e_0,\e_0)$ let  $\mc{H}_t$  be the submanifold  transverse to the Hamiltonian vector field $H_p$ defined by
$$
\mc{H}_{t}:=\{ (x',t, \xi', \xi_n):\; 2\xi_n> |\xi|_g, \; {|x'| \leq \delta_0}\},
$$
{where $\delta_0>0$ is chosen such that }
$
{\Gamma_{\gamma(t)}}\cap \mc{H}_t= \{(0,t,\xi', \xi_n): \, 2\xi_n>|\xi|_g,\; \big||\xi|_g-1\big|<\delta\}.
$

In particular, {as a subset of $\big\{\big||\xi|_g-1\big|<\delta\big\}$, } $\Gamma_{\gamma(t)}\cap \mc{H}_t$ is defined by the coordinate functions $\{x_i\}_{i=1}^{n-1}$. {For each $t \in (-\e_0,\e_0)$} let $\tilde{q}_i(t):\mc{H}_t\to \re$ be given by $\tilde{q}_i(t)=x_i$ for  $i=1,\dots, n-1$. Then, define $\{q_i(t)\}_{i=1}^{n-1}$ {on $T^*M$} by
$$
H_pq_i(t)=0,\qquad q_i(t)|\sub{\mc{H}_t}=\tilde{q}_i(t).
$$
Note that {for all $t$}
$
H_p(H_{q_i(t)}q_j(t))=0
$
and 
$$
\{q_i(t),q_j(t)\}\mid\sub{\mc{H}_t}=\partial_{\xi_n}q_i(t)\partial_{x_n}q_j(t)-\partial_{\xi_n}q_j(t)\partial_{x_n}q_i(t)+\tilde{H}_{q_i(t)}q_j(t),$$
where $\tilde{H}$ is the Hamiltonian vector field in $T^*\{x_n=t\}$. In particular, since $\partial_{\xi_n}\tilde{q}_i(t)=0$ and $\tilde{H}_{q_i(t)}$ is tangent to $\mc{H}_t$, we have
$
\{q_i(t),q_j(t)\}\mid\sub{\mc{H}_t}=0.
$
Hence, $\{q_i(t),q_j(t)\}\equiv 0$, $\{q_i(t),p\}=0$, {$q_i(0)=x_i$}, and $\{q_i(t)\}_{i=1}^{n-1}$ define $\Gamma_{\gamma(t)}.$
Next, observe that {there exists $s \in \R$ such that for each $i=1, \dots, n-1$,}
$
q_i(0)(x, \xi)=x_i(\varphi_{s}(x,\xi))$ with $\varphi_{s}(x,\xi)\in\mc{H}_0.
$
Since $\partial_{\xi_n}p\neq 0$ on $\mc{H}_0$, {for $E$ near $0$ there exist $a_{\sub{\!E}}$ and $e_{\sub{\!E}}$ such that}
$$
p(x, \xi)-E=e_{\sub{\!E}}(x,\xi)(\xi_n-a_{\sub{\!E}}(x,\xi'))
$$
with $e_{\sub{\!E}}>c$ {for some constant $c>0$.}
In particular, $\varphi_s=\exp(sH_p)$ {is a reparametrization of $\tilde \varphi_s:=\exp(s(H_{\xi_n-a_{\sub{\!E}}}))$ on $\{p=E\}$} and we have that for $(x,\xi)\in \{p=E\}$, {and all $i=1, \dots, n-1$},
{$$
q_i(0)(x,\xi)=x_i(\tilde \varphi_{-x_n}(x,\xi))={x_i}+x_n\partial_{\xi_i}a_{\sub{\!E}}(x,\xi')+O(x_n^2)_{C^\infty}.
$$}
In particular, on $\mc{H}_t\cap \{p=E\}$, since $H_{q_j(t)}$ is tangent to $\{p=E\}$, we have
\begin{align*}
\{q_j(t),q_i(0)\}|\sub{\mc{H}_t\cap \{p=E\}}&=\partial_{\xi_n}q_j(t)\partial_{x_n}q_i(0)-\partial_{x_n}q_j(t)\partial_{\xi_n}q_i(0) +\tilde{H}_{q_j(t)}q_i(0)
=O(t^2)+\partial_{\xi_j}{(t\partial_{\xi_i}a_{\sub{\!E}})}.
\end{align*}
Now, since $\partial_\xi^2p|\sub{T\{p=E\}}>0$, and for all $i, j=1, \dots, n$  
\begin{align*}
&\partial_{\xi_i\xi_j}p=\partial_{\xi_j}\partial_{\xi_i}e_{\sub{\!E}}(\xi_n-a_{\sub{\!E}})+\partial_{\xi_i}e_{\sub{\!E}}(\delta_{nj}-\partial_{\xi_j}a_{\sub{\!E}})+\partial_{\xi_j}e_{\sub{\!E}}(\delta_{ni}-\partial_{\xi_i}a_{\sub{\!E}})-e_{\sub{\!E}}\partial_{\xi_j}\partial_{\xi_i}a_{\sub{\!E}},
\end{align*}
then, as quadratic forms,
$
\partial_{\xi}^2p|\sub{T\{p=E\}}=-e_{\sub{\!E}}{\partial_\xi^2}a_{\sub{\!E}}|\sub{T\{p=E\}}.
$
Hence, $\partial_{\xi'}^2a_{\sub{\!E}}<0$, and  {there is $c>0$ with}
$$
c\delta_{ij}t+O(t^2)  \leq \Big|\{q_i(t),q_j(0)\}|_{\mc{H}_t\cap \{p=E\}}\Big|\leq  C\delta_{ij}t+O(t^2).
$$
Then,  
$
c\delta_{ij}t+O(t^2)  \leq \Big|\{q_i(t),q_j(0)\}|_{ \{p=E\}} \Big|\leq  C\delta_{ij}t+O(t^2) 
$
by invariance under $H_p$.
Since $E$ small is arbitrary, this holds on $\Gamma_{\gamma(0)}\cap \Gamma_{\gamma(t)}$.

Now, by Darboux's theorem, there is a symplectomorphism $\kappa$ such that for all $i=1, \dots, n-1$
$
\kappa^*q_i(0)=x_i$ and $\kappa^*p=\xi_n.
$
In particular, 
$
\kappa^{-1}(\Gamma_{\gamma(0)})\subset {\Gamma(0)=}\{(0,x_n,\xi', \xi_n): \;x_n\in \re,\,\xi\in \re^{n-1} \times \re\}
$
and, abusing notation slightly by relabeling $q_i(t)=\kappa^*q_i(t)$, we have that~\eqref{e:uncertainAssume1} and~\eqref{e:uncertainAssume2} hold.{ In particular, Proposition
~\ref{p:uncertainty} applies to $\Gamma(t)=\kappa^{-1}(\Gamma_{\gamma(t)})$. }

Now, let $U$ be a microlocally unitary quantization of $\kappa$, and $X(t)\in \Psi^{-\infty}_{\Gamma_{\gamma(t)},\rho}$. Then, 
$U^{-1}X(t)U\in \Psi^{-\infty}_{\Gamma(t),\rho}$ and hence the corollary is proved.
\end{proof}

\subsection{Almost orthogonality for coisotropic cutoffs}
\label{s:almostOrthog}

{In this section, we finally prove an estimate which shows that co-isotropic cutoffs associated with $\Gamma_{x_i}$ for many $x_i$ are almost orthogonal. This, together with the fact that these cutoffs respect pointwise values near $x_i$, is what allows us to control the number of points at which a quasimode may be large.}
\begin{proposition}\label{P:orthogonality}
Let $\{B(x_i,R)\}_{i=1}^{N(h)}$ be a $(\mathfrak{D},R)$-good cover for $M$, and $X_i\in \Psi^{-\infty}_{\Gamma_{x_i},\rho}$ $i=1,\dots N(h),$ with uniform symbol estimates. Then, there are $C>0$ and $h_0>0$ such that for all $0<h<h_0$,  $\mc{J}\subset \{1,\dots ,N(h)\}$, and  $u\in L^2(M)$, we have
\begin{equation}
\label{e:orthogonality}
{\sum_{j\in \mc{J}}\|X_j u\|^2_{L^2}\leq C \Big(1+(h^{2\rho-1}R^{-1})^{\frac{n-1}{2}}|\mc{J}|^{\frac{3n+1}{2n}}\big(1+(h^{2\rho-1}R^{-1})^{\frac{n-1}{4}}\big) \Big)\|u\|_{L^2}^2}
\end{equation}

\end{proposition}
\begin{proof}
To prove this bound we will decompose the sum in \eqref{e:orthogonality} as
\begin{equation}\label{e:Xgoal}
\sum_{i\in \mc{J}} \|X_i u\|_{L^2}^2=\Big\|\sum_{i\in \mc{J}} X_iu \Big\|_{L^2}^2- \Big\langle \sum_{\substack{i,j\in \mc{J}\\i\neq j}}{X_j^*}X_iu,u \Big\rangle. 
\end{equation}
First, we note that by Corollary~\ref{c:corUncertain}, {there exists $C>0$ such that for $i \neq j$}
$$
\|{X_j^*}X_i\|\leq Ch^{(n-1)(\rho-\frac{1}{2}})d({x_i},x_j)^{\frac{1-n}{2}}.
$$
Therefore, by the the Cotlar-Stein lemma, 
\begin{align*}
\Big\|\sum_{j\in \mc{J}} X_j\Big\|
&{\leq \sup_{j\in \mc{J}} \Big(\| X_j\| + \sum_{{i \in \mc{J} \backslash\{j\}}}\|X_j^*X_i\|^{\tfrac{1}{2}}\Big)}\leq 2+C h^{\frac{n-1}{2}(\rho-\frac{1}{2})}\, \sup_{j\in \mc{J}}\sum_{{i \in \mc{J} \backslash\{j\}}}d(x_i,x_j)^{\frac{1-n}{4}}.
\end{align*}
To estimate the sum, observe that {there exists $C>0$ such that for any $j \in \mc{J}$  and any positive integer $k$}
$
\tfrac{1}{C}2^{kn}\leq \#\{i:\;  2^{k}R\leq d(x_i,x_j)\leq 2^{k+1}R\}\leq C2^{(k+1)n}.
$
{In particular,
there is $C>0$ such that for any $j \in \mc{J}$
\begin{equation}\label{E:adding wo j}
\sum_{i \in \mc{J} \backslash\{j\}} d(x_i,x_j)^{\frac{1-n}{4}}\leq C\sum_{k=0}^{\frac{1}{n}\log_2|\mc{J}|}2^{kn}(2^kR)^{\frac{1-n}{4}}\leq C|\mc{J}|^{\frac{3n+1}{4n}}R^{\frac{1-n}{4}}.
\end{equation}}
Therefore, we shall bound the first term in \eqref{e:Xgoal} using
\begin{equation}\label{e:X_j}
\Big\|\sum_{j\in \mc{J}}X_j \Big\|\leq C+Ch^{\frac{n-1}{2}(\rho-\frac{1}{2})}R^{\frac{1-n}{4}}|\mc{J}|^{\frac{3n+1}{4n}}.
\end{equation}

We next proceed to control the second term in \eqref{e:Xgoal}. Let $\tilde{X}_j\in \Psi^{-\infty}_{\Gamma_{x_j},\rho}$ such that $\tilde{X}_jX_j=X_j+O(h^\infty)_{L^2\to L^2}$. 
By the Cotlar-Stein Lemma, 
\begin{equation}\label{e:sec term}
\Big\|\sum_{\substack{i,j\in \mc{J}\\i\neq j}} X_j^*X_i \Big\|
    \leq \sup_{\substack{k,\ell\in \mc{J}\\k\neq \ell}} \sum_{\substack{i,j\in \mc{J}\\i\neq j}} \|X_k^*\tilde{X}_\ell X_\ell X_j^*\tilde{X}_j^*X_i\|^{\frac{1}{2}}+O(h^\infty|\mc{J}|^2).
\end{equation}
By Corollary \ref{c:corUncertain} there exists $C>0$ such that for $k \neq \ell$, $i \neq j$, 
$$
\|{X_k^* \tilde X_\ell X_\ell X_j^*\tilde X_j^*X_i}\|\leq C  h^{(n-1)(2\rho-1)}\min\{1,   h^{\frac{(n-1)}{2}(2\rho-1)}d(x_j,x_\ell)^{-\frac{n-1}{2}}\}(d(x_k,x_\ell)d(x_j,x_i))^{\frac{1-n}{2}}.
$$

Using that $\sup_{\substack{k,\ell\in \mc{J}\\k \neq \ell}}d(x_k,x_\ell)^{\frac{1-n}{4}}\leq R^{\frac{1-n}{4}}$, adding first in $i \in \mc{J}\backslash \{j\}$ in \eqref{e:sec term}, and combining with the bound in \eqref{E:adding wo j},  yields
\begin{align}\label{e:X_jX_k}
    \Big\|\sum_{\substack{i,j\in \mc{J}\\i\neq j}} X_j^*X_i\Big\|
    &\leq  C h^{\frac{n-1}{2}(2\rho-1)}  (1+h^{\frac{n-1}{4}(2\rho-1)}|\mc{J}|^{\frac{3n+1}{4n}}R^{\frac{1-n}{4}}) |\mc{J}|^{\frac{3n+1}{4n}}R^{\frac{1-n}{2}}. 
\end{align}
In particular,  combining \eqref{e:X_j} and \eqref{e:X_jX_k} into \eqref{e:Xgoal} we obtain
\begin{align*}
\sum_{i\in \mc{J}} \|X_i u\|^2&\leq C \Big(1+h^{(n-1)(\rho-\frac{1}{2})}R^{\frac{1-n}{2}}|\mc{J}|^{\frac{3n+1}{2n}}+h^{\frac{3(n-1)}{4}(2\rho-1)}R^{\frac{3(1-n)}{4}}{|\mc{J}|^{\frac{3n+1}{2n}}}\Big)\|u\|_{L^2}^2 \notag\\
&\leq C (1+h^{(n-1)(\rho-\frac{1}{2})}R^{\frac{1-n}{2}}(1+(h^{2\rho-1}R^{-1})^{\frac{n-1}{4}}){|\mc{J}|^{\frac{3n+1}{2n}}})\|u\|_{L^2}^2.
\end{align*}

\vspace{-.5cm}
\end{proof}

\bibliography{biblio}{}
\bibliographystyle{alpha}
\end{document}